\documentclass[12pt]{amsart}

\usepackage{lmodern}
\usepackage[T1]{fontenc}
\usepackage{secnum}
\usepackage{mathtools}
\usepackage[
scr=stixtwofancy,
cal=euler,
bb=stixtwo,
]{mathalpha}
\usepackage{amssymb}
\usepackage[Symbol]{upgreek}
\usepackage{txfonts}

\usepackage{xcolor}
\usepackage[normalem]{ulem}
\usepackage{cancel}
\usepackage{graphicx}
\usepackage{interval}
\usepackage{xfrac}
\usepackage{fixdif}
\usepackage{leftindex}
\usepackage{accents}
\usepackage{aligned-overset}
\usepackage{mathcommand}

\usepackage{physics2}
\usephysicsmodule{ab,ab.braket,xmat}

\usepackage[shortlabels,inline]{enumitem}
\usepackage{keytheorems}
\usepackage[
hypcap=true,
font=small,
]{caption}
\usepackage{tikz}
\usepackage{tikz-cd}

\usepackage[pagebackref]{hyperref}
\usepackage{zref-clever}

\numberwithin{equation}{section}
\allowdisplaybreaks
\mathtoolsset{showonlyrefs}
\setlist[enumerate,1]{%
	label=(\alph*),
	ref=(\alph*),
}
\setlist[enumerate,2]{%
	label=(\alph{enumi}.\arabic*),
	ref=(\alph{enumi}.\arabic*),
}
\setsecnum{1.a}

\hypersetup{%
	linktoc=all,
	bookmarksnumbered=true, %
	colorlinks=true, %
	linkcolor=magenta, %
	citecolor=cyan, %
	pdfstartview=FitBH, %
}

\renewcommand*{\backrefalt}[4]{%
	\ifcase #1 %
	No citations.%
	\or
	$\uparrow$ #2%
	\else
	$\uparrow$ #2%
	\fi
}

\newkeytheorem{mainthm}[
name=Theorem,
counter-format=\Alph{mainthm},
]

\newkeytheorem{%
	theorem,
	proposition,
	corollary,
	lemma,
	question,
}[sharenumber=equation]

\newkeytheorem{%
	definition,
	remark,
	example,
	construction,
	notation,
	warning,
}[sharenumber=equation,style=remark]

\zcsetup{
	cap,
	abbrev,
	nameinlink=false,
}

\NewDocumentCommand \ListInThm { m O{\textup} O{(\alph{enumi})} O{.} }{%
	\AtBeginEnvironment{#1}{%
		\zcsetup{ countertype={enumi=#1} }
		\setlist[enumerate,1]{%
			label={#2{#3}},
			ref={\csname the#1\endcsname#4{#3}}
		}%
	}%
}

\ListInThm{theorem}
\ListInThm{proposition}
\ListInThm{lemma}
\ListInThm{corollary}
\ListInThm{mainthm}

\makeatletter
\zref@newprop{listitemref}{\ifnum\value{enumi}>0 (\alph{enumi})\fi}
\zref@addprop{main}{listitemref}
\makeatother
\NewDocumentCommand{\itemonlycref}{m}{\zcref[noname,ref=listitemref]{#1}}
\providecommand{\cref}[1]{\zcref{#1}}
\providecommand{\labelcref}[1]{\zcref[noname]{#1}}

\usetikzlibrary{
	cd,
	shapes,
	arrows,
	calc,
	decorations.pathmorphing,
}

\tikzcdset{
	scale cd/.style = {
		every label/.append style = {scale = #1},
		cells = {nodes = {scale = #1}}
	}
}

\let\le\leqslant
\let\ge\geqslant

\declaremathcommand{\iu}{\mathtt{i}} %
\declaremathcommand{\e}{\mathtt{e}} %
\declaremathcommand{\placeholder}{\:\cdot\:}
\NewDocumentCommand{\vect}{m}{\mathtt{#1}}
\newcommand{\isomto}{\overset{\simeq}{\longrightarrow}}

\DeclareMathOperator{\End}{End}
\DeclareMathOperator{\Hom}{Hom}
\DeclareMathOperator{\Id}{Id}
\DeclareMathOperator{\Image}{Im}
\DeclareMathOperator{\Res}{Res}
\DeclareMathOperator{\SL}{SL}
\DeclareMathOperator{\tr}{tr}
\DeclareMathOperator{\wt}{wt}

\NewDocumentCommand{\odv}{m}{\frac{\d}{\d{#1}}}
\NewDocumentCommand{\pdv}{m}{\frac{\partial}{\partial{#1}}}

\providecommand\given{}
\DeclareDocumentMathCommand{\SetSymbol}{O{}}{\nonscript\:#1\vert\allowbreak\nonscript\:\mathopen{}}

\DeclarePairedDelimiterX{\SetcAux}[1]{\lbrace}{\rbrace}{%
	\renewcommand\given{\SetSymbol[\delimsize]}#1%
}
\DeclarePairedDelimiterX{\SetaAux}[1]{\langle}{\rangle}{%
	\renewcommand\given{\SetSymbol[\delimsize]}#1%
}
\DeclarePairedDelimiterX{\SetbAux}[1]{\lbrack}{\rbrack}{%
	\renewcommand\given{\SetSymbol[\delimsize]}#1%
}
\DeclarePairedDelimiterX{\SetpAux}[1]{\lparen}{\rparen}{%
	\renewcommand\given{\SetSymbol[\delimsize]}#1%
}
\DeclarePairedDelimiterX{\SetbbAux}[1]{.}{.}{%
	\renewcommand\given{\SetSymbol[\delimsize]}%
	\mathopen{}\delimsize\lbrack\!\delimsize\lbrack
	#1%
	\delimsize\rbrack\!\delimsize\rbrack\mathclose{}%
}
\DeclarePairedDelimiterX{\SetccAux}[1]{.}{.}{%
	\renewcommand\given{\SetSymbol[\delimsize]}%
	\mathopen{}\delimsize\lbrace\!\delimsize\lbrace
	#1%
	\delimsize\rbrace\!\delimsize\rbrace\mathclose{}%
}
\DeclarePairedDelimiterX{\SetppAux}[1]{.}{.}{%
	\renewcommand\given{\SetSymbol[\delimsize]}%
	\mathopen{}\delimsize\lparen\!\delimsize\lparen
	#1%
	\delimsize\rparen\!\delimsize\rparen\mathclose{}%
}

\ExplSyntaxOn
\cs_new_protected:Npn \gaux_set_dispatch:nnnn #1#2#3#4
{
	\group_begin:
	\tl_if_novalue:nTF {#1}
	{ \cs_set_eq:NN \__gaux_set_base \SetcAux }
	{
		\str_case:nnF {#1}
		{
			{a}  { \cs_set_eq:NN \__gaux_set_base \SetaAux  }
			{b}  { \cs_set_eq:NN \__gaux_set_base \SetbAux  }
			{c}  { \cs_set_eq:NN \__gaux_set_base \SetcAux  }
			{p}  { \cs_set_eq:NN \__gaux_set_base \SetpAux  }
			{bb} { \cs_set_eq:NN \__gaux_set_base \SetbbAux }
			{cc} { \cs_set_eq:NN \__gaux_set_base \SetccAux }
			{pp} { \cs_set_eq:NN \__gaux_set_base \SetppAux }
		}
		{
			\PackageError{latex-writing-template}
			{Unknown~\string\Set~delimiter~selector~`#1'}
			{Use~a,b,c,p,bb,cc,or~pp.}
		}
	}
	\IfBooleanTF{#2}
	{ \__gaux_set_base * {#4} }
	{
		\IfNoValueTF{#3}
		{ \__gaux_set_base {#4} }
		{ \__gaux_set_base [#3] {#4} }
	}
	\group_end:
}
\DeclareDocumentMathCommand{\GauxSetDispatch}{mmmm}
{ \gaux_set_dispatch:nnnn {#1}{#2}{#3}{#4} }
\ExplSyntaxOff

\DeclareDocumentMathCommand{\Set}{e{_} s o m}{\GauxSetDispatch{#1}{#2}{#3}{#4}}
\DeclareDocumentMathCommand{\dbrack}{m}{\Set_{bb}{#1}}
\DeclareDocumentMathCommand{\dbrace}{m}{\Set_{cc}{#1}}
\DeclareDocumentMathCommand{\dparen}{m}{\Set_{pp}{#1}}
\NewDocumentCommand{\Fam}{m m}{\left\lgroup #1\right\rgroup_{\!#2}}
\DeclareDocumentMathCommand{\twobytwo}{mmmm}{%
	\begin{pmatrix}
		#1 & #2 \\
		#3 & #4
	\end{pmatrix}%
}

\NewDocumentCommand{\fusion}{O{M^1} O{M^2} O{M^3}}{{\textstyle\binom{#3}{#1\;#2}}}
\NewDocumentCommand{\Nfusion}{O{M^1} O{M^2} O{M^3}}{\vect{N}\fusion[#1][#2][#3]}
\NewDocumentCommand{\Ifusion}{O{M^1} O{M^2} O{M^3}}{\mathbf{I}\fusion[#1][#2][#3]}

\renewcommand{\H}{\mathbb{H}}

\providecommand{\C}{\mathbb{C}}
\providecommand{\D}{\mathbb{D}}
\providecommand{\N}{\mathbb{N}}
\providecommand{\Z}{\mathbb{Z}}
\providecommand{\Q}{\mathbb{Q}}
\renewcommand{\P}{\mathbf{P}}
\providecommand{\PP}{\mathbb{P}}

\renewcommand{\O}{\mathcal{O}}

\ProvideDocumentCommand{\fun}{m e{^_} o}{%
	\operatorname{#1}%
	\IfValueT{#2}{^{#2}}%
	\IfValueT{#3}{_{#3}}%
	\IfValueT{#4}{\ab(#4)}%
}

\renewcommand{\sl}{\mathfrak{sl}}

\declaremathcommand{\dblbackslash}{\backslash\!\!\backslash}

\newcommand{\vac}{\mathbf{1}}
\newcommand{\cfv}{\upomega}
\declaremathcommand{\ccharge}{\mathbf{c}}
\NewDocumentCommand{\vo}{m m}{#1_{(#2)}}
\NewDocumentCommand{\lo}{m m}{#1_{[#2]}}

\newcommand{\Fu}{\mathfrak{F}} %
\newcommand{\Irr}{\mathfrak{Irr}} %

\DeclareDocumentMathCommand{\mpt}{m}{\mathsf{#1}}
\DeclareDocumentMathCommand{\mPt}{m}{\mathcal{#1}}
\DeclareDocumentMathCommand{\sch}{m}{\mathscr{#1}}
\DeclareDocumentMathCommand{\cova}{m}{\ab*[{#1}]}
\DeclareDocumentMathCommand{\Hdual}{m}{{#1}^{\vee}}

\NewDocumentCommand{\concept}{m}{\emph{#1}}
\NewDocumentCommand{\proofstep}{m m}{\leavevmode\paragraph*{\textbf{Step~#1. #2}}}

\providecommand{\orgdiv}[1]{#1}
\providecommand{\orgname}[1]{#1}
\providecommand{\orgaddress}[1]{#1}
\providecommand{\street}[1]{#1}
\providecommand{\postcode}[1]{#1}
\providecommand{\city}[1]{#1}
\providecommand{\state}[1]{#1}
\providecommand{\country}[1]{#1}

\title{Trace-function frames and modularity of genus-one conformal blocks}

\author{Xu Gao}
\address{
	\orgdiv{Department of Mathematics}, 
	\orgname{Shanghai Maritime University}, 
	\orgaddress{
		\street{1550 Haigang Avenue}, 
		\city{Pudong New Area}, 
		\state{Shanghai}, 
		\postcode{201306}, 
		\country{China}
	}
}
\email{gaoxu@shmtu.edu.cn}

\author{Jianqi Liu}
\address{
	\orgdiv{Department of Mathematics}, 
	\orgname{New Uzbekistan University}, 
	\orgaddress{
		\street{Movarounnahr street 1}, 
		\city{Tashkent}, 
		\state{Uzbekistan}, 
		\postcode{100001}, 
		\country{Uzbekistan}
	}
}
\email{j.liu@newuu.uz}

\begin{document}
	
	\begin{abstract}
		The recursive relations and modular transformation laws of trace functions have a common geometric source: conformal blocks on pointed elliptic curves. For a simple, rational, $C_2$-cofinite, CFT-type vertex operator algebra and an irreducible module $W$ of it, we identify the bundle $\mathscr{C}_{\mathbb{H}}(W)$ of one-point conformal blocks over the upper half-plane $\mathbb{H}$ with the trivial bundle of intertwining operators. In particular, bases of intertwining operators give rise to global frames of $\mathscr{C}_{\mathbb{H}}(W)$ via their trace functions. Factorization of conformal-block bundles provides a rank formula used to establish this identification and also yields a geometric proof of basic properties of the fusion ring. Then the modular group $\SL(2,\Z)$ acts projectively on such a frame, with matrices independent of $\tau \in \H$. This modular-invariance result recovers that of Zhu for vertex-operator traces and gives a uniform geometric interpretation of the affine and Virasoro intertwining-operator traces studied by Etingof--Kirillov, Krauel--Marks, and Krauel--Shafiq--Wood. The trace frame is horizontal for a natural connection whose projectivization agrees with the pullback of the Damiolini--Gibney--Tarasca connection; after descent, its monodromy gives the modular action. With respect to this natural connection, we characterize global horizontal sections in terms of Miyamoto's four conditions. As an application of our results, we realize the cusp form $\Delta$ as a trace function for the affine VOA $L_{\widehat{\mathfrak{sl}}_2}(16,0)$.
		
	\end{abstract}
	
	\maketitle
	
	\tableofcontents

	\section{Introduction}
	
	Trace functions provide one of the fundamental links between the representation theory of vertex operator algebras (VOAs) and the geometry of elliptic curves. They encode data from representation theory while exhibiting the transformation properties characteristic of modular forms. 
	Zhu's modular invariance theorem made this connection precise for trace functions associated with vertex operators \cite[Theorem~5.3.2]{Zhu96}. 
	A natural question is how far Zhu's theorem extends to intertwining-operator traces and how these traces relate to the geometric framework of VOA conformal blocks developed in \cite{FBZ04,DGT24}. In particular, do they exhaust the genus-one one-point conformal blocks?
	
	In this paper, we answer this question for simple, rational, $C_2$-cofinite, CFT-type VOAs. Let $V$ be such a VOA and let $W$ be a finitely generated admissible $V$-module of conformal weight $h_W$ (see \cref{def:Adm} for the meaning of these terms). 
	We show that the trace construction identifies the bundle $\mathscr{C}_{\H}(W)$ of one-point conformal blocks over the upper half-plane $\H$ with the trivial bundle of intertwining operators.
	In particular, bases of the spaces of intertwining operators give rise to explicit global horizontal frames of the conformal block bundle via their trace functions. This identification also explains the common geometric source of the recursive relations and modular transformation laws of trace functions.
	
	The theory of conformal blocks for VOA modules \cite{Zhu94,FBZ04,NT05,DGT21} provides the geometric framework for our identification. These conformal blocks generalize WZNW conformal blocks arising from affine Lie algebras \cite{TK87,TUY89,Uen08}. We refer to \cite{DGT24,DGK25,DW25} for the recent developments, especially the factorization theorem and the vector bundle property that are necessary for our results.
	
	\subsection{The main theorem}
	
	Let $V$ be a simple, rational, $C_2$-cofinite\footnote{The $C_2$-cofiniteness hypothesis can be weakened; see \cref{rmk:C20-trace-frame}.} VOA of CFT type with central charge $\ccharge$.
	Let $W$ be a finitely generated admissible $V$-module of conformal weight $h_W$.
	In \cref{subsec:conformal-block-bundle-and-connection}, we introduce the bundle $\mathscr{C}_\H(W)$ by pulling back the analytic conformal-block bundle along the quotient map in the orbifold description of the moduli stack of elliptic curves. 
	
	In what follows, we fix a complete set
	$\Irr=\Set*{V=M^0,M^1,\cdots,M^N}$ of irreducible $V$-modules up to isomorphism.
	For any intertwining operator $I\in\Ifusion[W][M^k][M^k]$ and any $u\in W$, put
	\[
	o_I(u):=\Res_w w^{-1}I(w^{L(0)}u,w).
	\]
	With $q$ a formal variable, write
	\[
	\mathcal{Z}_I(u,q)
	:=\tr_{M^k}\ab(o_I(u)\,q^{L(0)-\frac{\ccharge}{24}}),
	\qquad u\in W,
	\]
	for the formal trace series. See \cref{subsec:convergence-formal-trace} for details.
	
	\begin{mainthm}[Trace frame and modular monodromy]\label{main:intro:revised}
		Under the aforementioned hypotheses, the bundle $\mathscr{C}_\H(W)$ carries a natural flat connection $\nabla$, and the following statements hold.
		\begin{enumerate}
			\item \textup{\textbf{Trace frame.}}
			\label{main:intro:revised:trace-frame}
			For $\tau\in\H$, put $q_\tau=\e^{2\pi\iu\tau}$ and use the convention $q_\tau^\beta:=\e^{2\pi\iu\tau\beta}$ for $\beta\in\C$. The formal trace series converge and define a conformal block
			\[
			\varphi_I(\tau)
			:=\tr_{M^k}\ab(o_I(\placeholder)\,q_\tau^{L(0)-\frac{\ccharge}{24}})
			\in\mathscr{C}(E_\tau,\mpt{p},z,W).
			\]
			These conformal blocks assemble into a holomorphic section $\varphi_I$ of $\mathscr{C}_\H(W)$.
			The map
			\[
			\varPhi\colon
			\ab(\bigoplus_{k=0}^N\Ifusion[W][M^k][M^k])\otimes_\C\O_\H \longrightarrow \mathscr{C}_\H(W),
			\qquad I \otimes f \longmapsto f\varphi_I,
			\]
			is an isomorphism of holomorphic vector bundles with flat connection, where the source carries the trivial connection. Thus bases of the intertwining-operator spaces define a global horizontal frame, whose value at each $\tau\in\H$ gives a basis of the conformal-block fiber $\mathscr{C}(E_\tau,\mpt{p},z, W)$.
			
			\item \textup{\textbf{Modular invariance.}}
			\label{main:intro:revised:modular-covariance}
			For each intertwining operator $I$, write $Z_I(u,\tau):=\braket<\varphi_I(\tau)|u>$ for its associated trace function.
			Choose bases $\Fam{I_j^k}{1\le j\le n_k}$ of the spaces $\Ifusion[W][M^k][M^k]$ for $0\le k\le N$.  
			For every $\gamma=\twobytwo{a}{b}{c}{d}\in\SL(2,\Z)$, there are unique constants $\gamma_{(k,j)}^{(k',j')} \in \C$ such that
			\[
			Z_{I_j^k}\ab((c\tau+d)^{-L[0]}u,\gamma\tau) = \sum_{k'=0}^N\sum_{j'=1}^{n_{k'}}\gamma_{(k,j)}^{(k',j')}Z_{I_{j'}^{k'}}(u,\tau)
			\]
			for all $u\in W$ and $\tau\in\H$.
			These matrices define a projective right action of $\SL(2,\Z)$ on the space of horizontal sections.
			\item \textup{\textbf{Geometric monodromy.}}
			\label{main:intro:revised:geometric-monodromy}
			The projectivization of $\nabla$ descends to $\mathcal{M}_{1,1}^{\mathrm{an}}$ and agrees with the Damiolini--Gibney--Tarasca connection. Furthermore, the projective modular action in part~\itemonlycref{main:intro:revised:modular-covariance} agrees with the projective monodromy of the DGT connection.
		\end{enumerate}
	\end{mainthm}
	The connection is constructed in \cref{subsec:conformal-block-bundle-and-connection}.
	Part~\itemonlycref{main:intro:revised:trace-frame} follows from \cref{thm:conv,coro:frame,thm:basiscfb}; part~\itemonlycref{main:intro:revised:modular-covariance} from \cref{thm:modularinv}; and part~\itemonlycref{main:intro:revised:geometric-monodromy} from \cref{thm:DGT-comparison,thm:monodromy-representation}.
	
	The trace-frame identification in the above theorem reveals the common geometric source of various properties of trace functions. The recursion relations are the chiral-Lie relations defining the conformal-block fibers, the differential identity expresses horizontality, and the modular action induced by changes of marking realizes the projective monodromy of the descended connection.
	
	The geometric description also raises questions about invariant Hermitian structures and the categorical interpretation of the generalized $S$-matrices \cite{MS89,BK01,EGNO15}; see \S~\ref{sec:modular-invariance-monodromy}.
	
	\subsection{Comparison with the literature}
	
	Our results lie at the intersection of the analytic study of trace functions and the algebraic geometry of conformal blocks. We compare them with pioneering work and recent developments in the literature. 
	
	When $W=V$, the intertwining operators in our main theorem are the vertex operators, so its basis theorem and modular transformation law recover \cite[Theorem~5.3.1 and 5.3.2]{Zhu96}. Indeed, the function $Z_{Y_{M^k}}(a,\tau)$ is exactly the one-point member $S_{M^k,1}((a,z),\tau)$ of Zhu's system of correlation functions
	\[
	\Fam{S_{M^i}((a_1,z_1),\cdots,(a_n,z_n),\tau)}{0\le i\le N}.
	\]
	Dong--Li--Mason extended Zhu's modularity to twisted sectors in orbifold theory \cite{DLM00}. We expect that our geometric method can be adapted to the twisted setting, yielding a trace-frame theorem and modular invariance for intertwining operators among twisted modules.
	
	For general module insertions, Miyamoto \cite{Mi00} established modular invariance and a spanning theorem for intertwining-operator traces.
	Under our weaker hypotheses, his four one-point-function conditions characterize the global horizontal sections of $\mathscr{C}_\H(W)$ (\cref{thm:Miyamoto-comparison}). The horizontal trace frame therefore gives a basis of this space, yielding unique expansions with constant coefficients. This strengthens Miyamoto's spanning result \cite[Theorem~4.15]{Mi00} with a relaxed hypothesis: neither irreducibility nor Virasoro-generation is required.
	
	For multiple insertions, Huang \cite{H05b} constructed genus-one correlation functions from products of geometrically modified intertwining operators. Under rationality, $C_2$-cofiniteness, and CFT-type hypotheses, he proved convergence of the trace series, genus-one associativity and commutativity, and modular invariance. 
	Beyond the rational setting, Miyamoto introduced pseudo-traces of vertex operators \cite{Mi04}, and Huang studied pseudo-traces of geometrically modified logarithmic intertwining operators \cite{H24}. Under their respective hypotheses, these works prove modular invariance and spanning results for their (pseudo-)trace functions.
	
	In these function-valued theories, genus-one recursion and differential identities are part of the definition, in addition to holomorphy. Huang's differential equations and modular transformation laws suggest that his correlation functions could be realized as multivalued horizontal sections of some vector bundles with flat connections over moduli spaces of coordinated pointed genus-one Riemann surfaces \cite[Remark~7.6]{H05b}.
	
	In contrast, our approach grows from the geometric theory of conformal blocks, where conformal blocks on a fixed (coordinated) pointed curve are defined as invariant linear functionals for the chiral Lie algebras. Such a theory was developed in \cite{FBZ04}, extending the pioneering WZNW theory \cite{TK87,TUY89,Uen08}. 
	Our connection is closely related to the projectively flat connection constructed by Damiolini, Gibney, and Tarasca over moduli of stable pointed curves \cite{DGT21}. This geometric theory was further developed in \cite{DGT24,DGK25,DW25}. In our rational one-point setting, the modular matrices of intertwining-operator traces describe the projective
	monodromy after descent to $\mathcal{M}_{1,1}^{\mathrm{an}}$.
	
	Modular transformation laws for characters in affine and Virasoro models are classical \cite{KP84,RW84,KW88}.
	For affine $\widehat{\sl}_n$ at positive-integral levels, Etingof and Kirillov realized affine Jack polynomials as normalized intertwiner traces with an additional Cartan variable and proved their modular transformation laws \cite{EK95}.
	For positive-integral-level affine $\sl_2$ VOAs, Krauel, Shafiq, and Wood determined the dimensions and natural bases of ordinary one-point trace spaces and described their modular action in categorical terms \cite{KSW26}.
	For Virasoro minimal-model VOAs, Krauel and Marks studied the modular representations and spaces of vector-valued modular forms arising from intertwining-operator traces, including congruence properties \cite{KM18}. These models, studied in \cite{DMZ94},
	are rational \cite{W93} and regular \cite[Theorem~3.13]{DLM97}, and hence fall within the scope of our main theorem.
	Beyond these affine and Virasoro models, our theorem also applies to the principal $\mathscr{W}$-algebras at the nondegenerate admissible levels \cite{Ara15,ACL19}.
	
	We can summarize the above comparisons in the following diagram.

	\[
	\color{black}
	\begin{tikzpicture}[
		fieldboundary/.style={draw=black,densely dotted,thick,rounded corners=9mm},
		topic/.style={draw=black,dashed,rounded corners=2pt,fill=white,
			align=center,text width=3cm,inner sep=2pt,font=\footnotesize},
		contribution/.style={circle,draw=black,very thick,fill=white,
			align=center,minimum size=2.7cm,inner sep=4pt,font=\sffamily},
		relation/.style={draw=black,thick,->,>=latex,
			shorten <=1pt,shorten >=1pt},
		relabel/.style={font=\tiny,align=center,fill=white,inner sep=1pt}
		]
		\draw[fieldboundary] (-7.5,-3) rectangle (4.2,3.3);
		\draw[fieldboundary] (0.2,-3.3) rectangle (8.5,3);
		\node[font=\small\bfseries,color=gray,align=center] at (-5.5,2.6) 
		{
			Modular\\ invariance
		};
		\node[font=\small\bfseries,color=gray] at (6.4,-2.6) {Conformal blocks};
		
		\node[topic] (vertexmodels) at (-5.5,1.2)
		{
			\textbf{Characters for Affine/Virasoro}\\
			{\tiny\cite{KP84,RW84,KW88}}
		};
		
		\node[topic] (vertexvoa) at (-5.5,-1.05)
		{
			\textbf{Traces for VOAs}\\
			{strongly rational}
			{\tiny\cite{Zhu96,DLM00}}\\
			$C_2$-cofinite 
			{\tiny\cite{Mi04}}
		};
		
		\node[topic] (intertwinermodels) at (-1.85,1.8)
		{
			\textbf{Intertwiner traces for Affine/Virasoro}\\
			{\tiny\cite{EK95,KM18,KSW26}}
		};
		
		\node[topic] (intertwinervoa) at (-1.85,-1)
		{
			\textbf{Intertwiner traces for VOAs}\\
			{Strongly rational}\\
			{\tiny\cite{Mi00,H05b}}\\
			Logarithmic 
			{\tiny\cite{H24}}
		};
		
		\node[topic,text width=3.4cm] (wznw) at (6.4,1.9)
		{
			\textbf{WZNW theory}\\
			{\tiny\cite{TK87,TUY89,Uen08}}
		};
		
		\node[topic,text width=3.4cm] (geometry) at (6.4,-1)
		{
			\textbf{Conformal blocks}\\
			Vector bundles and connections\\
			{\tiny\cite{FBZ04,DGT21,DGT24,DGK25,DW25}}
		};
		
		\node[contribution] (paper) at (2.2,1) {This paper};
		\node[font=\footnotesize] at ($(paper.south)+(0,-0.9)$) {basis $\rightarrow$ \begin{tabular}[c]{@{}c@{}}horizontal\\frame\end{tabular}};
		\node[font=\footnotesize] at ($(paper.south)+(0,-1.8)$) {\begin{tabular}[c]{@{}c@{}}modular\\action\end{tabular} $\leftarrow$ monodromy};
		
		\draw[relation] (vertexmodels.south) -- (vertexvoa.north);
		\draw[relation] (vertexmodels.east) -- (intertwinermodels.west |- vertexmodels.east);
		\draw[relation] (vertexvoa.east) -- (intertwinervoa.west |- vertexvoa.east);
		\draw[relation] (intertwinermodels.south) -- (intertwinervoa.north);
		\draw[relation] (wznw.south) -- (geometry.north);
		
		\draw[relation] (intertwinervoa.east)
		to[out=0,in=180,out distance=6mm,in distance=6mm]
		node[relabel,pos=0.65,xshift=-1mm,yshift=-2mm] {rational\\conditions} (paper.west);
		\draw[relation,<-] (paper.east)
		to[out=0,in=180,out distance=6mm,in distance=6mm]
		node[relabel,pos=0.4,xshift=1mm] {genus-one} (geometry.west);
	\end{tikzpicture}
	\]
	
	\subsection{Fusion rules and ranks}
	
	Factorization and local freeness \cite{DGT24,DGK25} reduce ranks of conformal blocks to the dimensions of three-point genus-zero ones, hence to fusion rules. Comparing different degenerations gives identities among these rules.
	
	Fusion-rule associativity has a geometric sewing interpretation \cite{V88,MS89,Uen08} and a categorical formulation through tensor product theories for suitable affine Lie algebra modules \cite{KL93} and VOA modules \cite{HL95,H95,HLep99,H05a,H08b,HLZ1,HLZ2}. The categorical associators are constructed in vertex tensor category theory \cite{H95,HLep99,HLZ2}. At the level of fusion rules, the VOA CohFT construction also yields an associative fusion algebra \cite{DGT22}. We prove the fusion-ring associativity identity by comparing two nodal degenerations of a four-pointed rational curve, using factorization and local freeness (see \cref{thm:fusion-associativity}).
	
	For ranks, the geometers' Verlinde formula expresses conformal-block dimensions in terms of fusion rules, as in WZNW theory \cite{Uen08} and VOA CohFT theory \cite{DGT22}. This rank formula should not be confused with VOA Verlinde formula, which expresses fusion rules through the $S$-matrix \cite{V88,H08a}. Repeated nodal factorization gives us the rank formula directly in terms of the fusion matrices defined below.
	
	Write $(\vect{N}_i)_j^k=\Nfusion[M^i][M^j][M^k]$ for the fusion matrices, let $M^{i^\dagger}\cong(M^i)'$, and put $\vect{H}=\sum_{i=0}^N\vect{N}_i\vect{N}_{i^\dagger}$ for the handle matrix. For $i_\bullet=(i_1,\cdots,i_n)$ and a stable $n$-pointed curve $(X,\mpt{p}_\bullet)$ of genus $g$, we obtain
	\[
	\dim\cova{M^{i_\bullet}}_{(X,\mpt{p}_\bullet)}
	=\ab(\vect{N}_{i_1}\cdots\vect{N}_{i_n}\vect{H}^g)_0^{0^\dagger},
	\qquad 2g-2+n>0
	\]
	(see \cref{thm:Verlinde}). For $g\ge1$, the matrix expression also equals
	\[
	\operatorname{Tr}\ab(\vect{N}_{i_1}\cdots\vect{N}_{i_n}\vect{H}^{g-1}).
	\]
	
	\subsection{Main ideas of the proof}
	
	The genus-one one-point case of the rank formula gives
	\[
	\operatorname{rank}\mathscr{C}_\H(W)
	=\dim\mathscr{C}(E_\tau,\mpt{p},z,W)
	=\sum_{k=0}^N\dim\Ifusion[W][M^k][M^k].
	\]
	To prove the trace-frame theorem, it therefore suffices to construct independent horizontal sections from independent intertwining operators. We start with the three-point conformal block on $\PP^1$ determined by an intertwining operator.
	Sewing the points carrying $M$ and $M'$ gives the formal genus-one
	trace. We extend the trace calculations of \cite[\S~4]{Zhu96} and
	\cite{DLM00} from vertex operators to intertwining operators. Finite
	generation of the formal elliptic coinvariants reduces the differential
	trace identity to a first-order differential system for evaluations on
	a finite-dimensional subspace of $W$. The system is regular singular
	at the cusp and has no other singularities on the punctured $q$-disk,
	giving convergence on $\H$. The recursion identities show that the
	convergent traces satisfy the chiral-Lie relations, while the
	differential identity expresses their horizontality.
	
	For a choice of bases in the intertwining-operator spaces, the leading
	trace coefficients are functionals on the top degrees of the
	irreducible modules. The fusion rules theorem
	\cite{FZ92,Li99,Liu23,GLZ25,Liu26}, together with the semisimplicity of
	Zhu's algebra, shows that these functionals are linearly independent on
	the bimodule $A(W)$. The trace sections consequently have an invertible
	evaluation matrix near the cusp. Their fiberwise rank is constant by
	horizontality, so they are independent throughout $\H$. Their number
	is the rank computed above; hence they form a basis of every fiber and
	a global horizontal frame, giving the flat-bundle isomorphism.
	
	Changes of marking preserve the connection and act on its horizontal
	frame through constant matrices. Comparing the deformation of the
	universal elliptic curve with Virasoro uniformization identifies the
	projectivization of the trace connection with the DGT pullback; after
	modular descent, these matrices describe its projective monodromy.
	
	\subsection{Modular forms}
	
	Evaluating the trace frame on homogeneous vectors gives
	vector-valued modular transformation laws, with Fourier exponents
	determined by the conformal weights. For holomorphic VOAs, Zhu's
	theorem relates characters to classical modular functions. The
	moonshine VOA, for instance, has character
	$\mathrm{Ch}_qV^\natural=j(q)-744$ \cite{Zhu96,FLM88}. Modularity also
	provides tools for studying the structure and classification of
	holomorphic VOAs of small central charge \cite{DM04}.
	
	Intertwining-operator traces also yield scalar modular forms for VOAs
	that need not be holomorphic. Suppose the trace space is
	one-dimensional, generated by $Z_I$ for a nonzero
	$I\in\Ifusion[W][M][M]$ with $M\in\Irr$, and
	$h_M-\frac{\ccharge}{24}$ is a nonnegative integer.
	For an $L[0]$-homogeneous $u\in W$ of nonnegative integral weight $r$,
	every nonzero such trace is a holomorphic modular form of weight $r$
	(\cref{thm:modularform}). For the type-$A_1$ affine VOA at level $16$,
	we realize the cusp form
	$\Delta(\tau)=q_\tau\prod_{n=1}^\infty(1-q_\tau^n)^{24}$ as such a trace
	(\cref{thm:affinecusp}).
	
	We discuss broader realization questions and the congruence properties of the associated modular representation \cite{DLN15,CDT25} in \S~\ref{sec:modularforms}.
	
	\subsection{Organization of the paper}
	
	Section~\ref{sec:prelim} recalls chiral Lie algebras, coinvariants, conformal blocks, factorization, and local freeness. We first apply this framework in Section~\ref{sec:fusionrank} to prove the fusion-ring identities and the rank formula.
	
	The genus-one analysis begins in Section~\ref{sec:cfbontorus}, which describes the elliptic conformal-block relations in the square-bracket realization and proves finite generation of formal coinvariants.
	Section~\ref{sec:tracefunctions} then establishes the trace identities and convergence. 
	Section~\ref{sec:trace-frame} is devoted to the trace-frame theorem. We first construct the connection and identify its projectivization with the pullback of the DGT connection. Then we prove the trace-frame theorem and compare with Miyamoto's one-point function theory. Section~\ref{sec:modular-invariance-monodromy} then derives modular invariance and identifies the resulting projective representation with the monodromy of the DGT connection. Section~\ref{sec:modularforms} develops applications to modular forms and discusses questions about their realization and the congruence properties of the modular representation.
	
\subsection{AI Declaration}
We used ChatGPT (GPT-5.6 and GPT-6) only for checking the correctness of the proofs of the theorems, lemmas, and propositions and for correcting grammar. All the proofs, the whole structure of the paper, and all the main results and corollaries are human-generated from the collaboration of the authors over the past two years. We take full responsibility for the correctness of all the theorems and results in this manuscript.

	\section{Preliminaries}\label{sec:prelim}
	
	We now fix the representation-theoretic and geometric framework in which all subsequent constructions take place.
	
	\subsection{Vertex operator algebras and their modules}
	
	We review the notions of a vertex operator algebra and its modules \cite{Bor86,DL93,FBZ04,FHL93,LL04,Zhu94,Zhu96}.
	\begin{definition}\label{def:VOA}
		A \concept{vertex operator algebra} (\emph{VOA} for short) is a quadruple $(V, Y(\placeholder,z), \vac, \cfv)$ where
		\begin{itemize}
			\item $V=\bigoplus_{k\in\N}V_k$ is a $\N$-graded vector space with $\dim{V_k}<\infty$, where throughout this paper $\N=\Set{0,1,2,\cdots}$ denotes the set of nonnegative integers;
			\item $Y(\placeholder,z) \colon V\to \End{V}\dbrack{z^{\pm1}}$ is a linear map assigning to each $a\in V$ the \concept{vertex operator} \[Y(a,z)=\sum\limits_{n\in\Z}\vo{a}{n}z^{-n-1};\]
			\item $\vac\in V_0$ and $\cfv\in V_2$ are two distinguished vectors, called the \concept{vacuum vector} and the \concept{conformal vector} respectively.
		\end{itemize}
		These data must satisfy the following axioms:
		\begin{enumerate}[label={V\arabic*.}]
			\item \concept{(Truncation property)} for all $a,b \in V$, $\vo{a}{n}b=0$ for sufficiently large $n$;
			\item \concept{(Vacuum property)} $Y(\vac, z)=\Id_V$;
			\item \concept{(Creation property)} $Y(a, z){\vac}\in V\dbrack{z}$ and $\lim\limits_{z\to 0}Y(a, z){\vac}=a$ for $a\in V$;
			\item \concept{(Jacobi identity)} for all $a,b \in V$,
			\begin{multline*}
				\Res_{z}(Y(a,z)Y(b,w)\iota_{z^{-1}}F\ab(z,w))
				-\Res_{z}(Y(b,w)Y(a,z)\iota_{z}F\ab(z,w))\\
				=\Res_{z-w}(Y(Y(a,z-w)b,w)\iota_{z-w}F\ab(z,w))
			\end{multline*}
			holds for every rational function $F\ab(z,w)\in\C[z^{\pm1},w^{\pm1},(z-w)^{\pm1}]$;
			\item \concept{(Virasoro relations)} defining $L(n)$ ($n\in\Z$) as the coefficients of $Y(\cfv, z)$:
			\[
			Y(\cfv, z)=\sum_{n\in \Z}\vo{\cfv}{n}z^{-n-1}=\sum_{n\in \Z}\vo{L}{n}z^{-n-2},
			\]
			then there is a complex number $\ccharge$, called the \concept{central charge}, such that
			\[
			\ab[L(m),L(n)] = (m-n)L(m+n) + \frac{\ccharge}{12}\delta_{m+n,0}(m^3-m)\Id_V;
			\]
			\item \concept{($L(0)$-eigenspace decomposition)} $L(0)a=na$ for all homogeneous $a\in V_n$. That is to say, the \concept{weight} of $a$, denoted by $\wt a$, is the corresponding eigenvalue of $L(0)$;
			\item \concept{($L(-1)$-derivative property)} $Y(L(-1)a, z)=\odv{z} Y(a, z)$.
		\end{enumerate}
		In particular, $V$ is a module over the Virasoro algebra $\mathrm{Vir}_{\cfv}=\operatorname{span}\Set{L(n), C \given n\in \Z}$.
	\end{definition}
	
	\begin{definition}\label{def:twistedmodule}
		A \concept{weak $V$-module} is a vector space $M$ equipped with a linear map
		$Y_{M}(\placeholder,z) \colon V\to \End{M}\dbrack{z}$ assigning to each $a\in V$ a \concept{vertex operator}
		\[
		Y_{M}(a,z) = \sum_{n\in\Z} \vo{a}{n} z^{-n-1},
		\]
		satisfying the following axioms:
		\begin{enumerate}[label={M\arabic*.}]
			\item \concept{(Truncation property)} for all $a\in V$ and $u\in M$, $\vo{a}{n}u=0$ for sufficiently large $n$;
			\item \concept{(Vacuum property)} $Y_{M}(\vac,z)=\Id_{M}$.
			\item \concept{(Jacobi identity)} for all $a,b\in V$,
			\begin{multline*}
				\Res_{z}(Y_M(a,z)Y_M(b,w)\iota_{z^{-1}}F\ab(z,w))
				-\Res_{z}(Y_M(b,w)Y_M(a,z)\iota_{z}F\ab(z,w))\\
				=\Res_{z-w}(Y_M(Y(a,z-w)b,w)\iota_{z-w}F\ab(z,w))
			\end{multline*}
			holds for every rational function $F\ab(z,w)\in\C[z^{\pm1},w^{\pm1},(z-w)^{\pm1}]$.
		\end{enumerate}
	\end{definition}
	\begin{definition}\label{def:Adm}
		A weak $V$-module $M$ is called an \concept{admissible $V$-module} if it admits a subspace decomposition $M=\bigoplus_{n\in \N}M(n)$ such that
		\begin{equation}\label{eq:admissible}
			\vo{a}{m}M(n)\subset M(\wt a-m-1+n)
		\end{equation}
		for any homogeneous $a\in V$, any $m\in \Z$, and any $n\in \N$.
	\end{definition}
	\begin{definition}
		An admissible module $M$ is said to be \concept{of conformal weight $h\in\C$} if $M(0)\neq0$, $L(0)$ acts semi-simply on $M$, and each $M(n)$ is the eigenspace $M_{h+n}$ of eigenvalue $h+n$.
		An admissible module $M$ of conformal weight $h$ is called \concept{ordinary} if $\dim M(n)<\infty$ for all $n\ge 0$.
	\end{definition}
	
	\begin{definition}\label{def:propertiesofVOAs}
		A vertex operator algebra $V$ is called
		\begin{itemize}
			\item \concept{of CFT type} if $V_0=\C\vac$;
			\item \concept{rational} if every admissible $V$-module is a direct sum of simple admissible $V$-modules;
			\item \concept{$C_2$-cofinite} if the subspace \[C_2(V):=\operatorname{span}\Set*{\vo{a}{-2}b \given a,b\in V}\] has finite codimension in $V$.
			\item \concept{strongly rational} if  $V$ is of CFT type, simple, rational, $C_2$-cofinite, and self-dual $V\cong V'$.  %
		\end{itemize}
	\end{definition}
	
	Rationality implies that there are finitely many irreducible admissible modules up to isomorphism and that each of them is ordinary \cite{DLM98}. If $V$ is also $C_2$-cofinite, their conformal weights and the central charge are rational numbers \cite[\S~11]{DLM00}.
	
	\begin{definition}
		For a weak $V$-module $M$, set
		\[
		C_2(M):=\operatorname{span}\Set{a_{(-2)}u\given a\in V,\ u\in M}.
		\]
		Then $M$ is called \concept{$C_2$-cofinite} if $\dim M/C_2(M)<\infty$.
	\end{definition}

	\begin{definition}[{\cite{FHL93,FZ92}}]\label{def:IO}
		Let $V$ be a VOA, and let $M^1,M^2$, and $M^3$ be admissible $V$-modules of conformal weights $h_1,h_2$, and $h_3$, respectively. Set $h:=h_1+h_2-h_3$.
		An intertwining operator of type $\fusion$ is a linear map $I \colon M^1\longrightarrow \Hom(M^2,M^3)\dbrack{w^{\pm1}}w^{-h}$ assigning to each $u\in M^1$ a field
		\[
		I(u,w)=\sum_{n\in \Z} \vo{u}{n} w^{-n-1-h},
		\]
		satisfying the following axioms:
		\begin{enumerate}[label={I\arabic*.}]
			\item \concept{(Truncation property)} for all $u\in M^1$ and $v\in M^2$, $\vo{u}{n}v=0$ for sufficiently large $n$;
			\item \concept{($L(-1)$-derivative property)} $I(L(-1)u,w)=\odv{w} I(u, w)$ for any $u\in M^1$.
			\item \concept{(Jacobi identity)} for all $a\in V$ and $u\in M^1$,
			\begin{multline*}
				\Res_{z}(Y_{M^3}(a,z)I(u,w)w^h\iota_{z^{-1}}F\ab(z,w))
				-\Res_{z}(I(u,w)Y_{M^2}(a,z)w^h\iota_{z}F\ab(z,w))\\
				=\Res_{z-w}(I(Y_{M^1}(a,z-w)u,w)w^h\iota_{z-w}F\ab(z,w))
			\end{multline*}
			holds for every rational function $F\ab(z,w)\in\C[z^{\pm1},w^{\pm1},(z-w)^{\pm1}]$.
		\end{enumerate}
		The vector space of intertwining operators of type $\fusion$ is denoted by $\Ifusion$. Its dimension, denoted by $\Nfusion$, is called the fusion rule among $M^1,M^2$, and $M^3$.
	\end{definition}

	\subsection{The chiral Lie algebra and its coinvariants}\label{subsec:chiral}
	
	Following \cite{Bor86}, there is a Lie algebra $\mathfrak{L}(V)$ ancillary to a VOA $V$.
	The definition of $\mathfrak{L}(V)$ is given as follows:
	
	Let $z$ be a formal variable and $\nabla$ the $\C$-linear operator $L(-1)\otimes\Id + \Id\otimes\odv{z}$ on the vector space $V\otimes\C[z^{\pm1}]$.
	Then the underlying vector space of $\mathfrak{L}(V)$ is the quotient
	\[\mathfrak{L}(V):=V\otimes\C[z^{\pm1}]/\Image{\nabla}.\]
	Let $\lo{a}{m}$ denote the projection of $a\otimes z^m$ in $\mathfrak{L}(V)$.
	Then the Lie algebra structure on $\mathfrak{L}(V)$ is given by the following formula:
	\[
	\ab[\lo{a}{m},\lo{b}{n}]
	=\sum_{j\ge 0}\binom{m}{j} \lo{(\vo{a}{j}b)}{m+n-j}.
	\]
	
	Moreover, $\mathfrak{L}(V)$ has a natural gradation given by $\deg \lo{a}{m}:=\wt a-m-1$, where $m\in\Z$ and $a$ is a homogeneous element of $V$.
	Let $\mathfrak{L}(V)_n$ be the subspace of $\mathfrak{L}(V)$ spanned by elements of degree $n\in \Z$.
	Then we have a triangular decomposition:
	\begin{equation*}
		\mathfrak{L}(V)=\mathfrak{L}(V)_{<0}\oplus \mathfrak{L}(V)_0\oplus \mathfrak{L}(V)_{>0}.
	\end{equation*}
	The above constructions extend to their $z$-adic completions (resp. $z^{-1}$-adic completions), resulting a topologically complete Lie algebra $\mathfrak{L}(V)^{\mathsf{L}}$ (resp. $\mathfrak{L}(V)^{\mathsf{R}}$).

	Every weak $V$-module $M$ admits a $\mathfrak{L}(V)$-module structure via the Lie algebra homomorphism
	\[
	\mathfrak{L}(V)\to \End(M),\ \lo{a}{m}\mapsto \vo{a}{m}:=\Res_{z}Y_{M}(a,z)z^{m},\quad a\in V, m\in \Z.
	\]
	Moreover, $M$ is an admissible $V$-module if and only if $M$ its $\mathfrak{L}(V)$-module structure respects the gradation of $\mathfrak{L}(V)$.
	
	There is an anti-isomorphism between the Lie algebras $\mathfrak{L}(V)^{\mathsf{L}}$ and $\mathfrak{L}(V)^{\mathsf{R}}$ given by the formula
	\begin{equation}\label{eq:def:theta}
		\theta(\lo{a}{m}):=\sum_{j\ge 0} \frac{(-1)^{\wt a}}{j!}
		\lo{(L(1)^ja)}{\wt a-j-1 + \deg\lo{a}{m}}.
	\end{equation}
	
	Let $M$ be an admissible $V$-module. As a \emph{left} $\mathfrak{L}(V)^{\mathsf{L}}$-module, its \concept{graded dual space} $M'=\bigoplus_{n\in \N}M(n)^{\ast}$ naturally carries a \emph{right} $\mathfrak{L}(V)^{\mathsf{R}}$-module structure and hence a \emph{left} $\mathfrak{L}(V)^{\mathsf{L}}$-module structure via the anti-isomorphism $\theta$.
	More precisely, for any $a\in V$ and $u'\in M'$, we have
	\[
	\braket<\lo{a}{m}.u'|u>=\braket<u'|\theta(\lo{a}{m})u>,\quad u\in M.
	\]
	Moreover, the above $\mathfrak{L}(V)^{\mathsf{L}}$-module structure on $M'$ is \concept{coherent} in the sense that it arises from a $V$-module structure; see \cite{FHL93}
	
	Following \cite[\S 2]{DGT24}, for each stable complex algebraic curve $X$, we can construct a left $\mathscr{D}$-module $(\mathcal{V},\nabla)$ on $X$ from the vertex operator algebra $V$.
	\begin{definition}\label{df:chiralLie}
		The \concept{chiral Lie algebra} $\mathcal{L}_X(V)$ of $V$ on $X$ is the sheaf of de Rham cohomology of the $\mathscr{D}$-module $(\mathcal{V},\nabla)$
		\[
		\Omega\longmapsto
		H^0\ab(\Omega,\mathcal{V}\otimes\omega_X/\Image{\nabla}),
		\]
		where $\omega_X$ is the dualizing sheaf.
		
		Let $\mpt{p}_\bullet$ be an $n$-tuple of distinct points on $X$. For each $i$, let $D_{\mpt{p}_i}^\times$ be the punctured formal disk at $\mpt{p}_i$ on $X$, let $z_i$ be a local coordinate at $\mpt{p}_i$, and let $\mathfrak{L}_{z_i}(V)$, $\mathfrak{L}_{z_i}(V)^{\mathsf{L}}$, and $\mathfrak{L}_{z_i}(V)^{\mathsf{R}}$ be the ancillary Lie algebras of $V$ with formal variable $z_i$.
		Then there is a Lie algebra homomorphism
		\[
		\iota_{z_i} \colon \mathcal{L}_{X\setminus\mpt{p}_\bullet}(V)\longrightarrow\mathcal{L}_{D_{\mpt{p}_i}^\times}(V)\cong\mathfrak{L}_{z_i}(V)^{\mathsf{L}},
		\]
		where the first homomorphism $\iota_{\mpt{p}_i}$ comes from the pullback along $D_{\mpt{p}_i}^\times\to X$, while the second isomorphism is given by
		\[
		b\otimes f \longmapsto
		\Res_{z_i=0}\ab(\sum_{n\in\Z}\lo{b}{n}z_i^{-n-1})\iota_{z_i}f.
		\]
	\end{definition}

	\begin{definition}\label{def:coinv}
		Let $M^\bullet$ be an $n$-tuple of admissible $V$-modules.
		When $X\setminus\mpt{p}_\bullet$ is affine, the space of \concept{coinvariants} at $(X,\mpt{p}_\bullet,z_\bullet)$ is
		\[
		\cova{M^\bullet}_{(X,\mpt{p}_\bullet,z_\bullet)}:=
		M^\bullet_{\mathcal{L}_{X\setminus\mpt{p}_\bullet}(V)} =
		M^\bullet/\mathcal{L}_{X\setminus\mpt{p}_\bullet}(V). M^\bullet,
		\]
		where $\mathcal{L}_{X\setminus\mpt{p}_\bullet}(V)$ acts diagonally on the tensor product $M^\bullet$ via the sum of the homomorphisms $\iota_{z_i}$. Its linear dual space $(\cova{M^\bullet}_{(X,\mpt{p}_\bullet,z_\bullet)})^\ast$ is called the space of \concept{conformal blocks} at $(X,\mpt{p}_\bullet,z_\bullet)$, denoted by $\mathscr{C}(X,\mpt{p}_\bullet,z_\bullet,M^\bullet)$.
		In general, the space of coinvariants $\cova{M^\bullet}_{(X,\mpt{p}_\bullet,z_\bullet)}$ and conformal blocks $\mathscr{C}(X,\mpt{p}_\bullet,z_\bullet,M^\bullet)$ is defined as a direct limit.
		
		It was proved in \cite[\S 8]{DGT24} that the space of coinvariants does not depend on the choice of local coordinates $z_\bullet$. We also denote the coinvariants $\cova{M^\bullet}_{(X,\mpt{p}_\bullet,z_\bullet)}$ by $\cova{M^\bullet}_{(X,\mpt{p}_\bullet)}$ for simplicity; see also \cite{FBZ04,Zhu94}.
	\end{definition}
	
	It is well-known that there is a one-to-one correspondence between the space of conformal blocks on three-pointed $\PP^1$ and the space of intertwining operators of VOAs; see \cite{TUY89,FBZ04,NT05,GLZ25}. In explicit:
	\begin{proposition}\label{prop:threepointcfbIO} 
		Let $M^1, M^2, M^3$ be irreducible ordinary $V$-modules. Then the space of three-point conformal blocks associated to $(M^3)', M^1\otimes, M^2$ on the sphere is isomorphic to the space of intertwining operators $ \Ifusion$, where $(M^3)'$ is the contragredient $V$-module  \cite{FHL93}. These vector spaces are finite-dimensional if $V$ is $C_2$-cofinite.
		In particular,
		\[
		\Nfusion = \dim\mathscr{C}((\PP^1,\mpt{p}_\bullet,z_\bullet), (M^3)'\otimes  M^1\otimes  M^2).
		\]
	\end{proposition}

	\subsection{Factorization and the vector bundle corollary}\label{subsec:FTVB}
	
	The construction extends to a flat family of stable pointed curves $\pi \colon \sch{X}\to S$ with marked sections $\mPt{P}_\bullet$.
	Let
	$\sch{X}^{\circ} := \sch{X}\setminus\mPt{P}_\bullet$
	and let
	$\pi^{\circ} \colon \sch{X}^{\circ}\to S$ be the induced morphism.
	The resulting sheaf of chiral Lie algebras on $S$ is
	\begin{equation}\label{eq:familychiralLie}
		\mathscr{L}_{\sch{X}\setminus\mPt{P}_\bullet}(V)
		:=
		\pi^{\circ}_*\ab(
		\mathcal{V}_{\sch{X}}\otimes\omega_{\sch{X}/S}/\Image{\nabla}
		),
	\end{equation}
	where $(\mathcal{V}_{\sch{X}},\nabla)$ is the corresponding relative vertex-algebra sheaf.
	Its fiber at $s\in S$, written
	$\ab(\mathscr{L}_{\sch{X}\setminus\mPt{P}_\bullet}(V))_s$, is canonically isomorphic to
	$\mathcal{L}_{X_s\setminus P_{\bullet,s}}(V)$.
	Then, for any $n$-tuple of admissible $V$-modules $W^{\bullet}$ and any choice of local coordinates $z_\bullet$, we also have the sheaf $\cova{W^\bullet}_{(\sch{X}/S,\mPt{P}_\bullet,z_\bullet)}$ of coinvariants and the sheaf $\mathscr{C}(\sch{X}/S,\mPt{P}_\bullet,z_\bullet,W^\bullet)$ of conformal blocks on $S$.
	
	We thus obtain sheaves on the moduli space of \concept{coordinated} stable $n$-pointed curves. After descent, we obtain sheaves on the moduli space $\overline{\mathcal{M}}_{g,n}$ of stable $n$-pointed curves.
	For the fixed-weight descent, see \cite[\S~6.3.2]{DGT21} and \cite[\S~8.7, Remark~8.7.3(2)]{DGT24}.
	We denote these sheaves by $\cova{W^\bullet}$ and $\mathscr{C}(W^\bullet)$. In this subsection, we only discuss coinvariants, and statements for conformal blocks are obtained by dualizing.
	
	We recall the factorization and sewing constructions of \cite{DGT24,DGK25}. Throughout this subsection we assume that $V$ is rational and of CFT type, and that every admissible $V$-module inserted at a marked point is finitely generated, following the conventions of \cite{DGT24}.
	\begin{theorem}
		[Propagation of Vacua]\label{thm:propagation-of-vacua}
		Assume $\sch{X}\setminus\mPt{P}_\bullet$ is affine over $S$.
		Then the linear map
		\[
		W^\bullet \longrightarrow W^\bullet\otimes V,\qquad u\longmapsto u\otimes\vac
		\]
		induces an isomorphism of $\mathcal{O}_S$-modules
		\[
		\cova{W^\bullet}_{(\sch{X}/S,\mPt{P}_\bullet,z_\bullet)}\isomto\cova{W^\bullet\otimes V}_{(\sch{X}/S,\mPt{P}_\bullet\sqcup\mPt{Q},z_\bullet\sqcup{w})},
		\]
		where $\mPt{Q}$ is an $S$-point on $\sch{X}$ disjoint from $\mPt{P}_\bullet$ and $w$ is a local coordinate at $\mPt{Q}$.
		Moreover, this isomorphism is equivariant with respect to changes of coordinates.
	\end{theorem}

	Let $\Irr$ be a complete set of irreducible $V$-modules up to isomorphism.
	For each $M\in\Irr$, denote the unique element of $M(0)\otimes M(0)^\ast$ corresponding to $\Id_{M(0)}\in\End(M(0))$ by $\vac^{M(0)}$.
	\begin{theorem}[Factorization Theorem]\label{thm:factorization}
		Let $(X,\mpt{p}_\bullet)$ be a stable $n$-pointed curve with exactly one simple node $\mpt{q}$.
		The linear map
		\[
		W^\bullet\longrightarrow \bigoplus_{M\in\Irr}W^\bullet\otimes M\otimes M',\qquad u\longmapsto \bigoplus_{M\in\Irr}u\otimes\vac^{M(0)}
		\]
		induces an isomorphism of vector spaces
		\[
		\cova{W^\bullet}_{(X,\mpt{p}_\bullet,z_\bullet)}\isomto\bigoplus_{M\in\Irr}\cova{W^\bullet\otimes M\otimes M'}_{(\widetilde{X},\mpt{p}_\bullet\sqcup\mpt{q}_\pm,z_\bullet\sqcup{w_\pm})},
		\]
		where $\widetilde{X}\to X$ is the normalization of $X$, $\mpt{q}_\pm$ are the preimages of $\mpt{q}$, and $w_\pm$ are the local coordinates at $\mpt{q}_\pm$, respectively.
		Moreover, this isomorphism is equivariant with respect to changes of coordinates.
	\end{theorem}  
	
	\begin{remark}\label{rmk:factorization-cofiniteness}
		\cref{thm:factorization} requires no cofiniteness hypothesis on $V$. The PBW argument in \cite[Proposition~6.2.1]{DGT24} does not require a finite-dimensional complement of $C_1(V)$, and rationality supplies the finite decomposition in \cite[Lemma~6.1.1]{DGT24}.
	\end{remark}
	
	Formal sewing lifts the factorization isomorphism from the nodal fiber to a smoothing family. To state it, we fix the following smoothing data, as in \cite[\S 4.3]{DGK25} and \cite[\S 8.1]{DGT24}.
	
	Let $S=\operatorname{Spec}(\C\dbrack{q})$, and let $X_0$ be a projective curve over $\C$ with a single node $\mpt{q}$ and smooth marked points $\mpt{p}_\bullet=(\mpt{p}_1,\cdots,\mpt{p}_n)$ such that $X_0\setminus\mpt{p}_\bullet$ is affine, together with formal coordinates $t_\bullet=(t_1,\cdots,t_n)$ at $\mpt{p}_\bullet$. Let $\eta \colon \widetilde{X_0}\longrightarrow X_0$ be the normalization at $\mpt{q}$, with $\eta^{-1}(\mpt{q})=\mpt{q}_\pm$, and choose formal coordinates $s_\pm$ at $\mpt{q}_\pm$. The associated sewing construction gives a smoothing family $(\sch{X},\mpt{p}_\bullet,t_\bullet)$ over $S$ with central fiber $(X_0,\mpt{p}_\bullet,t_\bullet)$ and smooth generic fiber. Its completed local model at the node is
	\[
	\widehat{\mathcal{O}}_{\sch{X},\mpt{q}} \cong \C\dbrack{q,s_+,s_-}/\ab(s_+s_- - q)
	\]
	as a $\C\dbrack{q}$-algebra. Let $(\widetilde{\sch{X}},\mpt{p}_\bullet\sqcup\mpt{q}_\pm,t_\bullet\sqcup s_\pm)$ be the trivial extension $\widetilde{X_0}\times S$ with its markings and coordinates. Then the two coordinates $s_+$ and $s_-$ separately generate the ideals of the two sections $\mpt{q}_+$ and $\mpt{q}_-$.
	
	For $M=\bigoplus_{i=0}^\infty M(i)\in \Irr$, define
	\[
	\vac^M:=\sum_{i\ge 0} \vac^{M(i)} q^i\in (M\otimes M')\dbrack{q}.
	\]
	where $\vac^{M(i)}:=\Id_{M(i)}\in \End(M(i))=M(i)\otimes M(i)^\ast$ for all $i$.
	\begin{theorem}[Formal Sewing]
		Assume that the two relative coinvariant sheaves in the top row below are coherent. The map
		\[
		W^\bullet\longrightarrow \bigoplus_{M\in\Irr} (W^\bullet \otimes M\otimes M') \dbrack{q},\qquad u\longmapsto \bigoplus_{M\in\Irr} u\otimes \vac^M,
		\]
		induces a canonical $\C\dbrack{q}$-module isomorphism $\Psi$ such that the following diagram commutes
		\[
		\begin{tikzcd}
			\cova{W^\bullet\dbrack{q}}_{(\sch{X},P_\bullet,t_\bullet)}   \arrow[r,"\Psi"]\arrow[d,two heads]&\bigoplus_{M\in \Irr}\cova{W^\bullet \otimes M\otimes M'}_{(\widetilde{X_0}, P_\bullet \sqcup Q_\pm, t_\bullet \sqcup s_\pm)}\dbrack{q}\arrow[d, two heads]\\
			\cova{W^\bullet}_{(X_0,P_\bullet, t_\bullet)}\arrow[r,"\cong"] &\bigoplus_{M\in \Irr}\cova{W^\bullet \otimes M\otimes M'}_{(\widetilde{X_0}, P_\bullet \sqcup Q_\pm, t_\bullet \sqcup s_\pm)}
		\end{tikzcd}
		\]
	\end{theorem}

	As a consequence, we have
	\begin{corollary}[VB Corollary]\label{coro:VB}
		Let $W^\bullet$ be a tuple of finitely generated admissible $V$-modules. Assume that $\cova{W^\bullet}$ is coherent. Then the sheaves $\cova{W^\bullet}$ and $\mathscr{C}(W^\bullet)$ are vector bundles of finite rank on the moduli space $\overline{\mathcal{M}}_{g,n}$.
	\end{corollary}
	
	\begin{remark}\label{rmk:VB-coherence}
		For $C_2$-cofinite VOAs, the coherence is guaranteed by \cite[Theorem~8.4.2]{DGT24}. 
		In the proof of \cite[Lemma~5.2.5]{DGK25}, $C_1$-cofiniteness is used additionally to obtain the central-fiber isomorphism. 
		With this already supplied by \cref{thm:factorization}, the same argument of \cite[Corollary~5.2.6]{DGK25} gives local freeness under the stated coherence assumptions.
	\end{remark}

	For use below, we introduce a name for the geometric data used in factorization.
	\begin{definition}\label{def:fact}
		A \concept{factorization datum} consists of the following data:
		\begin{itemize}
			\item a smooth $n$-pointed curve $(X_0,\mpt{p}_\bullet)$;
			\item a stable $n$-pointed curve $(X_1,\mpt{p}_\bullet)$ with $m$ nodes $\mpt{q}_\bullet=(\mpt{q}_1,\cdots,\mpt{q}_m)$, obtained as a \concept{stable degeneration} of $(X_0,\mpt{p}_\bullet)$; and
			\item a (possibly disconnected) smooth $(n+2m)$-pointed curve $(\widetilde{X}_1,\mpt{p}_\bullet\sqcup\mpt{q}_{\bullet\pm})$, obtained as the normalization of $(X_1,\mpt{p}_\bullet)$, where $\mpt{q}_{\bullet\pm}$ are the preimages of $\mpt{q}_{\bullet}$.
		\end{itemize}
	\end{definition}

	\section{Fusion rings and ranks of coinvariant bundles}\label{sec:fusionrank}
	
	Before specializing to genus one, we apply factorization theorem to derive basic properties of the fusion ring and a rank formula of coinvariant bundles. They illustrate the geometric method that later supplies the dimension count for the trace-frame theorem. Throughout this section, $V$ is a simple, rational VOA of CFT type, with a complete set of irreducible ordinary modules $\Irr = \Set*{M^0,M^1,\cdots,M^N}$, and we assume that the coinvariant sheaves associated to tuples of modules in $\Irr$ are coherent.
	
	\subsection{The fusion ring}
	
	Consider the free abelian group generated by the isomorphism classes of irreducible ordinary $V$-modules: $\Fu(V)=\bigoplus_{i=0}^N \Z [M^i]$, which is referred to as the \concept{fusion ring} of $V$. Define multiplication on $\Fu(V)$ by
	\begin{equation}\label{eq:fusionmult}
		[M^i]\cdot [M^j]:=\sum_{k=0}^N \Nfusion[M^i][M^j][M^k] [M^k],\quad 0\le i,j\le N.
	\end{equation}
	
	\begin{proposition}\label{prop:FusionId}
		The multiplication on $\Fu(V)$ has an identity.
	\end{proposition}
	\begin{proof}
		For simplicity, assume $V\in\Irr$.
		It suffices to show
		\begin{equation}\label{eq:Nidentity}
			\Nfusion[M^1][V][M^3]=\begin{cases*}
				1 & if $M^1\cong M^3$, \\
				0 & otherwise.
			\end{cases*}
		\end{equation}
		By \nameref{thm:propagation-of-vacua}, $\cova{M^1\otimes V\otimes (M^3)'}_{(\PP^1,(0,1,\infty))}\cong\cova{M^1\otimes(M^3)'}_{(\PP^1,(0,\infty))}$.
		On the other hand, the dual space of $\cova{M^1\otimes(M^3)'}_{(\PP^1,(0,\infty))}$ is precisely $\Hom_V(M^1,M^3)$.
		Then the statement follows.
	\end{proof}
	\begin{proposition}\label{prop:FusionComm}
		The multiplication on $\Fu(V)$ is commutative.
	\end{proposition}
	\begin{proof}
		By \eqref{eq:fusionmult}, we only need to show
		\begin{equation}\label{eq:Ncommutative}
			\Nfusion[M^1][M^2][M^3]=\Nfusion[M^2][M^1][M^3].
		\end{equation}
		Since there is an automorphism of $\PP^1$ interchanging $0$ and $1$, we obtain the desired isomorphism of coinvariants by functoriality.
	\end{proof}
	\begin{theorem}\label{thm:fusion-associativity}
		The multiplication on $\Fu(V)$ is associative.
	\end{theorem}
	\begin{proof}
		It is easy to see from \eqref{eq:fusionmult} that we only need to show
		\begin{equation}\label{eq:Nassociative}
			\sum_{M\in\Irr}\Nfusion[M^1][M^2][M]\Nfusion[M][M^3][M^4] =
			\sum_{M\in\Irr}\Nfusion[M^2][M^3][M]\Nfusion[M^1][M][M^4].
		\end{equation}
		Consider the following factorization datum:
		\[
		\begin{tikzpicture}[dot/.style={circle,draw,fill,inner sep=1pt},baseline=(current bounding box.center),scale=0.9,decoration=snake]
			\coordinate (A) at (0,0);
			\node[label=below:{$X_0$}] at ($(A)+(0,-1.5)$) {};
			\draw (A) circle (1.4);
			\node[dot,label=left:$\mpt{p}_1$] at ($(A)+(-0.7,0)$) {};
			\node[dot,label=below:$\mpt{p}_2$] at ($(A)+(0,-0.7)$) {};
			\node[dot,label=right:$\mpt{p}_3$] at ($(A)+(0.7,0)$) {};
			\node[dot,label=above:$\mpt{p}_4$] at ($(A)+(0,0.7)$) {};
			\path[draw,decorate,->] (2,0) -- (4,0);
			\coordinate (B) at (6,0);
			\node[label=below:{$X_1$}] at ($(B)+(0.2,-1.5)$) {};
			\draw ($(B)+({-sqrt(0.5)},{-sqrt(0.5)})$) circle (1);
			\draw ($(B)+({sqrt(0.5)},{sqrt(0.5)})$) circle (1);
			\node[dot] at (B) {};
			\node[dot,label=left:$\mpt{p}_1$] at ($(B)+(-1,-0.5)$) {};
			\node[dot,label=below:$\mpt{p}_2$] at ($(B)+(-0.5,-1)$) {};
			\node[dot,label=right:$\mpt{p}_3$] at ($(B)+(1,0.5)$) {};
			\node[dot,label=above:$\mpt{p}_4$] at ($(B)+(0.5,1)$) {};
			\path[draw,<-] (8,0) -- (9,0);
			\coordinate (C) at (11,0);
			\draw ($(C)+({-0.1-sqrt(0.5)},{-0.1-sqrt(0.5)})$) circle (1);
			\node[label=below:{$X_{1}^+$}] at ($(C)+(0.5,-1.2)$) {};
			\draw ($(C)+({0.1+sqrt(0.5)},{0.1+sqrt(0.5)})$) circle (1);
			\node[label=below:{$X_{1}^-$}] at ($(C)+(-0.8,1.7)$) {};
			\node[dot,label=left:$\mpt{p}_1$] at ($(C)+(-1.1,-0.6)$) {};
			\node[dot,label=below:$\mpt{p}_2$] at ($(C)+(-0.6,-1.1)$) {};
			\node[dot,label=left:$\mpt{q}_+$] at ($(C)+(-0.1,-0.1)$) {};
			\node[dot,label=right:$\mpt{p}_3$] at ($(C)+(1.1,0.6)$) {};
			\node[dot,label=above:$\mpt{p}_4$] at ($(C)+(0.6,1.1)$) {};
			\node[dot,label=right:$\mpt{q}_-$] at ($(C)+(0.1,0.1)$) {};
		\end{tikzpicture}
		\]
		Then, by \nameref{coro:VB} and \nameref{thm:factorization}, we have
		\begin{align*}
			\MoveEqLeft[1]
			\dim\cova{M^1\otimes M^2\otimes M^3\otimes (M^4)'}_{(X_0,(\mpt{p}_1,\mpt{p}_2,\mpt{p}_3,\mpt{p}_4))} \\
			& = \sum_{M\in\Irr}
			\dim\cova{M^1\otimes M^2\otimes M'}_{(X_1^+,(\mpt{p}_1,\mpt{p}_2,\mpt{q}_+))}\cdot
			\dim\cova{M\otimes M^3\otimes (M^4)'}_{(X_1^-,(\mpt{q}_-,\mpt{p}_3,\mpt{p}_4))}  \\
			& = \sum_{M\in\Irr}\Nfusion[M^1][M^2][M]\Nfusion[M][M^3][M^4].
		\end{align*}
		On the other hand, consider the following factorization datum:
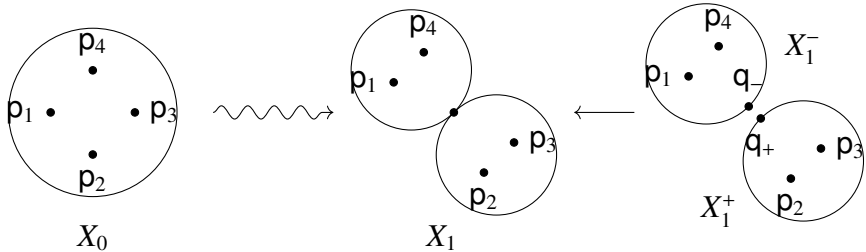
		\[
		\begin{tikzpicture}[dot/.style={circle,draw,fill,inner sep=1pt},baseline=(current bounding box.center),scale=0.9,decoration=snake]
			\coordinate (A) at (0,0);
			\node[label=below:{$X_0$}] at ($(A)+(0,-1.5)$) {};
			\draw (A) circle (1.4);
			\node[dot,label=left:$\mpt{p}_1$] at ($(A)+(-0.7,0)$) {};
			\node[dot,label=below:$\mpt{p}_2$] at ($(A)+(0,-0.7)$) {};
			\node[dot,label=right:$\mpt{p}_3$] at ($(A)+(0.7,0)$) {};
			\node[dot,label=above:$\mpt{p}_4$] at ($(A)+(0,0.7)$) {};
			\path[draw,decorate,->] (2,0) -- (4,0);
			\coordinate (B) at (6,0);
			\node[label=below:{$X_1$}] at ($(B)+(-0.2,-1.5)$) {};
			\draw ($(B)+({-sqrt(0.5)},{sqrt(0.5)})$) circle (1);
			\draw ($(B)+({sqrt(0.5)},{-sqrt(0.5)})$) circle (1);
			\node[dot] at (B) {};
			\node[dot,label=left:$\mpt{p}_1$] at ($(B)+(-1,0.5)$) {};
			\node[dot,label=below:$\mpt{p}_2$] at ($(B)+(0.5,-1)$) {};
			\node[dot,label=right:$\mpt{p}_3$] at ($(B)+(1,-0.5)$) {};
			\node[dot,label=above:$\mpt{p}_4$] at ($(B)+(-0.5,1)$) {};
			\path[draw,<-] (8,0) -- (9,0);
			\coordinate (C) at (11,0);
			\draw ($(C)+({-0.1-sqrt(0.5)},{0.1+sqrt(0.5)})$) circle (1);
			\node[label=below:{$X_{1}^+$}] at ($(C)+(-0.6,-1)$) {};
			\draw ($(C)+({0.1+sqrt(0.5)},{-0.1-sqrt(0.5)})$) circle (1);
			\node[label=below:{$X_{1}^-$}] at ($(C)+(0.8,1.7)$) {};
			\node[dot,label=left:$\mpt{p}_1$] at ($(C)+(-1.1,0.6)$) {};
			\node[dot,label=below:$\mpt{p}_2$] at ($(C)+(0.6,-1.1)$) {};
			\node[dot,label=below:$\mpt{q}_+$] at ($(C)+(0.1,-0.1)$) {};
			\node[dot,label=right:$\mpt{p}_3$] at ($(C)+(1.1,-0.6)$) {};
			\node[dot,label=above:$\mpt{p}_4$] at ($(C)+(-0.6,1.1)$) {};
			\node[dot,label=above:$\mpt{q}_-$] at ($(C)+(-0.1,0.1)$) {};
		\end{tikzpicture}
		\]
		Then we have
		\begin{align*}
			\MoveEqLeft[1]
			\dim\cova{M^1\otimes M^2\otimes M^3\otimes (M^4)'}_{(X_0,(\mpt{p}_1,\mpt{p}_2,\mpt{p}_3,\mpt{p}_4))}  \\
			& = \sum_{M\in\Irr}
			\dim\cova{M^2\otimes M^3\otimes M'}_{(X_1^+,(\mpt{p}_2,\mpt{p}_3,\mpt{q}_+))}\cdot
			\dim\cova{M^1\otimes M\otimes (M^4)'}_{(X_1^-,(\mpt{p}_1,\mpt{q}_-,\mpt{p}_4))} \\
			& = \sum_{M\in\Irr}\Nfusion[M^2][M^3][M]\Nfusion[M^1][M][M^4].
		\end{align*}
		Equation~\eqref{eq:Nassociative} follows.
	\end{proof}
	
	By \cref{prop:FusionId,prop:FusionComm,thm:fusion-associativity}, the fusion ring $\Fu(V)$ is a commutative unital associative algebra.

	\subsection{A rank formula for the coinvariants bundle}
	
	We then package the fusion rules into matrices and use nodal degeneration to compute coinvariant dimensions through the handle matrix.
	We fix a numbering by the bijection
	\[
	\mathcal{I}=\Set{0,1,\cdots,N}\isomto\Irr,
	\qquad i\longmapsto M^i,
	\]
	such that $M^0$ is the irreducible ordinary module in $\Irr$ isomorphic to $V$.
	Define the involution $\dagger$ on $\mathcal{I}$ as follows: $M^{k^\dagger}$ is the irreducible ordinary module in $\Irr$ isomorphic to $(M^k)'$. In particular, $M^{0^\dagger}\cong V'$.
	
	For a matrix $\vect{M}$, we use $\vect{M}_i^j$ to denote its $(i,j)$-entry.
	The matrix $\vect{N}_i$ is defined as $(\vect{N}_i)_j^k = \vect{N}_{i,j}^k$, where $\vect{N}_{i,j}^k$ denotes the fusion rule $\Nfusion[M^i][M^j][M^k]$.
	\begin{lemma}\label{lem:fusion-matrix-transpose}
		The transpose $\vect{N}_i^{\mathsf{t}}$ of the matrix $\vect{N}_i$ is
		$(\vect{N}_i^{\mathsf{t}})_j^k=\vect{N}_{i,j^\dagger}^{k^\dagger}$.
	\end{lemma}
	\begin{proof}
		It suffices to show $\vect{N}_{i,k}^j=\vect{N}_{i,j^\dagger}^{k^\dagger}$, which follows from
		\[\dim\cova{M^i,M^k,(M^j)'}_{(\PP^1,(0,1,\infty))}=\dim\cova{M^i,(M^j)',M^k}_{(\PP^1,(0,1,\infty))}.\qedhere\]
	\end{proof}
	\begin{lemma}\label{lem:average}
		We have
		\[
		\sum_{i\in\mathcal{I}}\vect{N}_i\vect{N}_{i^\dagger} = \sum_{l\in\mathcal{I}}\operatorname{Tr}\ab(\vect{N}_{l^\dagger})\vect{N}_{l}.
		\]
		This matrix is called the \concept{handle matrix} of $\Irr$ and is denoted
		by $\vect{H}$.
	\end{lemma}
	\begin{proof} Let $k,j\in \Set{0,1,\cdots,N}$. Then
		\begin{align*}
			(\sum_{i\in\mathcal{I}}\vect{N}_i\vect{N}_{i^\dagger})_j^k
			& = \sum_{i\in\mathcal{I}}\sum_{l\in\mathcal{I}}\vect{N}_{i,j}^{l}\vect{N}_{i^\dagger,l}^{k} = \sum_{i\in\mathcal{I}}\sum_{l\in\mathcal{I}}\vect{N}_{i^\dagger,i}^{l}\vect{N}_{l,j}^{k} & \text{by \labelcref{eq:Nassociative}} \\
			& = \sum_{l\in\mathcal{I}}\sum_{i\in\mathcal{I}}\vect{N}_{l^\dagger,i}^{i}\vect{N}_{l,j}^{k}
			& \text{by \labelcref{eq:Ncommutative} and \cref{lem:fusion-matrix-transpose}} \\
			& = (\sum_{l\in\mathcal{I}}\operatorname{Tr}\ab(\vect{N}_{l^\dagger})\vect{N}_{l})_j^k.
		\end{align*}
		This proves the lemma.
	\end{proof}
	
	The following formula is referred to as the ``Verlinde formula'' by geometers \cite{Bea96,Uen08}. However, in the VOA context, the term ``Verlinde formula'' usually refers to the statement that the $S$-matrix diagonalizes the fusion matrices \cite{V88,H08a}.
	\begin{theorem}\label{thm:Verlinde}
		Let $(X,\mpt{p}_\bullet)$ be a stable $n$-pointed curve of genus $g$, so that $2g-2+n>0$. Then 
		\begin{equation}\label{eq:Verlinde}
			\dim\cova{M^{i_\bullet}}_{(X,\mpt{p}_\bullet)} = (\vect{N}_{i_\bullet}\vect{H}^g)_0^{0^\dagger},
		\end{equation}
		where $\vect{N}_{i_\bullet}=\vect{N}_{i_1}\cdots\vect{N}_{i_n}$.
		Moreover, if $g \ge 1$, then
		\[
		(\vect{N}_{i_\bullet}\vect{H}^g)_0^{0^\dagger}
		=\operatorname{Tr}\ab(\vect{N}_{i_\bullet}\vect{H}^{g-1}).
		\]
	\end{theorem}
	
	\begin{proof}
		Propagation of vacua and the two-point pairing give
		\[
		\vect{N}_{l^\dagger,k}^{0^\dagger}
		=\dim\cova{(M^l)',M^k,V}_{(\PP^1,(0,1,\infty))}
		=\dim\Hom_V(M^k,M^l)=\delta_{l,k}.
		\]
		Together with \eqref{eq:Nidentity}, this gives the second statement:
		
		\begin{align*}
			(\vect{N}_{i_\bullet}\vect{H}^g)_0^{0^\dagger}            & =
			\sum_{l\in\mathcal{I}}(\vect{N}_{i_\bullet}\vect{H}^{g-1}\vect{N}_{l}\vect{N}_{l^\dagger})_0^{0^\dagger} & \text{by \cref{lem:average}} \\
			& = \sum_{j,k,l\in\mathcal{I}}\vect{N}_{l,0}^j(\vect{N}_{i_\bullet}\vect{H}^{g-1})_j^k\vect{N}_{l^\dagger,k}^{0^\dagger}  \\
			& = \sum_{j,k,l\in\mathcal{I}}\delta_{l,j}(\vect{N}_{i_\bullet}\vect{H}^{g-1})_j^k\delta_{l,k}    & \text{by the preceding identities} \\
			& = \sum_{l\in\mathcal{I}}(\vect{N}_{i_\bullet}\vect{H}^{g-1})_l^l
			= \operatorname{Tr}\ab(\vect{N}_{i_\bullet}\vect{H}^{g-1}).
		\end{align*}
		
		We prove the first statement by induction on $g$.
		First, for the case $g=0$, we consider the following factorization datum:
		\[
		\begin{tikzcd}
			{X_0=\PP^1\colon}
			{\begin{tikzpicture}[dot/.style={circle,draw,fill,inner sep=1pt}]
					\draw (0,0) ellipse (3 and 1);
					\node[dot,label=right:$0$] at (-3,0) {};
					\node[dot,label=$\mpt{p}_1$] at (-2,0) {};
					\node[dot,label=$\mpt{p}_2$] at (-1,0) {};
					\node at (0.5,0) {\huge\ldots};
					\node[dot,label=$\mpt{p}_n$] at (2,0) {};
					\node[dot,label=left:$\infty$] at (3,0) {};
			\end{tikzpicture}}\arrow[d,rightsquigarrow] \\
			{X_1\colon}
			{\begin{tikzpicture}[dot/.style={circle,draw,fill,inner sep=1pt}]
					\draw (-4,0) circle (1);
					\node[dot,label=right:$0$] at (-5,0) {};
					\node[dot,label=$\mpt{p}_1$] at (-4,0) {};
					\draw (-2,0) circle (1);
					\node[dot] at (-3,0) {};
					\node[dot,label=$\mpt{p}_2$] at (-2,0) {};
					\node[dot] at (-1,0) {};
					\node at (0,0) {\huge\ldots};
					\draw (2,0) circle (1);
					\node[dot] at (1,0) {};
					\node[dot,label=$\mpt{p}_n$] at (2,0) {};
					\node[dot,label=left:$\infty$] at (3,0) {};
			\end{tikzpicture}} \\
			{\widetilde{X}_1\colon}
			{\begin{tikzpicture}[dot/.style={circle,draw,fill,inner sep=1pt}]
					\draw (-4.4,0) circle (1);
					\node[dot,label=right:$0$] at (-5.4,0) {};
					\node[dot,label=$\mpt{p}_1$] at (-4.4,0) {};
					\node[dot,label=left:$\infty$] at (-3.4,0) {};
					\draw (-2,0) circle (1);
					\node[dot,label=right:$0$] at (-3,0) {};
					\node[dot,label=$\mpt{p}_2$] at (-2,0) {};
					\node[dot,label=left:$\infty$] at (-1,0) {};
					\node at (0,0) {\huge\ldots};
					\draw (2,0) circle (1);
					\node[dot,label=right:$0$] at (1,0) {};
					\node[dot,label=$\mpt{p}_n$] at (2,0) {};
					\node[dot,label=left:$\infty$] at (3,0) {};
			\end{tikzpicture}} \arrow[u]
		\end{tikzcd}
		\]
		Then, by \nameref{coro:VB} and \nameref{thm:factorization}, we have
		\begin{align*}
			\MoveEqLeft[1]
			\dim\cova{V,M^{i_\bullet},V}_{(\PP^1,(0,\mpt{p}_\bullet,\infty))}  \\
			& =
			\sum_{j_1,\cdots,j_{n-1}\in\mathcal{I}}
			\prod_{\ast=1,\cdots,n}
			\dim\cova{M^{j_{\ast-1}},M^{i_\ast},(M^{j_{\ast}})'}_{(\PP^1,(0,\mpt{p}_\ast,\infty))} \\
			& =
			\sum_{j_1,\cdots,j_{n-1}\in\mathcal{I}}
			\vect{N}_{j_{0},i_{1}}^{j_{1}}
			\vect{N}_{j_{1},i_{2}}^{j_{2}}
			\cdots
			\vect{N}_{j_{n-1},i_{n}}^{j_{n}}  \\
			& = (\vect{N}_{i_1}\vect{N}_{i_2}\cdots\vect{N}_{i_n})_{j_{0}}^{j_{n}},
		\end{align*}
		where $j_0=0$ and $j_n=0^\dagger$, since $(M^{0^\dagger})'\cong V$.
		On the other hand, by \nameref{thm:propagation-of-vacua},
		\[
		\dim\cova{M^{i_\bullet}}_{(\PP^1,\mpt{p}_\bullet)} = \dim\cova{V,M^{i_\bullet},V}_{(\PP^1,(0,\mpt{p}_\bullet,\infty))}.
		\]
		Therefore, $\dim\cova{M^{i_\bullet}}_{(\PP^1,\mpt{p}_\bullet)} = (\vect{N}_{i_\bullet})_0^{0^\dagger}$.
		
		Next, assume the first statement holds for every stable $m$-pointed curve of genus $g$, and let $(X,\mpt{p}_\bullet)$ be a stable $n$-pointed curve of genus $g+1$. Consider the following factorization datum:
		\[
		\begin{tikzcd}
			{X_0\colon}
			{\begin{tikzpicture}[dot/.style={circle,draw,fill,scale=0.9,inner sep=1pt}]
					\draw (0,0) ellipse (3 and 1);
					\node[dot,label=$\mpt{p}_1$] at (-2.2,0) {};
					\node[dot,label=$\mpt{p}_2$] at (-1.2,0.22) {};
					\node[dot,label=$\mpt{p}_n$] at (2,0.2) {};
					\draw[bend left=45] (-1.95,-0.25) to (-0.85,-0.25); \draw[bend right=45] (-2,-0.2) to (-0.8,-0.2);
					\draw[bend left=45] (-0.57,0.17) to (-0.03,0.27); \draw[bend right=45] (-0.6,0.2) to (0,0.3);
					\node at (0.6,0) {\huge\ldots};
					\draw[bend left=45] (1.3,-0.1) to (2,-0.1); \draw[bend right=45] (1.25,-0.08) to (2.05,-0.08);
					\draw[bend left=45] (-0.55,-0.7) to (0.55,-0.7); \draw[bend right=45] (-0.6,-0.65) to (0.6,-0.65);
			\end{tikzpicture}}\arrow[d,rightsquigarrow] \\
			{X_1\colon}
			{\begin{tikzpicture}[dot/.style={circle,draw,fill,inner sep=0.9pt}]
					\draw (0,0) ellipse (3 and 1);
					\node[dot,label=$\mpt{p}_1$] at (-2.2,0) {};
					\node[dot,label=$\mpt{p}_2$] at (-1.2,0.22) {};
					\node[dot,label=$\mpt{p}_n$] at (2,0.2) {};
					\draw[bend left=45] (-1.95,-0.25) to (-0.85,-0.25); \draw[bend right=45] (-2,-0.2) to (-0.8,-0.2);
					\draw[bend left=45] (-0.57,0.17) to (-0.03,0.27); \draw[bend right=45] (-0.6,0.2) to (0,0.3);
					\node at (0.6,0) {\huge\ldots};
					\draw[bend left=45] (1.3,-0.1) to (2,-0.1); \draw[bend right=45] (1.25,-0.08) to (2.05,-0.08);
					\draw[bend left=45] (-0.6,-0.7) to (0.6,-0.7); \draw[bend right=60] (-0.65,-0.65) to coordinate (m) (0.65,-0.65);
					\node[dot,label=$\mpt{q}$] at (m) {};
			\end{tikzpicture}} \\
			{\widetilde{X}_1\colon}
			{\begin{tikzpicture}[dot/.style={circle,draw,fill,inner sep=0.9pt}]
					\draw (1,-1) arc (-70:250:3 and 1);
					\node[dot,label=$\mpt{p}_1$] at (-2.2,0) {};
					\node[dot,label=$\mpt{p}_2$] at (-1.2,0.22) {};
					\node[dot,label=$\mpt{p}_n$] at (2,0.13) {};
					\draw[bend left=45] (-1.95,-0.25) to (-0.85,-0.25); \draw[bend right=45] (-2,-0.2) to (-0.8,-0.2);
					\draw[bend left=45] (-0.57,0.17) to (-0.03,0.27); \draw[bend right=45] (-0.6,0.2) to (0,0.3);
					\node at (0.6,0) {\huge\ldots};
					\draw[bend left=45] (1.3,-0.1) to (2,-0.1); \draw[bend right=45] (1.25,-0.08) to (2.05,-0.08);
					\node[dot,label=$\mpt{q}_+$] at (-1,-1) {};
					\node[dot,label=$\mpt{q}_-$] at (1,-1) {};
					\draw (-1,-1) to[in=180,out=0] (0,-0.6) to[in=180,out=0] (1,-1);
			\end{tikzpicture}} \arrow[u]
		\end{tikzcd}
		\]
		Then, by \nameref{coro:VB} and \nameref{thm:factorization}, we have
		\begin{align*}
			\MoveEqLeft[1]
			\dim\cova{M^{i_\bullet}}_{(X_0,\mpt{p}_\bullet)} \\
			& =\sum_{j\in\mathcal{I}}\dim\cova{M^{i_\bullet},M^j,(M^j)'}_{(\widetilde{X}_1,\mpt{p}_\bullet\sqcup\mpt{q}_\pm)} \\
			& =\sum_{j\in\mathcal{I}}(\vect{N}_{i_\bullet}\vect{H}^{g}\vect{N}_{j}\vect{N}_{j^\dagger})_0^{0^\dagger}
			& \text{inductive hypothesis} \\
			& =(\vect{N}_{i_\bullet}\vect{H}^{g+1})_0^{0^\dagger}. & \text{by the handle-matrix identity}
		\end{align*}
		This completes the proof.
	\end{proof}
	
	For a single insertion $M^i$ in genus one, \cref{thm:Verlinde} specializes to
	\[
	\dim\mathscr{C}(E_\tau,\mpt{p},z,M^i)
	=\operatorname{Tr}\ab(\vect{N}_i)
	=\sum_{k\in\mathcal{I}}\Nfusion[M^i][M^k][M^k].
	\]
	For the reader's convenience, we provide the factorization datum for this case:
	\[
	\begin{tikzpicture}[dot/.style={circle,draw,fill,inner sep=1pt},decoration=snake]
		\coordinate (A) at (-0.5,0);
		\draw (A) ellipse (1.5 and 0.75);
		\node[dot,label=$\mpt{p}$] at (-1.4,0) {};
		\node[label=below:{$(E_\tau,\mpt{p})$}] at ($(A)+(0,-1.1)$) {};
		\draw[bend left=45] (-1,-0.23) to (0,-0.23);
		\draw[bend right=45] (-1.1,-0.15) to (0.1,-0.15);
		\path[draw,decorate,->] (1.05,0) -- (2.4,0);
		\coordinate (B) at (4,0);
		\draw (B) ellipse (1.5 and 0.75);
		\node[dot,label=$\mpt{p}$] at (3.1,0) {};
		\node[label=below:{$(E_\infty,\mpt{p})$}] at ($(B)+(-0.2,-1.1)$) {};
		\draw[bend left=45] (3.3,-0.46) to (4.7,-0.46);
		\draw[bend right=60] (3.2,-0.34) to coordinate (mB) (4.8,-0.34);
		\node[dot,label=$\mpt{q}$] at (mB) {};
		\path[draw,<-] (5.7,0) -- (6.6,0);
		\coordinate (C) at (8,0);
		\draw (C) circle (1.1);
		\node[dot,label=$1$] at ($(C)+(-0.2,0.4)$) {};
		\node[dot,label=below:$0$] at ($(C)+(0.6,-0.3)$) {};
		\node[dot,label=below:$\infty$] at ($(C)+(-0.5,-0.3)$) {};
		\node[label=below:{$(\PP^1,1,\infty,0)$}] at ($(C)+(0,-1)$) {};
	\end{tikzpicture}
	\]
	Here $E_\tau$ denotes the smooth elliptic curve; see \cref{subsec:square-bracket-voa-and-modules} for its two equivalent descriptions. 
	The above dimension count is used in Step~4 of the proof of \cref{thm:basiscfb}.
	
	\section{One-point conformal blocks on elliptic curves}\label{sec:cfbontorus}
	
	In this section, we specialize the general theory to genus one, where elliptic functions and their formal $q$-expansions make the conformal-block relations explicit. Throughout this section, $V$ is a VOA of CFT type.

	\subsection{Elliptic functions}
	
	First, we recall the notation and properties of elliptic functions following \cite{Zhu96}.
	Recall the following power series from \S 3 loc.~cit.:
	\begin{equation}\label{eq:defPm+1}
		P_{m+1}(x,q) = 
		\frac{(2\pi\iu)^{m+1}}{m!} \sum_{k=1}^\infty \ab( \frac{k^mx^k}{1-q^k}+\frac{(-1)^{m+1}k^mx^{-k}q^k}{1-q^k} ),
	\end{equation}
	which converges locally uniformly in the domain
	$\Set*{ (x,q) \in \C^2 \given |q|<|x|<1 }$.
	
	\begin{lemma}
		Let $q=\e^{2\pi \iu \tau}$. 
		The functions $P_m(\e^{2\pi \iu z},q)$ defined by Zhu's convergent series satisfy
		\begin{align*}
			P_1(\e^{2\pi \iu z},q) 
			& = -\hat\wp_1(z,\tau)-\pi \iu ,  \\
			P_m(\e^{2\pi \iu z},q) 
			& = (-1)^m\hat\wp_m(z,\tau),\quad m \ge 2,
		\end{align*}
		where $\hat\wp_m(z,\tau)$ are the modified Weierstrass $\wp$-functions defined below.
	\end{lemma}
	
	We now spell out the modified Weierstrass functions appearing in this comparison.
	Let $\H$ denote the upper half-plane.
	For any $\tau\in \H$, we write $\zeta(z,\tau)$ for the \concept{Weierstrass zeta function} associated to the lattice 
	$\Lambda_\tau := \Z + \Z\tau$:
	\[
	\zeta(z,\tau) :=
	\frac{1}{z}
	+
	\sum_{ \omega \in \Lambda_\tau\setminus\Set{0} }
	\ab(
	\frac{1}{z-\omega} + \frac{1}{\omega} + \frac{z}{\omega^2}
	).
	\]
	Its negative derivative $\wp^{\mathrm{cl}}(z,\tau) := -\pdv{z}\zeta(z,\tau)$ is the \concept{classical Weierstrass elliptic function}.
	The Eisenstein series $G_{2k}(\tau)$ is defined as the sum of $1/\omega^{2k}$ taken over all $\omega \in \Lambda_\tau\setminus\Set{0}$ if $k \ge 2$, and when $k = 1$, the series is defined via a specific summation order: $G_2(\tau)$ the sum of two series of $1/\omega^2$, one is taking over all nonzero real points in $\Lambda_\tau$ and the other is taking over all non-real points in $\Lambda_\tau$; see \cite{Apo90}.
	
	Next, we introduce the family $\Fam{\hat{\wp}_m}{m\ge 1}$ of \concept{modified Weierstrass $\wp$-functions}:
	\begin{equation}\label{eq:modifiedwp}
		\begin{aligned}
			\hat{\wp}_1(z,\tau)&:=\zeta(z,\tau)-G_2(\tau)z,\\
			\hat{\wp}_{m+1}(z,\tau)&:=-\frac{1}{m}
			\pdv{z}\hat{\wp}_m(z,\tau),\quad m\ge 1.
		\end{aligned}
	\end{equation}
	Comparing them with the Weierstrass $\wp$-functions used in \cite{Zhu96}, we have
	\begin{gather*}
		\hat{\wp}_1(z,\tau)=\wp_1(z,\tau)-G_2(\tau)z,\qquad
		\hat{\wp}_2(z,\tau)=\wp_2(z,\tau)+G_2(\tau),\\
		\text{and} \qquad
		\hat{\wp}_m(z,\tau)=\wp_m(z,\tau)\quad(m \ge 3).
	\end{gather*}
	In the rest of this paper, we use only the \concept{modified Weierstrass $\wp$-functions}. \emph{To lighten the notation, we suppress the hats thereafter.}

	Note that $\wp_m(z,\tau)$ are doubly periodic meromorphic functions in $z$ with periods $1$ and $\tau$ for all $m\ge 2$.   But $\wp_1(z,\tau)$ is only quasi-doubly periodic:
	\begin{equation}\label{eq:quasiwp1}
		\wp_1(z+1,\tau)=\wp_1(z,\tau),\quad \wp_1(z+\tau,\tau)=\wp_1(z,\tau)-2\pi \iu .
	\end{equation}
	Furthermore, the functions $\wp_m$ have the following Laurent series expansion around $z=0$:
	\begin{equation}\label{expansionofwpm}
		\iota_z \wp_m(z,\tau) =
		\frac{1}{z^m}+(-1)^m\sum_{k=1}^\infty \binom{2k-1}{m-1} G_{2k}(\tau) z^{2k-m},\quad m\ge 1,
	\end{equation}
	where the binomial coefficients are zero for $2k < m$.
	In particular, $\wp_1(z,\tau)$ has the following expansion:
	\begin{equation}\label{eq:Laurentwp1}
		\iota_z \wp_1(z,\tau) = z^{-1}-\sum_{\ell=1}^\infty G_{2\ell}(\tau)z^{2\ell-1}.
	\end{equation}
	
	Let $E_\tau$ be the smooth projective elliptic curve over $\C$ defined by the homogeneous equation
	\[
	Y^2Z=4X^3-g_2(\tau)XZ^2-g_3(\tau)Z^3,
	\]
	where $g_2(\tau)=60G_4(\tau)$ and $g_3(\tau)=140G_6(\tau)$. 
	The \concept{Weierstrass uniformization} identifies $E_\tau$ with the quotient torus $\C/\Lambda_\tau$ by
	\[
	z\longmapsto
	\ab[
	\wp^{\mathrm{cl}}(z,\tau) : \pdv{z}\wp^{\mathrm{cl}}(z,\tau) : 1
	],
	\]
	where $z$ is a local coordinate around the marked point $\mpt{p}=[0]$. 
	We thereafter do not distinguish between the algebraic elliptic curve $E_\tau$ and the analytic torus $\C/\Lambda_\tau$. 
	For an algebraic sheaf $\mathcal{F}$, write $\mathcal{F}^{\mathrm{an}}$ for the corresponding analytic sheaf, obtained by passing to the analytic topology and extending scalars from regular to holomorphic functions; unadorned sheaves remain algebraic. For instance, $\mathcal{O}_{E_\tau}^{\mathrm{an}}$ and $\omega_{E_\tau}^{\mathrm{an}}$ denote the analytic structure and canonical sheaves. 
	Write
	\[
	\mathcal{O}_{E_\tau}^{\mathrm{an}}(\ast\mpt{p}) :=
	\varinjlim_{N\ge 0}\mathcal{O}_{E_\tau}^{\mathrm{an}}(N\mpt{p})
	\]
	for the sheaf of meromorphic functions with poles of finite order only at $\mpt{p}$.
	Since $\wp_m(z,\tau)$ has a pole of exact order $m$ at $\mpt{p}$, Riemann--Roch gives us
	\begin{equation}\label{eq:elliptic-function-space-description}
		H^0\ab(E_\tau,\mathcal{O}_{E_\tau}^{\mathrm{an}}(\ast\mpt{p}))
		= \operatorname{span}\Set*{1,\wp_m(z,\tau) \given m \ge 2}
		= \C[\wp_2(z,\tau),\wp_3(z,\tau)].
	\end{equation}
	
	The (quasi-)periodicity of the $\wp_m$ implies the following properties of $P_{m+1}(z,q)$, which we use in the next section.
	\begin{corollary}\label{coro:Pfunctioncov}
		We have
		\begin{enumerate}
			\item\label{coro:Pfunctioncov:P1-convergence}
			The power series $P_1(x,q)$ converges for $|q|<|x|<1$, while $P_1(qx,q)-2\pi \iu $ converges for $1<|x|<|q|^{-1}$.
			On their respective convergence domains, these two series are expansions of the same meromorphic function $\widetilde{P}_1(x,q)$.
			\item\label{coro:Pfunctioncov:Pm-convergence}
			For $m\ge 1$, the power series $P_{m+1}(x,q)$ converges for $|q|<|x|<1$, while the series $P_{m+1}(qx,q)$ converges for $1<|x|<|q|^{-1}$.
			On their respective convergence domains, these two series are expansions of the same meromorphic function $\widetilde{P}_{m+1}(x,q)$.
		\end{enumerate}
	\end{corollary}
	
	We next record two further properties of the modified Weierstrass $\wp$-functions for later use. The derivative identity below enters the construction of the connection in \S~\ref{subsec:conformal-block-bundle-and-connection}, while the quasi-modular transformation of $\wp_1$ enters the modular-invariance calculation in \S~\ref{sec:modular-invariance-monodromy}.
	\begin{lemma}\label{lem:wpproperty}
		For any $m\ge 2$,
		\[
		\frac{1}{2\pi \iu}\pdv{\tau}\wp_m(z,\tau) - \frac{1}{(2\pi \iu)^2}\pdv{z}\ab(\wp_m(z,\tau)\wp_1(z,\tau)) \in H^0\ab(E_\tau,\mathcal{O}_{E_\tau}^{\mathrm{an}}(\ast\mpt{p})).
		\]
	\end{lemma}
	\begin{proof}
		Denote the left-hand side by $F_m(z,\tau)$. As a function of $z$, $F_m(z,\tau)$ is meromorphic, with its only possible pole at $z=0$ in the fundamental domain of the lattice $\Lambda_\tau$.
		It remains to show that $F_m(z,\tau)$ is doubly periodic with periods $1$ and $\tau$.
		
		Since $\wp_m(z+1,\tau)=\wp_m(z,\tau)$, differentiating with respect to $\tau$ gives
		\[
		\frac{\partial}{\partial\tau}(\wp_m(z+1,\tau)) = \frac{\partial}{\partial\tau}(\wp_m(z,\tau)).
		\]
		Moreover, $\pdv{z}(\wp_m(z,\tau)\wp_1(z,\tau))$ has period $1$ in $z$. Thus, $F_m(z+1,\tau)=F_m(z,\tau)$.
		On the other hand, since $\wp_m(z+\tau,\tau)=\wp_m(z,\tau)$, the chain rule gives
		\[
		\pdv{z}(\wp_m(z+\tau,\tau)) + \frac{\partial}{\partial\tau}(\wp_m(z+\tau,\tau)) = \frac{\partial}{\partial\tau}(\wp_m(z,\tau)).
		\]
		Using quasi-periodicity \eqref{eq:quasiwp1}, we have
		\begin{gather*}
			\frac{1}{2\pi \iu}\pdv{\tau}(\wp_m(z+\tau,\tau))-\frac{1}{2\pi \iu}\pdv{\tau}(\wp_m(z,\tau)) 
			= -\frac{1}{2\pi \iu } \frac{\partial }{\partial z}(\wp_m(z+\tau,\tau))  \\
			=\frac{1}{(2\pi \iu )^2} \pdv{z} (\wp_m(z+\tau,\tau)\wp_1(z+\tau,\tau))-\frac{1}{(2\pi \iu )^2} \pdv{z} (\wp_m(z,\tau)\wp_1(z,\tau)).
		\end{gather*}
		Therefore, $F_m(z+\tau,\tau)=F_m(z,\tau)$ and hence
		$F_m(z,\tau)\in H^0\ab(E_\tau,\mathcal{O}_{E_\tau}^{\mathrm{an}}(\ast\mpt{p}))$.
	\end{proof}
	
	Moreover, $\wp_1(z,\tau)$ has the following quasi-modular property:
	
	\begin{lemma}\label{lem:wp1modular}
		Let 
		$\gamma=\twobytwo{a}{b}{c}{d}\in \SL(2,\Z)$.
		Then
		\[
		\wp_1\ab((c\tau+d)^{-1}z,\gamma.\tau) = (c\tau+d) \cdot \wp_1(z,\tau)+2\pi \iu cz.
		\]
	\end{lemma}
	
	\subsection{The square-bracket VOA and its modules}\label{subsec:square-bracket-voa-and-modules}

	For a fixed $\tau\in\H$, let $E_\tau=\C/\Lambda_\tau$, and let $\mpt{p}=[0]\in E_\tau$ with local coordinate $z$. The pointed elliptic curve $(E_\tau,\mpt{p})$ can also be obtained by identifying $0$ and $\infty$ on $\PP^1$ and smoothing the resulting node. For $0<|q|<1$, this gives the multiplicative uniformization
	\[
	E_\tau\simeq\C^\times/\ab(x\sim q^n x)_{n\in\Z},
	\qquad
	[z]\longmapsto[\e^{2\pi\iu z}],
	\qquad
	q=\e^{2\pi\iu\tau},
	\]
	which sends $\mpt{p}=[0]$ to $[1]$. If $x$ denotes the multiplicative coordinate on $\C^\times$, then $w=x-1$ is a local coordinate at $[1]$, and $w=\e^{2\pi\iu z}-1 = 2\pi\iu z + O(z^2)$; see \cite[\S 4.2]{Zhu96}.
	
	To match the resulting trace-function formulas, we pass to the square-bracket VOA and the corresponding module structures.
	Write
	\[
	V^\square=(V,Y[\placeholder,z],\vac,\widetilde{\cfv}),
	\]
	where
	\[
	Y[\placeholder,z] = Y\ab(\e^{2\pi\iu zL(0)}\placeholder,\e^{2\pi\iu z}-1),\qquad
	\widetilde{\cfv} = (2\pi\iu)^2\ab(\cfv-\frac{\ccharge}{24}\vac).
	\]
	This square-bracket VOA is isomorphic to the original $V$, see \cite[\S~7.4]{H97} and \cite[\S~4]{Lep99}.
	Explicitly, the isomorphism is given by the operator
	\[
	\mathcal{U}_V := \exp\ab(-\sum_{j\ge1}b_jL(j))(2\pi\iu)^{L(0)},
	\]
	where the numbers $b_j\in\Q$ are determined by
	\[
	\exp\ab(\sum_{j\ge1}b_jx^{j+1}\partial_x)x=\log(1+x).
	\]
	The coordinate-change identity gives
	\[
	\mathcal{U}_V\vac = \vac,\qquad \mathcal{U}_V\cfv = \widetilde{\cfv},
	\qquad\text{and}\qquad
	Y[\mathcal{U}_Va,z] = \mathcal{U}_VY(a,z)\mathcal{U}_V^{-1}.
	\]
	Equivalently,
	\[
	\mathcal{U}_V=(2\pi\iu)^{L(0)}\exp\ab(-\sum_{j\ge1}(2\pi\iu)^jb_jL(j)).
	\]
	This is the operator $\mathcal{U}(1)$ in \cite{H05b}.
	Note that $\mathcal{U}_V\cfv = \widetilde{\cfv}$ implies that $\mathcal{U}_V$ preserves the gradings and central charges. Thus simplicity, rationality, $C_2$-cofiniteness, and CFT type are all preserved under this isomorphism.
	
	For a weak $V$-module $(W,Y_W(\placeholder,z))$, we write $W^\square$ for the same vector space with
	\[
	Y_W[\placeholder,z] = Y_W\ab(\e^{2\pi\iu zL(0)}\placeholder,\e^{2\pi\iu z}-1).
	\]
	Write $Y_W[a,z]=\sum_{n\in \Z} a[n] z^{-n-1}$. Then
	\begin{equation}\label{eq:defa[m]}
		a[m]=(2\pi \iu )^{-m-1} \Res_z \ab(Y_W(a,z)\ab(\ln(1+z))^m (1+z)^{\wt a-1}).
	\end{equation}
	In particular, for the Virasoro element $\widetilde{\cfv}$, we write
	\[
	L[n]=\widetilde{\cfv}[n+1]=\Res_z z^{n+1} Y_W[\widetilde{\cfv},z],
	\qquad n \in \Z.
	\]
	Then $W^\square$ is a weak $V^\square$-module; see \cite[\S~7.4]{H97}.
	If $W$ is admissible of conformal weight $h_W$, put
	\[
	\mathcal{U}_W=\exp\ab(-\sum_{j\ge1}b_jL_W(j))(2\pi\iu)^{L_W(0)},
	\]
	with the convention
	\begin{equation}\label{eq:2piipower}
		(2\pi\iu)^{h} = \exp\ab(\ab\Big(\log(2\pi)+\frac{\pi\iu}{2})h).
	\end{equation}
	The coordinate-change formula \cite[Proposition~1.2]{H05b}, together with $\mathcal{U}_V\cfv=\widetilde{\cfv}$, gives
	\begin{equation}\label{eq:square-module-conjugation}
		Y_W[\placeholder,z] = \mathcal{U}_WY_W(\mathcal{U}_V^{-1}\placeholder,z)\mathcal{U}_W^{-1},
		\qquad L_[n] = \mathcal{U}_WL_W(n)\mathcal{U}_W^{-1}\quad(n\in\Z).
	\end{equation}

	For ordinary $V$-modules $M^1,M^2,M^3$ and $I\in\mathbf{I}_V\fusion[M^1][M^2][M^3]$, set
	\begin{equation}\label{eq:square-intertwiner}
		I^\square(\placeholder,z)
		:=\mathcal{U}_{M^3}I(\mathcal{U}_{M^1}^{-1}\placeholder,z)\mathcal{U}_{M^2}^{-1}.
	\end{equation}
	In explicit, the coordinate-change formula gives the following expression
	\[
	I^\square(\placeholder,z)=
	I\ab(\e^{2\pi\iu zL_{M^1}(0)}\placeholder,\e^{2\pi\iu z}-1),
	\]
	where, for $r\in\C$, we keep the convention \eqref{eq:2piipower} and set
	\[
	(\e^{2\pi\iu z}-1)^r
	:=(2\pi\iu)^rz^r\sum_{m\ge0}\binom{r}{m}
	\ab(\frac{\e^{2\pi\iu z}-1}{2\pi\iu z}-1)^m.
	\]
	For original homogeneous $u\in M^1$, this gives $I(u,\e^{2\pi\iu z}-1)\e^{2\pi\iu z\wt u}$, essentially the same formula of \cite[Theorem~2.6, (2.10)]{Mi00}.
	
	\begin{proposition}[Square-bracket transport]\label{prop:square-transport}
		Let $W$ be an admissible $V$-module of conformal weight $h_W$ and $M^1,M^2,M^3$ be ordinary $V$-modules.
		\begin{enumerate}
			\item\label{prop:square-transport:module-grading}
			The module $W^\square$ is admissible of conformal weight $h_W$. Its grading components satisfy
			\[
			W[n]:=\Set*{u\in W\given L_[0]u=(h_W+n)u}=\mathcal{U}_W W(n),
			\]
			and the two gradings induce the same filtration:
			\[
			W(0)\oplus\cdots\oplus W(n)=W[0]\oplus\cdots\oplus W[n].
			\]
			In particular, $W$ is ordinary if and only if $W^\square$ is ordinary.
			\item\label{prop:square-transport:cofiniteness}
			Irreducibility, finite generation, and $C_2$-cofiniteness are preserved under $W\mapsto W^\square$.
			\item\label{prop:square-transport:contragredient}
			If $W$ is ordinary, then $(W')^\square\cong(W^\square)'$ as $V^\square$-modules.
			\item\label{prop:square-transport:intertwining-operators}
			The map $I\mapsto I^\square$ gives a linear isomorphism
			\[
			\mathbf{I}_V\fusion[M^1][M^2][M^3]
			\isomto
			\mathbf{I}_{V^\square}\fusion[(M^1)^\square][(M^2)^\square][(M^3)^\square].
			\]
		\end{enumerate}
	\end{proposition}
	\begin{proof}
		For \itemonlycref{prop:square-transport:module-grading}, the second identity in \eqref{eq:square-module-conjugation} gives $W[n]=\mathcal{U}_W W(n)$. Then the commutator formula makes $W^\square$ an admissible $V^\square$-module. Put $W_{\le n}=W(0)\oplus\cdots\oplus W(n)$. Since positive Virasoro modes lower the original degree, both $\mathcal{U}_W$ and its inverse preserve $W_{\le n}$; hence
		\[
		W[0]\oplus\cdots\oplus W[n]
		=\mathcal{U}_W W_{\le n}=W_{\le n}.
		\]
		In particular, $W[0]=W(0)\neq0$, so the conformal weight is preserved. The equality $\dim W[n]=\dim W(n)$ proves the claim about ordinariness.
		
		For \itemonlycref{prop:square-transport:cofiniteness}, the first identity in \eqref{eq:square-module-conjugation} identifies submodules and generating sets via $\mathcal{U}_W$, and hence irreducibility and finite generation are preserved. Taking its $(-2)$-mode gives, for $a\in V$ and $u\in W$,
		\[
		a[-2]\mathcal{U}_Wv = \mathcal{U}_W\ab((\mathcal{U}_V^{-1}a)_{(-2)}u).
		\]
		Since $\mathcal{U}_V$ and $\mathcal{U}_W$ are bijective, varying $a$ and $u$ yields
		\[
		\mathcal{U}_W C_2(W) = C_2(W^\square).
		\]
		Thus $C_2$-cofiniteness is preserved as well.
		
		For \itemonlycref{prop:square-transport:contragredient}, since $\mathcal{U}_V$ is a VOA isomorphism, the assignment of the contragredient to a module is preserved by the conjugation identity. Therefore, the contrgredient action on $(W^\square)'$ is identified with the action on $(W')^\square$. Furthermore, by \itemonlycref{prop:square-transport:module-grading}, it also identifies the finite-dimensional grading pieces used in the restricted duals.
		
		For \itemonlycref{prop:square-transport:intertwining-operators}, the conjugation identities carry the axioms for the intertwining operators from $I$ to the operator $I^\square$ defined in \cref{eq:square-intertwiner}. This process is invertible since $\mathcal{U}_{M^i}$ ($i=1,2,3$) are all bijective.
	\end{proof}
	
	\subsection{Chiral Lie algebras and one-point conformal blocks}\label{sec:elliptic-cfb}
	
	We now specialize the conformal-block construction of \S~\ref{sec:prelim} to one-pointed elliptic curves.
	First, we identify the various sheaves in \S~\ref{sec:prelim} with original and square-bracket inputs.
	\begin{proposition}[Comparison with the square-bracket realization]\label{prop:original-square-realizations}
		The VOA isomorphism $\mathcal{U}_V\colon V\to V^\square$ identifies the corresponding chiral Lie algebra sheaves $\mathcal{L}(V)$ and $\mathcal{L}(V^\square)$. For inserted admissible modules $M^i$ of conformal weight, the maps $\mathcal{U}_{M^i}$ identify the corresponding coinvariant and conformal-block sheaves.
	\end{proposition}
	\begin{proof}
		On a pointed elliptic curve, the identification is
		\[
		\mathcal{L}_{E_\tau\setminus\mpt{p}}(V)
		\isomto\mathcal{L}_{E_\tau\setminus\mpt{p}}(V^\square),
		\qquad [a\otimes f(z)]\longmapsto[\mathcal{U}_Va\otimes f(z)].
		\]
		It is well-defined since $\mathcal{U}_VV_k=V_{[k]}$ and $L[-1]\mathcal{U}_V=\mathcal{U}_VL(-1)$. As $\mathcal{U}_V$ is a VOA isomorphism, these maps respect Lie brackets and coordinate changes, giving the sheaf isomorphism. The identity $Y_{M^i}[\mathcal{U}_Va,z]\mathcal{U}_{M^i}=\mathcal{U}_{M^i}Y_{M^i}(a,z)$ from \eqref{eq:square-module-conjugation} shows that $\bigotimes_i\mathcal{U}_{M^i}$ carries the original chiral-action onto the square-bracket one. It therefore identifies the coinvariant sheaves and hence the conformal-block bundles.
	\end{proof}
	
	\emph{In what follows, we keep using the square-bracket realizations of the sheaves of chiral Lie algebras, coinvariants, and conformal blocks unless otherwise stated.}
	
	Next, we pull the genus-one chiral Lie algebra sheaf back to $\H$ and identify its fibers. Let $\overline{\mathcal{M}}_{1,1}$ denote the moduli stack of stable one-pointed genus-one curves, and let $j \colon \mathcal{M}_{1,1} \hookrightarrow \overline{\mathcal{M}}_{1,1}$ be the open substack of smooth curves.
	We write $\mathscr{L}(V)$ for the sheaf of chiral Lie algebras of $V$ on $\overline{\mathcal{M}}_{1,1}$, as in \S~\ref{subsec:chiral}.
	Analytification commutes with restrictions, so we can write
	\[
	\mathscr{L}_{\mathcal{M}^{\mathrm{an}}_{1,1}}(V) :=
	\ab(j^{\ast}\mathscr{L}(V))^{\mathrm{an}} =
	(j^{\mathrm{an}})^{\ast}\ab( \mathscr{L}(V)^{\mathrm{an}}),
	\]
	where $j^{\mathrm{an}}$ is the analytification of $j$.
	Recall that the analytification of $\mathcal{M}_{1,1}$ has the following quotient stack description \cite[Definition~3.12]{Ha21}:
	\[
	\mathcal{M}_{1,1}^{\mathrm{an}} \simeq \ab[\SL(2,\Z)\dblbackslash\H].
	\]
	Hence, we can pull back $\mathscr{L}_{\mathcal{M}^{\mathrm{an}}_{1,1}}(V)$ to obtain a $\SL(2,\Z)$-equivariant sheaf $\mathscr{L}_{\H}(V)$ on $\H$.
	
	More explicitly, fiberwise we have:
	\begin{lemma}\label{lem:analyticChiral}
		For any $\tau\in \H$, the fiber of $\mathscr{L}_{\H}(V)$ at $\tau$ identifies with the chiral Lie algebra for the elliptic curve $E_\tau$ with marked point $\mpt{p}=\ab[0]$, namely:
		\[
		\ab(\mathscr{L}_{\H}(V))_\tau
		\cong
		\mathcal{L}_{E_\tau\setminus\mpt{p}}(V)
		=
		H^0\ab(
		E_\tau, 
		\mathcal{V}_{E_\tau}\otimes\omega_{E_\tau}(\ast \mpt{p})\middle/\Image{\nabla}
		),
		\]
		where $\omega_{E_\tau}^{1-k}(\ast \mpt{p}) := \varinjlim_{N \ge 0} \omega_{E_\tau}^{1-k}(N\mpt{p})$ is the sheaf of rational $(1-k)$-forms with poles only at $\mpt{p}$.
	\end{lemma}
	\begin{proof}
		By the description following \eqref{eq:familychiralLie}, the algebraic fiber of $j^{\ast}\mathscr{L}(V)$ at $[E_\tau,\mpt{p}]$ is $\mathcal{L}_{E_\tau\setminus\mpt{p}}(V)$. 
		For every $k,N\ge0$, GAGA \cite{S56} identifies the spaces of sections of algebraic and analytic versions of the sheaf $\omega_{E_\tau}^{1-k}(N\mpt{p})$. These identifications commute with $\nabla$. Then Passing to the filtered limit in $N$, the direct sum over $k$, and the quotient by $\Image\nabla$ proves that analytification preserves this fiber.
		Finally, pullback along $\pi$ gives the statement.
	\end{proof}
	
	Let $W$ be an admissible $V$-module of conformal weight $h_W$, and let $\mathscr{C}(W)$ be its genus-one one-point conformal-block sheaf on $\overline{\mathcal{M}}_{1,1}$ from \cref{subsec:FTVB}. 
	When $W$ is finitely generated, $\mathscr{C}(W)$ is a finite-rank vector bundle by \cref{coro:VB}. To describe its fibers explicitly, we use Weierstrass functions to compute the chiral Lie algebra action on $W^\square$.
	The GAGA comparison in \cref{lem:analyticChiral} together with $\omega_{E_\tau}^{\mathrm{an}}\simeq\mathcal{O}_{E_\tau}^{\mathrm{an}}$ and \eqref{eq:elliptic-function-space-description} gives, for every $k \ge 0$,
	\begin{equation}\label{eq:H0sectionidentification}
		\begin{aligned}
			H^0\ab(E_\tau,\omega_{E_\tau}^{1-k}(\ast\mpt{p}))
			&\cong H^0\ab(E_\tau,
			\mathcal{O}_{E_\tau}^{\mathrm{an}}(\ast\mpt{p})) \\
			&= \operatorname{span}\Set*{1,\wp_m(z,\tau)\given m \ge 2}
			= \C[\wp_2(z,\tau),\wp_3(z,\tau)].
		\end{aligned}
	\end{equation}
	Passing to the quotient by $\Image\nabla$, one finds that $\mathcal{L}_{E_\tau\setminus\mpt{p}}(V)$ is spanned by the classes of
	\[
	a\otimes 1, \quad a\otimes \wp_{m}(z,\tau),\quad a\in V,\ m\ge 2.
	\]
	In the coordinate $z$, the local action is
	\[
	\rho_{z} \colon \mathcal{L}_{E_\tau\setminus \mpt{p}}(V)\longrightarrow \mathfrak{L}_{z} (V)^{\mathsf{L}}\longrightarrow \mathfrak{gl}(W),
	\]
	and on the spanning classes above it is given by
	\begin{equation}\label{eq:defspanningaction}
		\begin{aligned}
			(a\otimes 1).u
			&:= \Res_{z} Y_W[a,z]u=a[0]u, \\
			(a\otimes \wp_m(z,\tau)).u
			&:= \Res_z Y_W[a,z]\iota_z(\wp_m(z,\tau))u \\
			&\,= a[-m]u+(-1)^m\sum_{k=1}^\infty \binom{2k-1}{m-1}G_{2k}(\tau) a[2k-m]u,\quad m \ge 2,
		\end{aligned}
	\end{equation}
	where $a\in V$ and $u\in W$.
	Note that the sum is finite since $a[2k-m]u=0$ when $k\gg 0$.
	
	\begin{proposition}[Explicit description of genus-one one-point conformal blocks]\label{prop:oneptcfb}
		Let $(E_\tau,\mpt{p})$ be a one-pointed elliptic curve with the additive local coordinate $z$ from its uniformization. In the square-bracket realization above, $\mathscr{C}(E_\tau,\mpt{p},z,W)$ consists precisely of the linear functionals $\psi\colon W \to \C$ satisfying
		\begin{equation}\label{eq:oneptcfb-characterization}
			\begin{aligned}
				&\braket<\psi|a[0]u>=0,\\
				&\braket<\psi|a[-m]u+(-1)^m\sum_{k=1}^\infty \binom{2k-1}{m-1}G_{2k}(\tau) a[2k-m]u>=0,\quad m \ge 2,
			\end{aligned}
		\end{equation}
		for all $a\in V$ and $u\in W$.
	\end{proposition}
	\begin{proof}
		By \eqref{eq:H0sectionidentification} and \eqref{eq:defspanningaction}, the two displayed relations say exactly that the functional annihilates
		$\mathcal{L}_{E_\tau\setminus\mpt{p}}(V).W$.
		The claim now follows from \cref{def:coinv} applied to the square-bracket realization and the preceding transport isomorphism.
	\end{proof}

	\subsection{Coinvariants on smoothing}\label{subsec:formalcoinv}
	
	To prepare for trace recursion, we extend the coinvariant construction from a fixed elliptic curve to a smoothing family\footnote{Unlike in \cref{subsec:FTVB}, we are able to work in an actual family rather than just a formal one.}. 
	Throughout this subsection, $W$ is an admissible $V$-module of conformal weight $h_W$, and $q$ is a formal variable.
	
	Write $\widetilde{G}_{2k}(q)$ for the $q$-expansion of $G_{2k}(\tau)$:
	\begin{equation}\label{eq:formalG2kq}
		\widetilde{G}_{2k}(q)
		=2\zeta(2k)+\frac{2(2\pi\iu)^{2k}}{(2k-1)!}
		\sum_{n=1}^\infty\sigma_{2k-1}(n)q^n
		\in\C\dbrack{q}.
	\end{equation}
	For $k\ge2$, these series belong to $\C[\widetilde{G}_4(q),\widetilde{G}_6(q)]$ \cite[Chapter~7]{S73}. Following \cite{Zhu96}, set
	\[
	R:=\C[\widetilde{G}_2(q),\widetilde{G}_4(q),\widetilde{G}_6(q)]
	\subset\C\dbrack{q}.
	\]
	The ring $R$ is Noetherian, and its elements converge for $|q|<1$.
	Let $\sch{E}_R\subset\PP^2_R$ be the Weierstrass family over $\operatorname{Spec}R$ defined by
	\[
	Y^2Z=4X^3-60\widetilde{G}_4(q)XZ^2-140\widetilde{G}_6(q)Z^3,
	\]
	with marked section $\mPt{P}_R=[0:1:0]$ and formal coordinate $z$ normalized by $z(\mPt{P}_R)=0$ and $\d{z}=\frac{\d(X/Z)}{(Y/Z)}$.
	Its restriction to $B:=\Set*{\widetilde{G}_4\neq0 \vee \widetilde{G}_6\neq0}$ is a stable genus-one family. We keep the notation $(\sch{E}_R,\mPt{P}_R,z)$ for the family restricted to $B$.
	
	On $\sch{E}_R$, set $\wp_{2,R}=X/Z+\widetilde{G}_2(q)$ and define $\wp_{m,R}$ ($m\ge3$) by the derivative recursion as in \eqref{eq:modifiedwp}.
	At the marked section, the functions $\wp_{m,R}$ have the Laurent expansions \eqref{expansionofwpm}, with $G_{2k}(\tau)$ replaced by $\widetilde{G}_{2k}(q)$. Hence the residue formula \eqref{eq:defspanningaction} becomes, over $R$,
	\begin{equation}\label{eq:formalspanning}
		\begin{aligned}
			(a\otimes1).u
			&:=\Res_z Y_W[a,z]u=a[0]u,\\
			(a\otimes\wp_{m,R}).u
			&:=\Res_z Y_W[a,z]\iota_z(\wp_{m,R})u\\
			&=a[-m]u+(-1)^m\sum_{k\ge1}
			\binom{2k-1}{m-1}\widetilde{G}_{2k}(q)a[2k-m]u,
			\quad m\ge2,
		\end{aligned}
	\end{equation}
	where $a\in V$ and $u\in W$. Each sum is finite by truncation, so the displayed actions take values in $W\otimes R$. The corresponding classes satisfy the node matching condition of \cite[Remark~8.3.1]{DGT24}. The same argument as in previous subsection shows that the values in \eqref{eq:formalspanning} span the image of the chiral action on $B$ over $\mathcal{O}_B$.
	
	Define the $R$-module
	\begin{equation}\label{eq:formalcoinv}
		[W\otimes R]
		:=\frac{W\otimes R}
		{\text{$R$-span of the vectors given by \eqref{eq:formalspanning}}}.
	\end{equation}
	Denote the corresponding quasi-coherent sheaf on $\operatorname{Spec}R$ by
	\[
	\cova{W}_R:=\widetilde{[W\otimes R]}.
	\]
	Over $B$, the presentation \eqref{eq:formalcoinv} identifies $\cova{W}_R$ with the coinvariant sheaf of $(\sch{E}_R,\mPt{P}_R,z)$.
	Since every element of $R$ converges for $|q|<1$, evaluating $q$ at $q_\tau:=\e^{2\pi\iu\tau}$ for various $\tau\in\H$ gives a holomorphic map $\H\to B^{\mathrm{an}}$. Pulling back $\sch{E}_R^{\mathrm{an}}$ gives a family over $\H$ with fiber $E_\tau^{\mathrm{an}}$ at $\tau$. By the action formula \eqref{eq:defspanningaction}, its coinvariant sheaf is $\mathcal{O}_\H\otimes_R[W\otimes R]$, with presentation \eqref{eq:formalcoinv} and fibers
	\[
	[W\otimes R]\otimes_{R,\operatorname{ev}_\tau}\C
	\cong[W]_{(E_\tau,\mpt{p},z)},
	\]
	where $\operatorname{ev}_\tau$ evaluates at $q_\tau:=\e^{2\pi\iu\tau}$.
	
	To prove coherence of the coinvariant sheaf amounts to showing finite generation of $[W\otimes R]$. In view of the $m=2$ action in \eqref{eq:formalspanning}, we consider the square-bracket analogue of $C_2(W)$:
	\[
	C_{[2]}(W):=C_2(W^\square)
	=\operatorname{span}\Set*{a[-2]u\given a\in V,\ u\in W}.
	\]
	By \cref{prop:square-transport:cofiniteness}, $C_{[2]}$-cofiniteness is equivalent to $C_2$-cofiniteness.
	
	\begin{lemma}\label{lem:L[-2]}
		Let $U\subseteq W$ be an $L(0)$-homogeneous subspace with $W=U+C_2(W)$. Then $W=U+C_{[2]}(W)$, and the natural map $U\otimes R\to[W\otimes R]$ is surjective.
	\end{lemma}
	\begin{proof}
		The two gradings give the common filtration
		\[
		W_{\le n}:=\bigoplus_{j=0}^nW(j)=\bigoplus_{j=0}^nW[j]
		\]
		by \cref{prop:square-transport:module-grading}. For $L(0)$-homogeneous $a$ and $u$,
		\[
		(2\pi\iu)^{-1}a[-2]u
		=a_{(-2)}u+\sum_{i\ge-1}\lambda_i a_{(i)}u,
		\]
		where the sum is finite by truncation and every nonzero term has smaller $L(0)$-degree than $a_{(-2)}u$. The decomposition $W=U+C_2(W)$ respects the $L(0)$-grading. Replacing its $C_2$-generators by the displayed mode expansion and inducting on the common filtration gives
		\begin{equation}\label{eq:filteredUC2}
			W_{\le n}=(U\cap W_{\le n})+(C_{[2]}(W)\cap W_{\le n}).
		\end{equation}
		The base case uses $W_{\le1}\subseteq U$, since nonzero $C_2$-generators have degree at least $2$. Taking the union gives $W=U+C_{[2]}(W)$.
		
		In the quotient, the $m=2$ action in \eqref{eq:formalspanning} gives
		\begin{equation}\label{eq:a[-2]}
			a[-2]u+\sum_{k\ge1}(2k-1)\widetilde{G}_{2k}(q)a[2k-2]u=0
			\quad\text{in }[W\otimes R].
		\end{equation}
		For $L[0]$-homogeneous $a$ and $u$, every correction term has smaller degree than $a[-2]u$. Since $C_{[2]}(W)$ is $L[0]$-homogeneous, \eqref{eq:filteredUC2} reduces the induction to such generators. Applying \eqref{eq:a[-2]} then represents every class by $U\otimes R$.
	\end{proof}
	
	For $W=V$, subtracting the zero-mode term $\widetilde{G}_2(q)a[0]u$ from the $m=2$ action in \eqref{eq:formalspanning} recovers Zhu's generators of $O_q(V)$. The following proposition generalizes \cite[Lemmas~4.4.1 and~4.4.2]{Zhu96}.
	\begin{proposition}\label{prop:C2coinv}
		If $W$ is $C_2$-cofinite, then $[W\otimes R]$ is a finitely generated $R$-module. In particular, for every $u\in W$ there exist $s\in\N$ and $g_i(q)\in R$ such that
		\[
		[L[-2]^su]+\sum_{i=0}^{s-1}g_i(q)[L[-2]^iu]=0\quad\text{in }[W\otimes R].
		\]
	\end{proposition}
	\begin{proof}
		Choose a finite-dimensional $U$ as in \cref{lem:L[-2]}. Its surjection $U\otimes R\to[W\otimes R]$ proves finite generation. Since $R$ is Noetherian, the quotient is a Noetherian $R$-module. For fixed $u$, the increasing chain of submodules generated by $[u],[L[-2]u],\ldots$ therefore stabilizes, giving the displayed relation.
	\end{proof}
	
	\begin{remark}\label{rmk:C20coinvfd}
		The same argument proves finite generation of $[W\otimes R]$ when $W/C_{[2,0]}(W)$ is finite-dimensional, where $C_{[2,0]}(W)$ denote the following subspace introduced in \cite{Mi00}:
		\[
		\operatorname{span}\Set*{a[-2]u,\ a[0]u\given a\in V,\ u\in W}.
		\]
		Indeed, the vectors $a[0]u$ vanish in the quotient, and \eqref{eq:a[-2]} lowers $L[0]$-degree, so induction gives a surjection $U\otimes R\to[W\otimes R]$. No $C_2$-cofiniteness of $V$ is needed.
	\end{remark}

	\section{Trace functions associated to intertwining operators}\label{sec:tracefunctions}
	
	We now turn algebraic intertwining-operator data into formal $q$-traces and prove that these formal objects converge to holomorphic conformal blocks.
	As in the previous subsection, $q$ is a formal variable.
	
	\subsection{Formal trace functions associated to intertwining operators}

	Let $M^1,M^2,M^3$ be three admissible $V$-modules of conformal weights $h_1,h_2,h_3$ respectively.
	An intertwining operator $I(\placeholder,w)\in \Ifusion$ can be written as 
	$I(u,w) = \sum_{k\in \Z} u_{(k)} w^{-k-1-h}$, where $h=h_1+h_2-h_3$.
	For a homogeneous $u\in M^1(n)=M^1_{h_1+n}$, with $n\in \N$, write $\deg u=n$.   
	Then we have 
	\[
	\vo{u}{k}M^2(m) \subseteq M^3(\deg u-k-1+m)
	\]
	for all $m\in\N$.
	For homogeneous $u\in M^1$, define
	\begin{equation}\label{eq:oI}
		o_I(u):=\Res_w I(u,w)w^{h+\deg u-1}=\vo{u}{\deg u-1},
	\end{equation}
	and extend $o_I$ linearly to $M^1$. Then
	$o_I(u)M^2(n)\subseteq M^3(n)$ for all $u\in M^1$ and $n\in\N$;
	see \cite{FZ92,Li99,Liu23,GLZ25} for more details.
	
	\begin{remark}
		For arbitrary $u\in M^1$, one can also write $o_I(u)$ as $\Res_w w^{h_2-h_3-1}I(w^{L(0)}u,w)$. In particular, $o_I(u)=\Res_w w^{-1}I(w^{L(0)}u,w)$ when $M^2=M^3$.
	\end{remark}
	
	Let $W$ be an admissible $V$-module of conformal weight $h_W$ and
	$M$ an ordinary $V$-module of conformal weight $h_M$.
	Fix $I\in\Ifusion[W][M][M]$.
	Since each component $M(n)$ is finite-dimensional, we obtain a well-defined linear map:
	\[
	\tr_{M}\ab(o_I(\placeholder )q^{L(0)-\frac{\ccharge}{24}}) \colon W \longrightarrow \C \dbrack{q}q^{h_M-\frac{\ccharge}{24}},\quad
	u \longmapsto \sum_{n=0}^\infty \tr_{M(n)}\ab(o_I(u)) q^{n+h_M-\frac{\ccharge}{24}}.
	\]
	Note that $I(u,w)=\sum_{n\in\Z}u_{(n)}w^{-n-1-h_W}$.
	Define
	\begin{equation}\label{onepointF}
		\begin{aligned}
			F\ab((u,w),q):&= \tr_{M}\ab(I(w^{L(0)}u,w)q^{L(0)-\frac{\ccharge}{24}}),\\
			F\ab((a,x)(u,w),q):&=x^{\wt a} \tr_{M}\ab(Y_M(a,x)I(w^{L(0)}u,w) q^{L(0)-\frac{\ccharge}{24}}),\\
			F\ab((u,w)(a,x),q):&=x^{\wt a} \tr_{M}\ab(I(w^{L(0)}u,w)Y_M(a,x) q^{L(0)-\frac{\ccharge}{24}}),
		\end{aligned}
	\end{equation}
	These formal $q$-series have coefficients in
	$\C[x^{\pm1},w^{\pm1},(x-w)^{\pm1}]$.
	Extend the last two expressions linearly in $a$; they differ only in
	the order of the operators inside the trace.
	
	By the $L(0)$-bracket relations, we have
	\begin{equation} \label{L0comm}
		\begin{aligned}
			Y_M(a,x)q^{L(0)} &= q^{-\wt a} q^{L(0)} Y_M(a,xq^{-1}),\\
			I(u,w) q^{L(0)} &= q^{L(0)}I(q^{-L(0)}u,wq^{-1}).
		\end{aligned}
	\end{equation}
	Taking residues in \eqref{L0comm}, we have
	\begin{equation}\label{L0commcomponent}
		a_{(\wt a-1+k)}q^{L(0)-\frac{\ccharge}{24}} = 
		q^{L(0)+k-\frac{\ccharge}{24}}a_{(\wt a-1+k)},\quad a\in V,\ k\in \Z.
	\end{equation}
	
	The following lemma is useful for deriving the recursive formulas.
	\begin{lemma}\label{lemma:sumformulaforP}
		For any $a\in V$ and $u\in W$, we have
		\[
		\sum_{i \ge 0}\sum_{k= 1}^\infty  \ab(\binom{\wt a-1+k}{i} \frac{x^k}{1-q^k} + \binom{\wt a-1-k}{i} \frac{x^{-k}}{1-q^{-k}} ) a_{(i)}u 
		= \sum_{m=0}^\infty P_{m+1}(x,q) a[m]u.
		\]
	\end{lemma}
	\begin{proof}
		It follows from \eqref{eq:defPm+1} and \eqref{eq:defa[m]} that
		\begin{align*}
			\MoveEqLeft[1]
			\sum_{m=0}^\infty P_{m+1}(x,q) a[m]u\\
			&=
			\Res_{z}\sum_{m=0}^\infty\sum_{k=1}^\infty 
			\ab( 
			\frac{(k\ln(1+z))^m}{m!} \cdot \frac{x^k}{1-q^k} - 
			\frac{(-k\ln(1+z))^m}{m!} \cdot \frac{x^{-k}q^k}{1-q^k}
			) \cdot (1+z)^{\wt a-1}Y_W(a,z)u\\
			&=
			\Res_z\sum_{k=1}^\infty 
			\ab(
			\exp(k\ln(1+z)) \cdot \frac{x^k}{1-q^k} + 
			\exp(-k\ln(1+z)) \cdot \frac{x^{-k}}{1-q^{-k}} 
			)\cdot (1+z)^{\wt a-1}Y_W(a,z)u\\
			&=
			\Res_z\sum_{k=1}^\infty 
			\ab(
			(1+z)^{\wt a-1+k} \cdot \frac{x^k}{1-q^k} + 
			(1+z)^{\wt a-1-k} \cdot \frac{x^{-k}}{1-q^{-k}} 
			)\cdot Y_W(a,z)u\\
			&=
			\sum_{i \ge 0}\sum_{k= 1}^\infty 
			\ab(
			\binom{\wt a-1+k}{i} \frac{x^k}{1-q^k} + 
			\binom{\wt a-1-k}{i} \frac{x^{-k}}{1-q^{-k}} 
			) a_{(i)}u,
		\end{align*}
		where we used the fact that $P_{m+1}(x,q)$ converges locally uniformly to swap the order of the infinite sums.
	\end{proof}

	We first treat the two operator orders in \eqref{onepointF} separately. The resulting formulas generalize \cite[Propositions 4.3.2, 4.3.3]{Zhu96}.
	By linearity, we may assume in the proofs below that $a$ and $u$ are
	$L(0)$-homogeneous.
	
	\begin{lemma}\label{lem:recursiveright}
		For any $a\in V$ and $u\in W$, we have
		\[
		\begin{aligned}
			F\ab((u,w)(a,x),q)
			&= \tr_{M}\ab(a_{(\wt a-1)}I(w^{L(0)}u,w) q^{L(0)-\frac{\ccharge}{24}})-2\pi\iu F\ab((a[0]u,w),q)\\
			&\quad +\sum_{m=0}^\infty P_{m+1}\ab(\tfrac{w}{x}q,q) F\ab((a[m]u,w),q).
		\end{aligned}
		\]
	\end{lemma}
	\begin{proof}
		By \eqref{eq:defa[m]}, $a[0]=(2\pi\iu)^{-1}\Res_z \ab(Y_W(a,z)(1+z)^{\wt a-1}) =(2\pi\iu)^{-1} \sum_{j \ge 0} \binom{\wt a-1}{j} a_{(j)}$. Furthermore, for any $k \ge 0$, we have
		\[I(u,w)a_{(\wt a-1+k)}=a_{(\wt a-1+k)}I(u,w)-\sum_{j \ge 0}\binom{\wt a-1+k}{j} I(a_{(j)}u,w)w^{\wt a-1+k-j}. \]
		Therefore, by \eqref{onepointF}, we have
		\[
		\tr_{M}\ab(I(w^{L(0)}u,w)a_{(\wt a-1)}q^{L(0)-\frac{\ccharge}{24}}) =
		\tr_{M}\ab(a_{(\wt a-1)}I(w^{L(0)}u,w)q^{L(0)-\frac{\ccharge}{24}}) -
		2\pi\iu F\ab((a[0]u,w),q).
		\]
		It follows that
		\begin{align*}
			F\ab((u,w)(a,x),q)
			&=
			x^{\wt a} \tr_{M}\ab(I(w^{L(0)}u,w)\sum_{k\in \Z} a_{(\wt a-1+k)}x^{-\wt a-k} q^{L(0)-\frac{\ccharge}{24}})\\
			&=
			\tr_{M}\ab(I(w^{L(0)}u,w)a_{(\wt a-1)}q^{L(0)-\frac{\ccharge}{24}})\\
			&\quad 
			+ \sum_{k\neq 0} x^{-k} \tr_{M}\ab(I(w^{L(0)}u,w) a_{(\wt a-1+k) } q^{L(0)-\frac{\ccharge}{24}})\\
			&= 
			\tr_{M}\ab(a_{(\wt a-1)}I(w^{L(0)}u,w) q^{L(0)-\frac{\ccharge}{24}})-2\pi\iu F\ab((a[0]u,w),q)\\
			&\quad 
			+ \sum_{k\neq 0} x^{-k} \tr_{M}\ab(I(w^{L(0)}u,w) a_{(\wt a-1+k) } q^{L(0)-\frac{\ccharge}{24}}).
		\end{align*}
		Note that for finite-dimensional spaces $X$ and $Y$, if $A \colon X \longrightarrow Y$ and $B \colon Y \longrightarrow X$ are linear operators, then
		$
		\tr_{X}\ab(BA) = \tr_{Y}\ab(AB).
		$
		For $k\neq 0$ and $n \ge 0$, let $A=a_{(\wt a-1+k)} \colon M(n) \longrightarrow M(n-k)$ and $B=u_{(\deg u-1-k)} \colon M(n-k) \longrightarrow M(n)$. Then \eqref{L0commcomponent} gives
		\begin{align*}
			\MoveEqLeft[1]
			\tr_{M}\ab(I(w^{L(0)}u,w) a_{(\wt a-1+k)} q^{L(0)-\frac{\ccharge}{24}}) =
			\tr_{M}\ab((I(w^{L(0)}u,w) q^{L(0)-\frac{\ccharge}{24}}) a_{(\wt a-1+k)}) q^k\\
			&=
			\sum_{n=0}^\infty \tr_{M(n)}\ab(\ab( u_{(\deg u-1-k)} q^{n-k+h_M-\frac{\ccharge}{24}}) a_{(\wt a-1+k)}) w^k q^k \\
			&= 
			\sum_{n=0}^\infty \tr_{M(n-k)}\ab(a_{(\wt a-1+k)} \ab( u_{(\deg u-1-k)} q^{n-k+h_M-\frac{\ccharge}{24}})) w^k q^k \\
			&=
			\tr_{M}\ab(a_{(\wt a-1+k) } (I(w^{L(0)}u,w) q^{L(0)-\frac{\ccharge}{24}})) q^k\\
			&=
			\tr_{M}\ab(I(w^{L(0)}u,w)a_{(\wt a-1+k)} q^{L(0)-\frac{\ccharge}{24}}) q^k\\
			&\quad 
			+ \sum_{i \ge 0}\binom{\wt a-1+k}{i} \tr_{M}\ab(I(w^{L(0)}a_{(i)}u,w)w^k q^{L(0)-\frac{\ccharge}{24}}) q^k.
		\end{align*}
		Moving the first term on the right-hand side of the equation above to the left gives
		\begin{align*}
			\MoveEqLeft[1]
			\sum_{k \neq 0}x^{-k} \tr_{M}\ab(I(w^{L(0)}u,w) a_{(\wt a-1+k)} q^{L(0)-\frac{\ccharge}{24}})\\
			&=
			\sum_{k \neq 0}x^{-k} \sum_{i \ge 0}\binom{\wt a-1+k}{i} \frac{q^k}{1-q^k} \tr_{M}\ab(I(w^{L(0)}a_{(i)}u,w)w^k q^{L(0)-\frac{\ccharge}{24}})\\
			&=
			\sum_{i \ge 0} \ab(\sum_{k \ge 1} \binom{\wt a-1+k}{i} \frac{q^k}{1-q^k} \ab(\frac{w}{x})^k+\sum_{k \ge 1} \binom{\wt a-1-k}{i} \frac{q^{-k}}{1-q^{-k}} \ab(\frac{w}{x})^{-k} ) F\ab((a_{(i)}u,w),q)\\
			&=
			\sum_{m=0}^\infty P_{m+1}\ab(\tfrac{w}{x}q,q) F\ab((a[m]u,w),q),
		\end{align*}
		where the last equality follows from \cref{lemma:sumformulaforP}, with $x$ replaced by $wq/x$. %
	\end{proof}
	
	\begin{lemma}\label{lem:recursiveleft}
		For any $a\in V$ and $u\in W$, we have
		\[
		F\ab((a,x)(u,w),q)
		= \tr_{M}\ab(a_{(\wt a-1)}I(w^{L(0)}u,w)q^{L(0)-\frac{\ccharge}{24}})\\
		+ \sum_{m = 0}^\infty P_{m+1}\ab(\tfrac{w}{x},q) F\ab((a[m]u,w),q).
		\]
	\end{lemma}
	\begin{proof}
		As in the proof of \cref{lem:recursiveright}, \eqref{L0commcomponent} gives
		\begin{align*}
			\MoveEqLeft[1]
			F\ab((a,x)(u,w),q)\\
			&= 
			x^{\wt a} \tr_{M}\ab(\sum_{k\in \Z} a_{(\wt a-1+k)} I(w^{L(0)}u,w)x^{-\wt a-k} q^{L(0)-\frac{\ccharge}{24}})\\
			&=
			\tr_{M}\ab(a_{(\wt a-1)}I(w^{L(0)}u,w) q^{L(0)-\frac{\ccharge}{24}})\\
			&\quad 
			+ \sum_{k\neq 0} x^{-k} \tr_{M}\ab(a_{(\wt a-1+k)} I(w^{L(0)}u,w) q^{L(0)-\frac{\ccharge}{24}})\\
			&=
			\tr_{M}\ab(a_{(\wt a-1)}I(w^{L(0)}u,w) q^{L(0)-\frac{\ccharge}{24}})\\
			&\quad 
			+ \sum_{k \ge 1} \sum_{i \ge 0} \ab(\binom{\wt a-1+k}{i} \frac{1}{1-q^k}\ab(\frac{w}{x})^k+\binom{\wt a-1-k}{i} \frac{1}{1-q^{-k}}\ab(\frac{w}{x})^{-k} ) F\ab((a_{(i)}u,w),q)\\
			&=
			\tr_{M}\ab(a_{(\wt a-1)}I(w^{L(0)}u,w) q^{L(0)-\frac{\ccharge}{24}})+\sum_{m=0}^\infty P_{m+1}\ab(\tfrac{w}{x},q) F\ab((a[m]u,w),q),
		\end{align*}
		where the last equality follows from \cref{lemma:sumformulaforP}, with $x$ replaced by $w/x$.
	\end{proof}
	
	By \cref{coro:Pfunctioncov}, assuming the convergence of traces, the formulas of \cref{lem:recursiveright,lem:recursiveleft} are expansions, in their respective annuli, of the same meromorphic function.
	
	\subsection{Recursive formulas for the formal trace functions}
	
	We combine the two operator-order formulas with the Jacobi identity to derive recursion formulas for formal trace functions on square-bracket modes. 
	Then a special case yields a first-order differential equation for the trace functions.
	
	We need the following facts about the function $P_{m+1}$; see \cite[(4.3.8)--(4.3.11)]{Zhu96} for details.
	\begin{lemma}\label{lemma:intermediateformula}
		Write $(1+z)^{\wt a-1} \ab(\ln (1+z))^{-1}=\sum_{i \ge -1} c_i z^i$. Then for any $m \ge 1$,
		\begin{align*}
			&\sum_{i \ge -1} c_i \Res_x \ab(
			\iota_{x,w}((x-w)^i) \frac{ w^{\wt a-i-1} }{ x^{\wt a} } - 
			\iota_{w,x}((x-w)^i) \frac{ w^{\wt a-i-1} }{ x^{\wt a} }
			)
			= 1,\\
			&\sum_{i \ge -1} c_i\Res_x\ab(
			\iota_{x,w}((x-w)^i) \frac{ w^{\wt a-i-1} }{ x^{\wt a} } 
			P_{1}\ab(\tfrac{w}{x},q) - 
			\iota_{w,x}((x-w)^i) \frac{ w^{\wt a-i-1} }{ x^{\wt a} } 
			\ab( P_{1}\ab(\tfrac{wq}{x},q)-2\pi\iu )
			)
			= -\pi\iu,\\
			&\sum_{i \ge -1} c_i\Res_x\ab(
			\iota_{x,w}((x-w)^i) \frac{ w^{\wt a-i-1} }{ x^{\wt a} }
			P_{m+1}\ab(\tfrac{w}{x},q) -
			\iota_{w,x}((x-w)^i) \frac{ w^{\wt a-i-1} }{ x^{\wt a} }
			P_{m+1}\ab(\tfrac{wq}{x},q)
			)
			= \widetilde{G}_{m+1}(q).
		\end{align*}
	\end{lemma}
	
	\begin{proposition}\label{prop:recursive-trace-identity}
		For all $a\in V$ and $u\in W$, we have the following identity of formal power series:
		\[
		F\ab((a[-1]u,w),q)
		= \tr_{M}\ab(a_{(\wt a-1)}I(w^{L(0)}u,w) q^{L(0)-\frac{\ccharge}{24}})
		+ \sum_{k=1}^\infty\widetilde{G}_{2k}(q) F\ab((a[2k-1]u,w),q).
		\]
	\end{proposition}
	\begin{proof}
		Write $(1+z)^{\wt a-1} \ab(\ln (1+z))^{-1}=\sum_{i \ge -1} c_i z^i$. Combining the Jacobi identity with \cref{lem:recursiveright,lem:recursiveleft} and \cref{lemma:intermediateformula}, we obtain
		\begin{align*}
			\MoveEqLeft[1]
			F\ab((a[-1]u,w),q)\\
			&=
			\sum_{i \ge -1} c_i \tr_{M}\ab(I(w^{L(0)}a_{(i)}u,w)q^{L(0)-\frac{\ccharge}{24}})\\
			&=
			\sum_{i \ge -1} c_i w^{\wt a-i-1} \Res_{x-w} \iota_{w,x-w}((x-w)^i)
			\tr_{M}\ab(I(Y_W(a,x-w)w^{L(0)}u,w)q^{L(0)-\frac{\ccharge}{24}})\\
			&=
			\sum_{i \ge -1} c_i \Res_x \iota_{x,w}((x-w)^i) w^{\wt a-i-1}
			\tr_{M}\ab(Y_M(a,x)I(w^{L(0)}u,w)q^{L(0)-\frac{\ccharge}{24}})\\
			&\quad 
			-\sum_{i \ge -1} c_i \Res_x \iota_{w,x}((x-w)^i) w^{\wt a-i-1}
			\tr_{M}\ab(I(w^{L(0)}u,w)Y_M(a,x)q^{L(0)-\frac{\ccharge}{24}})\\
			&=
			\sum_{i \ge -1} c_i \Res_x \iota_{x,w}((x-w)^i) w^{\wt a-i-1} x^{-\wt a} F\ab((a,x)(u,w),q)\\
			&\quad 
			-\sum_{i \ge -1} c_i \Res_x \iota_{w,x}((x-w)^i) w^{\wt a-i-1} x^{-\wt a} F\ab((u,w)(a,x),q)\\
			&=
			\sum_{i \ge -1} c_i \Res_x \ab(\iota_{x,w}((x-w)^i) w^{\wt a-i-1} x^{-\wt a}-\iota_{w,x}((x-w)^i) w^{\wt a-i-1}x^{-\wt a} )\\
			&\qquad 
			\cdot \tr_{M}\ab(a_{(\wt a-1)}I(w^{L(0)}u,w)q^{L(0)-\frac{\ccharge}{24}})\\
			&\quad 
			+\sum_{i \ge -1} c_i\Res_x\pab[bigg]{\iota_{x,w}((x-w)^i) w^{\wt a-i-1} x^{-\wt a}P_{1}\ab(\tfrac{w}{x},q)\\
				&\qquad 
				-\iota_{w,x}((x-w)^i) w^{\wt a-i-1}x^{-\wt a}\ab(P_{1}\ab(\tfrac{wq}{x},q)-2\pi\iu)}\cdot F\ab((a[0]u,w),q)\\
			&\quad 
			+\sum_{m \ge 1}\sum_{i \ge -1} c_i\Res_x\pab[bigg]{\iota_{x,w}((x-w)^i) w^{\wt a-i-1} x^{-\wt a}P_{m+1}\ab(\tfrac{w}{x},q)\\
				&\qquad 
				-\iota_{w,x}((x-w)^i) w^{\wt a-i-1}x^{-\wt a}P_{m+1}\ab(\tfrac{wq}{x},q)} \cdot F\ab((a[m]u,w),q)\\
			&=
			\tr_{M}\ab(a_{(\wt a-1)}I(w^{L(0)}u,w) q^{L(0)-\frac{\ccharge}{24}})\\
			&\quad 
			-\pi\iu F\ab((a[0]u,w),q) + \sum_{k=1}^\infty\widetilde{G}_{2k}(q) F\ab((a[2k-1]u,w),q),
		\end{align*}
		where the last equality follows from \cref{lemma:intermediateformula}.
		We claim that $F\ab((a[0]u,w),q)=0$. 
		Indeed, since $(2\pi\iu) a[0]=\sum_{i \ge 0}\binom{\wt a-1}{i} a(i)$ \eqref{eq:defa[m]}, and
		$
		\ab[o(a),I(u,w)] =
		\sum_{i \ge 0} \binom{\wt a-1}{i} w^{\wt a-1-i}I(a(i)u,w),
		$
		it follows that
		\begin{align*}
			(2\pi\iu) F\ab((a[0]u,w),q)
			&=
			\sum_{i \ge 0} \binom{\wt a-1}{i} \tr_{M}\ab(I(w^{L(0)}a_{(i)}u,w) q^{L(0)-\frac{\ccharge}{24}})\\
			&= 
			\tr_{M}\ab(\ab[o(a),I(w^{L(0)}u,w)] q^{L(0)-\frac{\ccharge}{24}})\\
			&=
			\sum_{n=0}^\infty \tr_{M(n)}\ab(\ab[o(a),o_I(u)]) q^{h_M+n-\frac{\ccharge}{24}}
			= 0.
		\end{align*}
		Here we used $\tr_{M(n)}\ab(o(a)o_I(u))=\tr_{M(n)}\ab(o_I(u)o(a))$ for all $n \ge 0$.
		This finishes the proof.
	\end{proof}
	As a corollary of the proof, we have the following recursive formulas for the formal trace functions defined by intertwining operators, generalizing  \cite[Propositions 4.3.5, 4.3.6]{Zhu96}.
	\begin{corollary}\label{coro:recursive}
		Let $a\in V$ and $u\in W$, and let $m \ge 2$. We have the following identities of formal power series in $q$:
		\begin{enumerate}
			\item \label{traceformulafora[-1]}
			$\displaystyle \tr_{M}\ab(o_I(a[-1]u)q^{L(0)-\frac{\ccharge}{24}}) = \tr_{M}\ab(o(a)o_I(u)q^{L(0)-\frac{\ccharge}{24}})+\sum_{k=1}^\infty \widetilde{G}_{2k}(q) \tr_{M}\ab(o_I(a[2k-1]u)q^{L(0)-\frac{\ccharge}{24}})$;
			\item \label{tracea[0]}
			$\displaystyle \tr_{M}\ab(o_I(a[0]u)q^{L(0)-\frac{\ccharge}{24}}) = 0$;
			\item \label{tracea[m]}
			$\displaystyle \tr_{M}\ab(o_I(a[-m]u)q^{L(0)-\frac{\ccharge}{24}}) = (-1)^{m-1} \sum_{k=1}^\infty \binom{2k-1}{m-1}\widetilde{G}_{2k}(q) \tr_{M}\ab(o_I(a[2k-m]u)q^{L(0)-\frac{\ccharge}{24}})$.
		\end{enumerate}
	\end{corollary}
	\begin{proof}
		The first identity follows from \cref{prop:recursive-trace-identity}, and the vanishing $F\ab((a[0]u,w),q)=0$ established in its proof gives the second. We prove \itemonlycref{tracea[m]} by induction on $m$. First consider $m=2$. Note that
		\[L[-1]=\widetilde{\cfv}[0]=(2\pi\iu)^2\cfv[0]=2\pi\iu \Res_z Y(\cfv,z)(1+z)=2\pi\iu (L(-1)+L(0)),\]
		where $\cfv[0]=(2\pi\iu)^{-1}\Res_z Y(\cfv,z)(1+z)$ by \eqref{eq:defa[m]}, and the vacuum contribution to $\widetilde{\cfv}$ has vanishing zero mode.
		It follows that $o(L[-1]a)=2\pi\iu\,o(L(-1)a+L(0)a)=0$. Replacing $a$ by $L[-1]a$ in \itemonlycref{traceformulafora[-1]} and using $Y_W[L[-1]a,z]u=\odv{z}Y_W[a,z]u$ gives
		\begin{align*}
			\MoveEqLeft
			\tr_{M}\ab(o_I(a[-2]u)q^{L(0)-\frac{\ccharge}{24}})\\
			&= \tr_{M}\ab(o_I((L[-1]a)[-1]u)q^{L(0)-\frac{\ccharge}{24}})\\
			&= \tr_{M}\ab(o(L[-1]a)o_I(u)q^{L(0)-\frac{\ccharge}{24}})+\sum_{k=1}^\infty \widetilde{G}_{2k}(q) \tr_{M}\ab(o_I( (L[-1]a)[2k-1]u) q^{L(0)-\frac{\ccharge}{24}})\\
			&= (-1)\sum_{k=1}^\infty (2k-1) \widetilde{G}_{2k}(q) \tr_{M}\ab(o_I(a[2k-2]u) q^{L(0)-\frac{\ccharge}{24}}).
		\end{align*}
		This is \itemonlycref{tracea[m]} for $m=2$. Now let $m \ge 2$ and assume that \itemonlycref{tracea[m]} holds for $m$ and for all $a\in V$, $u\in W$. The $L[-1]$-derivative property $Y_W[L[-1]a,z]=\odv{z}Y_W[a,z]$ gives $(L[-1]a)[r]=-r\,a[r-1]$ for all $r\in\Z$; in particular $(L[-1]a)[-m]=m\,a[-m-1]$ and $(L[-1]a)[2k-m]=-(2k-m)\,a[2k-m-1]$. Replacing $a$ by $L[-1]a$ in \itemonlycref{tracea[m]} therefore yields
		\[
		\begin{aligned}
			\MoveEqLeft
			m\,\tr_{M}\ab(o_I(a[-m-1]u)q^{L(0)-\frac{\ccharge}{24}})\\
			&=(-1)^m\sum_{k=1}^\infty (2k-m)\binom{2k-1}{m-1}\widetilde{G}_{2k}(q)\tr_{M}\ab(o_I(a[2k-m-1]u)q^{L(0)-\frac{\ccharge}{24}}).
		\end{aligned}
		\]
		Dividing by $m$ and using $\frac{2k-m}{m}\binom{2k-1}{m-1}=\binom{2k-1}{m}$ (both sides vanish when $2k=m$), we obtain
		\[
		\tr_{M}\ab(o_I(a[-(m+1)]u)q^{L(0)-\frac{\ccharge}{24}}) =
		(-1)^{m}\sum_{k=1}^\infty \binom{2k-1}{m}\widetilde{G}_{2k}(q)\tr_{M}\ab(o_I(a[2k-(m+1)]u)q^{L(0)-\frac{\ccharge}{24}}),
		\]
		which is \itemonlycref{tracea[m]} for $m+1$. This completes the induction.
	\end{proof}
	
	A special case of \cref{traceformulafora[-1]} leads to a first-order differential equation for the formal trace functions, which is crucial in our later discussions.
	\begin{corollary}\label{coro:diffeqfortrace}
		For any $u\in W$, the following formal differential equation holds in $\C\dbrack{q}q^{h_M-\frac{\ccharge}{24}}$:
		\[
		\begin{aligned}
			\MoveEqLeft
			(2\pi\iu)^2\ab(q\odv{q})\ab(\tr_{M}\ab(o_I(u) q^{L(0)-\frac{\ccharge}{24}}))\\
			&= \tr_{M}\ab(o_I(L[-2]u) q^{L(0)-\frac{\ccharge}{24}}) - \sum_{k=1}^\infty \widetilde{G}_{2k}(q) \tr_{M}\ab(o_I(L[2k-2]u)q^{L(0)-\frac{\ccharge}{24}}).
		\end{aligned}
		\]
	\end{corollary}
	\begin{proof}
		By \eqref{eq:oI} and the $L(-1)$-derivative property,
		$o_I(L(-1)u)=-o_I(L(0)u)$. It follows that
		\[
		\ab[L(0),o_I(u)]=o_I(L(-1)u)+o_I(L(0)u)=0.
		\]
		Hence $o(\cfv)o_I(u)=L(0)o_I(u)=o_I(u)L(0)$ and
		\[
		\tr_{M}\ab(o(\cfv) o_I(u)q^{L(0)-\frac{\ccharge}{24}}) =
		\tr_{M}\ab(o_I(u)L(0)q^{L(0)-\frac{\ccharge}{24}}) =
		\ab(q\odv{q}+\frac{\ccharge}{24})\ab(\tr_{M}\ab(o_I(u)q^{L(0)-\frac{\ccharge}{24}})).
		\]
		Since $\widetilde{\cfv} = (2\pi\iu)^2 \ab(\cfv-\frac{\ccharge}{24})$ and $ o(\widetilde{\cfv}) = (2\pi\iu)^2 \ab(o(\cfv)-\frac{\ccharge}{24} \cdot \Id)$, we have
		\begin{equation}\label{eq:intermediatetrace}
			\tr_{M}\ab(o(\widetilde{\cfv})o_I(u)q^{L(0)-\frac{\ccharge}{24}})=(2\pi\iu)^2\ab(q\odv{q}) \cdot \tr_{M}\ab(o_I(u)q^{L(0)-\frac{\ccharge}{24}}).
		\end{equation}
		Setting $a=\widetilde{\cfv}$ in \cref{traceformulafora[-1]} and using \eqref{eq:intermediatetrace} give
		\begin{align*}
			\MoveEqLeft
			\tr_{M}\ab(o_I(L[-2]u) q^{L(0)-\frac{\ccharge}{24}})
			=
			\tr_{M}\ab(o_I(\widetilde{\cfv}[-1]u) q^{L(0)-\frac{\ccharge}{24}})\\
			&= 
			\tr_{M}\ab(o(\widetilde{\cfv})o_I(u)q^{L(0)-\frac{\ccharge}{24}})+\sum_{k=1}^\infty \widetilde{G}_{2k}(q) \tr_{M}\ab(o_I(\widetilde{\cfv}[2k-1]u)q^{L(0)-\frac{\ccharge}{24}})\\
			&=
			(2\pi\iu)^2\ab(q\odv{q}) \cdot \tr_{M}\ab(o_I(u)q^{L(0)-\frac{\ccharge}{24}}) +\sum_{k=1}^\infty \widetilde{G}_{2k}(q) \tr_{M}\ab(o_I(L[2k-2]u)q^{L(0)-\frac{\ccharge}{24}}).
		\end{align*}
		This proves the corollary.
	\end{proof}

	\subsection{Convergence of the formal trace functions}\label{subsec:convergence-formal-trace}
	
	We combine the differential identity of \cref{coro:diffeqfortrace} with the finite-dimensional reduction of \cref{lem:L[-2]} to establish convergence and identify the holomorphic evaluations with conformal blocks.
	In this subsection, we assume that $V$ is of CFT type and keep $W$, $M$, and $I$ as fixed above. For $u\in W$, abbreviate the \concept{formal trace function} by
	\begin{equation}\label{eq:formalZ}
		\mathcal{Z}_I(u,q) := \tr_{M}\ab(o_I(u)q^{L(0)-\frac{\ccharge}{24}})
		= \sum_{n=0}^{\infty}\tr_{M(n)}\ab(o_I(u)) q^{n+h_M-\frac{\ccharge}{24}}
		\in \C\dbrack{q}q^{h_M-\frac{\ccharge}{24}},
	\end{equation}
	and extend it $R$-linearly to $W\otimes R$. 
	By \cref{coro:recursive}, $\mathcal{Z}_I(\placeholder,q)$ vanishes on the vectors in \eqref{eq:formalspanning} and hence induces the $R$-linear map
	\begin{equation}\label{eq:formaltr}
		\mathcal{Z}_I(\placeholder,q) \colon [W\otimes R]\longrightarrow
		\C\dbrack{q}q^{h_M-\frac{\ccharge}{24}}.
	\end{equation}
	After convergence is established below, these series yield $Z_I(u,\tau)$ in \eqref{eq:defZ} using the convention $q_\tau^\beta=\e^{2\pi\iu\tau\beta}$.
	
	Motivated by \cref{coro:diffeqfortrace}, define $\widetilde{T}(q) \in \End(W \otimes R)$ as
	\begin{equation}\label{eq:defTqformal}
		u \in W \otimes R \longmapsto L[-2]u - \sum_{k \ge 1}\widetilde{G}_{2k}(q)L[2k-2]u \in W \otimes R.
	\end{equation}
	Here, the sum is finite: by \cref{prop:square-transport}, $L[j]W[n] \subseteq W[n-j]$, so $L[2k-2]u=0$ for $k\gg0$.
	Using this notation, \cref{coro:diffeqfortrace} states that
	\begin{equation}\label{eq:firstorder}
		(2\pi\iu)^2q\odv{q}\mathcal{Z}_I(u,q)=\mathcal{Z}_I\ab(\widetilde{T}(q)u,q),\qquad u\in W.
	\end{equation}
	The formal operator $\widetilde{T}(q)$ specializes to $T(\tau)$ after the substitution $q=q_\tau$.
	
	The following lemma provides a first-order differential system.   
	\begin{lemma}\label{lem:system}
		Suppose there is a finite-dimensional $L(0)$-homogeneous subspace $U\subseteq W$ such that $W=U+C_2(W)$. Let $u^1,\cdots,u^m$ be a basis of $U$. Then:
		\begin{enumerate}
			\item\label{lem:system:finite-dimensional-lift}
			There exist a matrix $\vect{T}(q):=\ab(t_{ij}(q)) \in \operatorname{Mat}_{m\times m}(R)$ such that 
			\[
			[\widetilde{T}(q)u^i]=\sum_{j=1}^m t_{ij}(q)[u^j]
			\quad\text{in }[W\otimes R],\qquad 1 \le i \le m.
			\]
			\item\label{lem:system:differential-system}
			The column vector $\vect{Z}_I(q):=\ab(\mathcal{Z}_I(u^1,q),\cdots,\mathcal{Z}_I(u^m,q))^{\mathsf{t}}$ satisfies
			\[
			q\odv{q}\vect{Z}_I(q)=\frac{1}{(2\pi\iu)^2}\vect{T}(q)\vect{Z}_I(q).
			\]
			\item\label{lem:system:reduction}
			For every $u\in W$, there exists a family $\Fam{r_j(q)}{1\le j\le m}$ in $R$ such that
			\[
			\mathcal{Z}_I(u,q)=\sum_{j=1}^m r_j(q)\mathcal{Z}_I(u^j,q).
			\]
		\end{enumerate}
	\end{lemma}
	\begin{proof}
		By \cref{lem:L[-2]}, the composition
		\[
		U\otimes R\hookrightarrow W\otimes R\twoheadrightarrow[W\otimes R]
		\]
		is surjective. In particular, the class of $\widetilde{T}(q)u^i\in W\otimes R$ is the image of some $\sum_j u^j\otimes t_{ij}(q)$, which proves \itemonlycref{lem:system:finite-dimensional-lift}.
		Here we do not claim that $\widetilde{T}(q)$ induces an endomorphism of the quotient and the matrix $\vect{T}(q)$ is not unique. 
		Applying \eqref{eq:formaltr} to this relation and using \eqref{eq:firstorder} gives
		\[
		(2\pi\iu)^2q\odv{q}\mathcal{Z}_I(u^i,q)
		=\mathcal{Z}_I\ab(\widetilde{T}(q)u^i,q)
		=\sum_j t_{ij}(q)\mathcal{Z}_I(u^j,q),
		\]
		and hence \itemonlycref{lem:system:differential-system}. Finally, for any $u\in W$, the same surjectivity gives $r_j(q)\in R$ such that
		\[
		[u]=\sum_j r_j(q)[u^j]\quad\text{in }[W\otimes R].
		\]
		Applying \eqref{eq:formaltr} proves \itemonlycref{lem:system:reduction}.
	\end{proof}
	
	We can now prove convergence.
	
	\begin{theorem}\label{thm:conv}
		Suppose there is a finite-dimensional $L(0)$-homogeneous subspace $U\subseteq W$ such that $W=U+C_2(W)$. Let $I\in \Ifusion[W][M][M]$ and put $\alpha=h_M-\frac{\ccharge}{24}$. Then, for every $u\in W$, the normalized power series
		\[
		f_u(q):=q^{-\alpha}\mathcal{Z}_I(u,q)\in\C\dbrack{q}
		\]
		converges absolutely and locally uniformly on $\D=\Set*{q\in\C\given\ab|q|<1}$.
		
		For every $\tau\in\H$, we use the convention $q_\tau^\beta:=\e^{2\pi\iu\tau\beta}$ for $\beta\in\C$. Evaluation of the convergent normalized series induces a conformal block
		\begin{equation}\label{eq:defofvarphitau}
			\varphi_I(\tau):=\tr_{M}\ab(o_I(\placeholder)q_\tau^{L(0)-\frac{\ccharge}{24}})
			\colon [W]_{(E_\tau,\mpt{p},z)}\longrightarrow\C.
		\end{equation}
		Moreover, for every $u\in W$, the function $\tau\mapsto\braket<\varphi_I(\tau)|u>$ is holomorphic on $\H$.
	\end{theorem}
	\begin{proof}
		Fix the basis $u^1, \cdots, u^m$ of $U$ and a matrix $\vect{T}(q)$ as in \cref{lem:system:finite-dimensional-lift}. Write
		\[
		\mathcal{Z}_I(u^i,q)=q^\alpha f_i(q),\qquad
		f_i(q)=\sum_{n=0}^{\infty}\tr_{M(n)}\ab(o_I(u^i)) q^n\in\C\dbrack{q},
		\]
		and put $\vect{f}:=\ab(f_1,\cdots,f_m)^{\mathsf{t}}$. The differential system in \cref{lem:system:differential-system} is equivalent to
		\[
		\ab(q\odv{q}+\alpha)\vect{f} = \frac{1}{(2\pi\iu)^2}\vect{T}(q)\vect{f}.
		\]
		By its Fourier expansion \eqref{eq:formalG2kq}, each $\widetilde{G}_{2k}(q)$ has radius of convergence $1$; hence every entry of $\vect{T}(q)$ is holomorphic on $\D$. 
		By \cite[Chapter~II, Proposition~2.18]{Sab08}, the formal solution $\vect{f}$ converges near $0$. It extends holomorphically to $\D$ by \cite[Chapter~0, Corollary~12.13]{Sab08}, since the differential equation has no other singularities in $\D$. 
		In particular, every $f_i$ converges absolutely and locally uniformly on $\D$. By \cref{lem:system:reduction}, for arbitrary $u\in W$ we have
		\[
		f_u(q) = \sum_{j=1}^m r_j(q)f_j(q),
		\]
		with $r_j(q)\in R$ holomorphic on $\D$, proving the first assertion.
		
		Fix $\tau\in\H$ and extend the evaluation map on $R$ to holomorphic functions on $\D$:
		\[
		\operatorname{ev}_\tau \colon \mathcal{O}^{\mathrm{an}}(\D)\longrightarrow\C, \qquad g(q)\longmapsto g(q_\tau).
		\]
		Since $\operatorname{ev}_\tau\ab(\widetilde{G}_{2k}(q))=G_{2k}(\tau)$, define
		\begin{equation}\label{eq:evdef}
			\braket<\varphi_I(\tau)|u> := \e^{2\pi\iu\tau\alpha}\operatorname{ev}_\tau\ab(f_u), \qquad u\in W.
		\end{equation}
		Note that it depends on $\tau$, not merely on $q_\tau$, when $\alpha\notin\Z$.
		
		For $a\in V$, $u\in W$, and $m \ge 2$, consider the formal relation
		\[
		X := a[-m]u+(-1)^m\sum_{k \ge 1}\binom{2k-1}{m-1}\widetilde{G}_{2k}(q)a[2k-m]u.
		\]
		The sum is finite. The $R$-linearity of $\mathcal{Z}_I$, the multiplicativity of $\operatorname{ev}_\tau$, and \eqref{eq:evdef} give
		\[
		\e^{2\pi\iu\tau\alpha}\operatorname{ev}_\tau\ab(q^{-\alpha}\mathcal{Z}_I(X,q)) = \braket<\varphi_I(\tau)|a[-m]u+(-1)^m\sum_{k \ge 1}\binom{2k-1}{m-1}G_{2k}(\tau)a[2k-m]u>.
		\]
		By \eqref{eq:formalspanning}, the class $[X] \in [W\otimes R]$ is zero. Hence \eqref{eq:formaltr} gives $\mathcal{Z}_I(X,q)=0$, and therefore the left-hand side vanishes.
		The same argument with $X = a[0]u$ instead proves $\braket<\varphi_I(\tau)|a[0]u>=0$. 
		Hence, by \cref{prop:oneptcfb}, $\varphi_I(\tau)$ descends to $[W]_{(E_\tau,\mpt{p},z)}$ and belongs to $\mathscr{C}(E_\tau,\mpt{p},z,W)$ in its square-bracket realization.
		
		Finally, $\tau\mapsto q_\tau$ is holomorphic from $\H$ to $\D$, $f_u$ is holomorphic on $\D$, and $\tau\mapsto \e^{2\pi\iu\tau\alpha}$ is entire. Thus $\tau\mapsto\braket<\varphi_I(\tau)|u>$ is holomorphic.
	\end{proof}
	
	The existence of a finite-dimensional $L(0)$-homogeneous $U$ as required in \cref{thm:conv} is equivalent to $C_2$-cofiniteness of $W$. 
	In particular, the hypothesis is automatic for every irreducible $W$ when $V$ is $C_2$-cofinite \cite[Proposition~5.2]{ABD04}.
	\begin{remark}\label{rmk:C20-convergence}
		The proofs of \cref{lem:system,thm:conv} require only a finite-dimensional representative space $U$ with $U\otimes R\twoheadrightarrow[W\otimes R]$. Hence, the same holds for $C_{[2,0]}$-cofinite $W$, without assuming $V$ is $C_2$-cofinite, by \cref{rmk:C20coinvfd}.
	\end{remark}

	\section{The trace-frame theorem}\label{sec:trace-frame}
	
	The convergent traces now give a global horizontal frame of the bundle of conformal blocks. We identify its connection with (a representative of) the pullback of the Damiolini--Gibney--Tarasca connection and compare the horizontal sections with Miyamoto's one-point functions. 
	Throughout this section, $V$ is rational, $C_2$-cofinite, and of CFT type, unless otherwise stated.
	
	\subsection{The conformal-block bundle on \texorpdfstring{$\H$}{H} and the connection}\label{subsec:conformal-block-bundle-and-connection}
	
	We begin by placing the trace functions associated to intertwining operators \eqref{eq:defofvarphitau} in a holomorphic bundle over the upper half-plane and constructing the connection for which they are horizontal.
	
	In this subsection, $W$ is a finitely generated admissible $V$-module of conformal weight $h_W$. 
	
	\subsubsection{The conformal-block bundle}
	
	Let $j \colon \mathcal{M}_{1,1}\hookrightarrow\overline{\mathcal{M}}_{1,1}$ be the open substack of smooth curves and let $\pi \colon \H \longrightarrow \mathcal{M}_{1,1}^{\mathrm{an}} \simeq [\SL(2,\Z)\dblbackslash\H]$ be the quotient map. 
	Analytifying the conformal-block bundle $\mathscr{C}(W)$ in \cref{sec:elliptic-cfb} and pulling it back along $\pi$, we define
	\[
	\mathscr{C}_{\H}(W) := \pi^\ast \ab((j^{\ast}\mathscr{C}(W))^{\mathrm{an}}).
	\]
	We call this bundle the \concept{conformal-block bundle} on
	$\H$ associated with $W$.
	
	For any $\tau\in\H$, there are canonical isomorphisms
	\begin{equation}\label{eq:fiberid}
		\mathscr{C}_\H(W)_\tau
		\cong
		\ab(j^{\ast}\mathscr{C}(W))^{\mathrm{an}}_{\pi(\tau)}
		\cong
		\mathscr{C}(W)_{(E_\tau,\mpt{p})}
		\cong
		\mathscr{C}(E_\tau,\mpt{p},z,W).
	\end{equation}
	The first two maps come from pullback and analytification; the last is the fiberwise coinvariant construction, written explicitly in \cref{prop:oneptcfb} after choosing the local coordinate $z$. 
	We call an element $\varphi \in \Gamma\ab(\H,\mathscr{C}_{\H}(W))$ a \concept{global section of (genus-one one-point) conformal blocks}.
	
	Let $\Hdual{\mathscr{W}}$ be the dual $\mathcal{O}_{\H}$-module of the constant sheaf $W$; explicitly, over any open $\Omega \subseteq \H$,
	\[
	\Hdual{\mathscr{W}}(\Omega):=\Hom_{\C}\ab(W,\mathcal{O}_{\H}(\Omega)).
	\]
	Equivalently, $\varphi \in \Hdual{\mathscr{W}}(\Omega)$ is the family $\Fam{\varphi(\tau)}{\tau\in \Omega}$ of linear functionals on $W$ such that, for any $u \in W$, the assignment $\tau \mapsto \braket<\varphi(\tau)|u>$ is holomorphic on $\Omega$.
	
	The pullback to $\H$ of the bundle $[W]^{\mathrm{an}}$ is of finite rank by \nameref{coro:VB}, and its fiber at $\tau$ is canonically $[W]_{(E_\tau,\mpt{p})}$.
	By construction it is a quotient of $W\otimes\O_\H$ by the image of the relative chiral Lie algebra. The conformal-block bundle $\mathscr{C}_{\H}(W)$ is the dual of this quotient, so dualizing the quotient map gives a canonical inclusion
	\begin{equation}\label{eq:CintoWvee}
		\mathscr{C}_{\H}(W)\ \hookrightarrow\ \Hdual{\mathscr{W}}.
	\end{equation}
	Since $\Hdual{\mathscr{W}}$ is not a finite-rank vector bundle, the usual finite subbundle criterion does not apply directly to \eqref{eq:CintoWvee}. The following lemma realizes $\mathscr{C}_\H(W)$ inside $U^\ast\otimes\O_\H$ and reduces holomorphy to finitely many scalar functions.
	
	\begin{lemma}\label{lem:sectioncriterion}
		Suppose there is a finite-dimensional $L(0)$-homogeneous subspace $U\subseteq W$ such that $W=U+C_2(W)$. Let $u^1,\cdots,u^m$ be a basis of $U$, let $\Omega\subseteq\H$ be open, and let $\varphi=\Fam{\varphi(\tau)}{\tau\in\Omega}$ be the family of linear functionals on $W$.
		Then $\varphi$ defines a section of $\mathscr{C}_\H(W)$ via \eqref{eq:CintoWvee} if and only if
		\begin{enumerate}
			\item\label{lem:sectioncriterion:fiberwise-condition} $\varphi(\tau)\in\mathscr{C}(E_\tau,\mpt{p},z,W)$ for every $\tau\in \Omega$, that is, $\varphi(\tau)$ satisfies \eqref{eq:oneptcfb-characterization}; and
			\item\label{lem:sectioncriterion:holomorphicity-condition} the $m$ functions $\tau\mapsto\braket<\varphi(\tau)|u^j>$, $1\le j\le m$, are holomorphic on $\Omega$.
		\end{enumerate}
		In particular, $\tau\mapsto\braket<\varphi(\tau)|u>$ is holomorphic on $\Omega$ for
		every $u\in W$.
	\end{lemma}
	\begin{proof}
		Specializing \cref{lem:L[-2]} through $\widetilde{G}_{2k}(q)\mapsto G_{2k}(\tau)$ gives a surjection $U\to\ab(\mathscr{C}_\H(W)^\vee)_\tau$ for every $\tau\in\H$. 
		Hence the induced bundle morphism $U\otimes\O_\H\to\mathscr{C}_\H(W)^\vee$ is surjective, with locally free kernel $\mathscr{K}$.
		Then dualizing identifies $\mathscr{C}_\H(W)$ with the annihilator of $\mathscr{K}$ in $U^\ast\otimes\O_\H$.
		
		If $\varphi\in\Gamma(\Omega,\mathscr{C}_\H(W))$, condition~\itemonlycref{lem:sectioncriterion:fiberwise-condition} follows from
		\eqref{eq:fiberid}, while condition~\itemonlycref{lem:sectioncriterion:holomorphicity-condition} and the last assertion follow from
		\eqref{eq:CintoWvee}.
		
		Conversely, assume the two conditions hold. Then $\varphi|_{U}$ is a holomorphic section of $U^\ast\otimes\O_\Omega$ by condition~\itemonlycref{lem:sectioncriterion:holomorphicity-condition}.
		Finally, condition~\itemonlycref{lem:sectioncriterion:fiberwise-condition} says that it annihilates $\mathscr{K}$. 
		Hence $\varphi \in \Gamma(\Omega,\mathscr{C}_\H(W))$. 
	\end{proof}
	
	By rationality and finite generation, $W$ is a finite direct sum of irreducible $V$-modules and hence is $C_2$-cofinite, so \cref{thm:conv} applies with an $L(0)$-homogeneous complement.
	More generally, by \cref{rmk:C20coinvfd}, the above discussion also applies to $C_{[2,0]}$-cofinite $W$, without assuming $C_2$-cofiniteness of $V$.
	
	We now apply \cref{lem:sectioncriterion} to the trace functions.
	Let $M$ be an irreducible $V$-module. Then for any $I\in \Ifusion[W][M][M]$, the associated trace function
	\begin{equation}\label{eq:sectiontrace}
		\varphi_I \colon \H \longrightarrow \mathscr{C}_\H(W),\quad \tau\mapsto \tr_{M}\ab(o_I(\placeholder) q_\tau^{L(0)-\frac{\ccharge}{24}})
	\end{equation}
	defines a global holomorphic section of $\mathscr{C}_\H(W)$.
	Indeed, \cref{thm:conv} supplies the conditions required by \cref{lem:sectioncriterion}.
	
	\subsubsection{Construction of the connection}
	
	We briefly recall the notion of horizontal sections of a vector bundle with connection; see \cite[Chapter~0, \S~11--12]{Sab08} for more details.
	Let $X$ be a smooth complex manifold and let $E\to X$ be a vector bundle,
	with sheaf of sections $\mathcal{E}$. A connection on $E$ is a $\C$-linear map
	\[
	\nabla \colon \mathcal{E}\to \mathcal{E}\otimes_{\mathcal{O}_X}\Omega^1_X
	\]
	satisfying the Leibniz rule $\nabla(fs)=f\nabla(s)+s\otimes df$. 
	Equivalently, for every vector field $\upxi\in\mathscr{T}_X$, a connection
	determines a $\C$-linear operator $\nabla_\upxi \colon \mathcal{E}\to \mathcal{E}$ such that
	\begin{equation}
		\label{eq:Leibniz}
		\nabla_\upxi(fs) = \upxi(f)s+f\nabla_\upxi(s)
	\end{equation}
	for any open subset $\Omega \subseteq X$, $f\in\Gamma(\Omega,\mathcal{O}_X)$, and $s\in\Gamma(\Omega,\mathcal{E})$. 
	A local section $s\in \Gamma(\Omega,\mathcal{E})$ is called \concept{horizontal} with respect to $\nabla$ if $\nabla_\upxi(s)=0$, for any $\upxi \in \Gamma(\Omega,\mathcal{T}_X)$.
	
	Now let $X = \H$, which is one-dimensional. Note that the vector field
	\[
	\upxi = \frac{1}{2\pi\iu}\odv{\tau} \in \mathscr{T}_\H
	\]
	is nowhere vanishing on $\H$. Hence $\mathscr{T}_\H = \O_\H \upxi$. 
	In particular, in order to define a connection $\nabla$ on the vector bundle $\mathscr{C}_\H(W)$ on $\H$, it suffices to define a $\C$-linear operator
	\[
	\nabla_\upxi \colon \mathscr{C}_\H(W) \longrightarrow \mathscr{C}_\H(W)
	\]
	satisfying the Leibniz rule \eqref{eq:Leibniz}.
	
	We now construct the connection explicitly. We first introduce some shorthand notation. Given a meromorphic function $f(z)$ with a possible pole at $0$ and an element $a \in V$, define the residue operator
	\begin{equation}\label{eq:Rf}
		\mathsf{R}_{f(z)}(a):=\Res_z Y_W[a,z]\iota(f(z)) \in \End(W).
	\end{equation}
	In particular, for any $a\in V$ and $u\in W$, we have
	\begin{align*}
		\mathsf{R}_1(a)u&=a[0]u,\\
		\mathsf{R}_{\wp_m(z,\tau)}(a)u&=a[-m]u+(-1)^m\sum_{k=1}^\infty \binom{2k-1}{m-1}G_{2k}(\tau) a[2k-m]u,\quad m \ge 2.
	\end{align*}

	Under the inclusion~\eqref{eq:CintoWvee}, $\varphi\in\Hdual{\mathscr{W}}(\Omega)$ belongs to $\Gamma(\Omega,\mathscr{C}_{\H}(W))$ if and only if
	\begin{equation}\label{eq:condforsec}
		\braket<\varphi(\tau)|\mathsf{R}_1(a)u>=0,\quad \mathrm{and}\quad \braket<\varphi(\tau)|\mathsf{R}_{\wp_m(z,\tau)}(a)u>=0,\quad m \ge 2,
	\end{equation}
	as identities of holomorphic functions on the open $\Omega \subset \H$, for every $a\in V$ and $u\in W$; see \eqref{eq:oneptcfb-characterization} and \cref{lem:sectioncriterion}. %

	In view of the differential equation of \cref{coro:diffeqfortrace} and the Laurent series expansion of $\wp_1(z,\tau)$, cf. \eqref{eq:Laurentwp1}, we introduce the notation
	\begin{equation}\label{eq:defTq}
		T(\tau) := \mathsf{R}_{\wp_1(z,\tau)}(\widetilde{\cfv})=L[-2]-\sum_{k \ge 1} G_{2k}(\tau) L[2k-2].
	\end{equation}
	Given a section of conformal blocks $\varphi \in \Gamma(\Omega, \mathscr{C}_\H(W))$, define $\nabla_{\upxi}~\varphi \in \Hdual{\mathscr{W}}(\Omega)$ by letting
	\begin{equation}\label{eq:defofconnectionvarphi}
		\braket<\nabla_{\upxi}~\varphi(\tau)|u> :=
		\upxi\braket<\varphi(\tau)|u> - 
		\frac{1}{(2\pi\iu)^2} \braket<\varphi(\tau)|T(\tau)u>, 
		\quad u\in W,
	\end{equation}
	where $\upxi=\frac{1}{2\pi\iu}\odv{\tau}$.
	It is easy to check that $\nabla_{\upxi}~\varphi$ satisfies the Leibniz rule \eqref{eq:Leibniz}. Since $\braket<\varphi(\tau)|T(\tau)u>$ is a finite sum of holomorphic functions on $\Omega$, it follows that $\tau\mapsto \braket<\nabla_{\upxi}~\varphi(\tau)|u>$ is holomorphic for every $u\in W$, and hence $\nabla_{\upxi}~\varphi\in\Hdual{\mathscr{W}}(\Omega)$.
	
	\begin{proposition}\label{prop:connection}
		For every open subset $\Omega \subseteq \H$, the operator defined by \eqref{eq:defofconnectionvarphi} satisfies 
		\[
		\nabla_\upxi \ab(\Gamma(\Omega, \mathscr{C}_\H(W))) \subseteq \Gamma(\Omega, \mathscr{C}_\H(W)).
		\]
		In particular, these operators are compatible with restriction to open subsets and define a holomorphic connection on $\mathscr{C}_{\H}(W)$. %
	\end{proposition}
	\begin{proof}
		In view of \eqref{eq:condforsec}, we need to show that $\braket<\nabla_{\upxi}~\varphi(\tau)|\mathsf{R}_{1}(a)u>=0$ and  $\braket<\nabla_{\upxi}~\varphi(\tau)|\mathsf{R}_{\wp_m(z,\tau)}(a)u>=0$, for all $a\in V, u\in W$, and $m \ge 2$, where $\tau\in \Omega$.
		Since $
		\braket<\varphi(\tau)|\mathsf{R}_{\wp_m(z,\tau)}(a)u>=0$ and $\mathsf{R}_{\wp_m(z,\tau)}(a)u$ is a vector-valued holomorphic function of $\tau$, differentiating this equation with respect to $\upxi$ gives
		\[
		\braket< \upxi\varphi(\tau)|\mathsf{R}_{\wp_m(z,\tau)}(a) u> = -\braket<\varphi(\tau)|\upxi(\mathsf{R}_{\wp_m(z,\tau)}(a))u>.
		\]
		Similarly, $\braket<\upxi\varphi(\tau)|\mathsf{R}_1(a)u>=0$ since $\upxi(1)=0$.
		Moreover, since $\varphi$ is a section of conformal blocks, we have $\braket<\varphi(\tau)|\mathsf{R}_{\wp_m(z,\tau)}(a)T(\tau)u>=0$ by \eqref{eq:oneptcfb-characterization}.
		It follows that
		\[
		\begin{aligned}
			\braket< \nabla_{\upxi}~\varphi(\tau) | \mathsf{R}_{\wp_m(z,\tau)}(a)u >
			&=
			\braket< \upxi\varphi(\tau) | \mathsf{R}_{\wp_m(z,\tau)}(a)u >  -
			\frac{1}{(2\pi\iu)^2}\braket< \varphi(\tau) | T(\tau)\mathsf{R}_{\wp_m(z,\tau)}(a)u > \\
			&=
			- \braket< \varphi(\tau) | \upxi(\mathsf{R}_{\wp_m(z,\tau)}(a))u >
			-\frac{1}{(2\pi\iu)^2}\braket< \varphi(\tau) | \ab[T(\tau),\mathsf{R}_{\wp_m(z,\tau)}(a)]u >\\
			&=
			-\braket< \varphi(\tau) | \ab(\mathsf{R}_{\upxi(\wp_m(z,\tau))}(a)+\frac{1}{(2\pi\iu)^2}\ab[T(\tau),\mathsf{R}_{\wp_m(z,\tau)}(a)])u >.
		\end{aligned}
		\]
		Similarly, $\braket<\nabla_{\upxi}~\varphi(\tau)| \mathsf{R}_{1}(a)u>=-\frac{1}{(2\pi\iu)^2}\braket<\varphi(\tau)|\ab[T(\tau),\mathsf{R}_{1}(a)]u>$.
		Thus it is enough to prove that
		\[
		\mathsf{R}_{\upxi(\wp_m(z,\tau))}(a)+\frac{1}{(2\pi\iu)^2}\ab[T(\tau),\mathsf{R}_{\wp_m(z,\tau)}(a)] \quad\text{and}\quad \ab[T(\tau),\mathsf{R}_1(a)]
		\]
		can be written as sums of elements of the form $\mathsf{R}_f(b)$ for some $b\in V$ and $f\in H^0(E_\tau,\mathcal{O}_{E_\tau}^{\mathrm{an}}(\ast\mpt{p}))$.
		
		Let $g(x) = \sum g_nx^n\in \C\dparen{x}$ and $a,b \in V$. We claim that
		\begin{equation}\label{eq:commutatorRes}
			\Res_x \ab \Big [ Y_W[b,x],Y_W[a,z] ] g(x) = 
			\sum_{i \ge 0} \frac{1}{i!}(\odv{z})^i (g(z)) Y_W[b[i]a,z]
		\end{equation}
		as an identity in $\End(W)\dbrack{z,z^{-1}}$. Indeed, since $(V,Y[\placeholder,z],\vac,\widetilde{\cfv})$ is a VOA and $(W,Y_W[\placeholder,z])$ is a module over it (cf. \S \ref{subsec:square-bracket-voa-and-modules}), the commutator formula for modules \cite{FHL93,LL04} applies to the square-bracket modes and gives
		\[
		\ab \Big [b[n],Y_W[a,z]] = 
		\sum_{i\ge 0}\binom{n}{i}z^{\,n-i}\,Y_W[b[i]a,z],\qquad n\in\Z.
		\]
		Since $\Res_x Y_W[b,x]g(x)=\sum g_n\,b[n]$, we obtain
		\[
		\Res_x \ab \Big [ Y_W[b,x],Y_W[a,z] ] g(x) = 
		\sum g_n\ab[b[n],Y_W[a,z]] = 
		\sum_{i\ge 0}\ab(\sum g_n\binom{n}{i}z^{\,n-i})Y_W[b[i]a,z],
		\]
		and $\sum_{n}g_n\binom{n}{i}z^{n-i}=\frac{1}{i!}(\odv{z})^i g(z)$, which proves \eqref{eq:commutatorRes}. 
		
		Applying \eqref{eq:commutatorRes} to $b=\widetilde{\cfv}$ and to $g(x)=\iota_x(\wp_1(x,\tau))$, using \eqref{eq:Rf} and \eqref{eq:defTq}, we have
		\begin{align*}
			\ab \Big [T(\tau),Y_W[a,z]] 
			&=
			\Res_x \ab \Big [Y_W[\widetilde{\cfv},x],Y_W[a,z]] 
			\iota_x(\wp_1(x,\tau)) \\
			&=
			\sum_{i \ge 0} \frac{1}{i!} (\odv{z})^i(\iota_z(\wp_1(z,\tau))) 
			Y_W[\widetilde{\cfv}[i]a,z] \\
			&=
			Y_W[L[-1]a,z]\iota_z(\wp_1(z,\tau)) + 
			\sum_{j \ge 0} \iota_z\ab(\frac{1}{(j+1)!}
			(\odv{z})^{j+1}\wp_1(z,\tau))Y_W[L[j]a,z] \\
			&=
			\pdv{z}\ab(Y_W[a,z])\iota_z(\wp_1(z,\tau)) + 
			\sum_{j \ge 0}(-1)^{j+1}\iota_z\ab(\wp_{j+2}(z,\tau))Y_W[L[j]a,z].
		\end{align*}
		Applying $\Res_z(\placeholder)\iota_{z}(\wp_m(z,\tau))$ to the equation above, using \eqref{eq:Rf} and the identity
		\[
		\Res_z \pdv{z}(f(z))g(z) = -\Res_zf(z) \pdv{z}(g(z)),
		\]
		we obtain
		\begin{equation}\label{eq:Tqbracket}
			\ab[T(\tau),\mathsf{R}_{\wp_m(z,\tau)}(a)] =
			- \mathsf{R}_{\pdv{z}(\wp_m(z,\tau)\wp_1(z,\tau))}(a) +
			\sum_{j \ge 0}(-1)^{j+1}\mathsf{R}_{\wp_m(z,\tau)\wp_{j+2}(z,\tau)}(L[j]a).
		\end{equation}
		Since $\mathsf{R}_{f}(a)+\mathsf{R}_g(a)=\mathsf{R}_{f+g}(a)$, we have
		\begin{align*}
			\MoveEqLeft 
			\mathsf{R}_{\upxi(\wp_m(z,\tau))}(a) + 
			\frac{1}{(2\pi\iu)^2}\ab\Big[T(\tau),\mathsf{R}_{\wp_m(z,\tau)}(a)] \\
			&=
			\mathsf{R}_{\upxi\wp_m(z,\tau)-\frac{1}{(2\pi\iu)^2}\frac{\partial }{\partial z}(\wp_m(z,\tau)\wp_1(z,\tau))}(a) +
			\frac{1}{(2\pi\iu)^2}\sum_{j \ge 0}(-1)^{j+1}\mathsf{R}_{\wp_m(z,\tau)\wp_{j+2}(z,\tau)}(L[j]a).
		\end{align*}
		For each $j \ge 0$ and $m \ge 2$, the function $\wp_m(z,\tau)\wp_{j+2}(z,\tau)$ is meromorphic in $z$ with only possible pole at $z=0$, and is doubly periodic in $z$ with periods $1$ and $\tau$. Hence
		\[
		\wp_m(z,\tau)\wp_{j+2}(z,\tau) \in H^0(E_\tau,\mathcal{O}_{E_\tau}^{\mathrm{an}}(\ast\mpt{p})).
		\]
		Moreover, by \cref{lem:wpproperty},
		\[
		\upxi\wp_m(z,\tau)-\frac{1}{(2\pi\iu)^2}\pdv{z}\ab(\wp_m(z,\tau)\wp_1(z,\tau)) \in H^0(E_\tau,\mathcal{O}_{E_\tau}^{\mathrm{an}}(\ast\mpt{p})).
		\]
		By \eqref{eq:H0sectionidentification} and \eqref{eq:oneptcfb-characterization}, $\varphi(\tau)$ annihilates $\mathsf{R}_f(b)W$ for every $f \in H^0(E_\tau,\mathcal{O}_{E_\tau}^{\mathrm{an}}(\ast\mpt{p}))$ and $b \in V$. Thus we have
		\[
		\braket<\nabla_{\upxi}~\varphi(\tau)|\mathsf{R}_{\wp_m(z,\tau)}(a)u> = 0,
		\quad m \ge 2.
		\]
		Finally, replacing $\wp_m$ by $1$ in \eqref{eq:Tqbracket}, we obtain
		\[
		\ab[T(\tau),\mathsf{R}_1(a)] = 
		\mathsf{R}_{\wp_2(z,\tau)}(a) + 
		\sum_{j \ge 0}(-1)^{j+1}\mathsf{R}_{\wp_{j+2}(z,\tau)}(L[j]a),
		\]
		Here $\wp_2(z,\tau),\wp_{j+2}(z,\tau)\in H^0(E_\tau,\mathcal{O}_{E_\tau}^{\mathrm{an}}(\ast\mpt{p}))$ for all $j \ge 0$. This shows $\braket<\nabla_{\upxi}~\varphi(\tau)|\mathsf{R}_1(a)u>=0$. Hence we have a connection $\nabla$ on the vector bundle $\mathscr{C}_\H(W)$. Since $\H$ is a one-dimensional complex manifold, $\nabla$ is automatically flat; see \cite[Chapter~0, Remark~12.5]{Sab08}.
	\end{proof}
	
	\begin{corollary}\label{coro:horizontalsection}
		Let $M$ be an ordinary $V$-module of conformal weight $h_M$ and $I\in \Ifusion[W][M][M]$. Then the trace section $\varphi_I\in \Gamma(\H,\mathscr{C}_\H(W))$ in \eqref{eq:sectiontrace} is horizontal with respect to the connection $\nabla$ defined by \eqref{eq:defofconnectionvarphi}.
	\end{corollary}
	\begin{proof}
		Let $I$ be an intertwining operator of type $\fusion[W][M][M]$ and write $\alpha=h_M-\frac{\ccharge}{24}$. By \cref{thm:conv}, for every $u\in W$ the normalized series $f_u(q)=q^{-\alpha}\mathcal{Z}_I(u,q)\in\C\dbrack{q}$ converges absolutely and locally uniformly on $\D$, and $\braket<\varphi_I(\tau)|u>=\e^{2\pi\iu\tau\alpha}f_u(q_\tau)$. The formal identity of \cref{coro:diffeqfortrace} is an identity in $\C\dbrack{q}q^{\alpha}$; dividing by $q^{\alpha}$ it becomes an identity of formal power series
		\[
		(2\pi\iu)^2\ab(q\odv{q}+\alpha)f_u(q) = 
		f_{L[-2]u}(q) - 
		\sum_{k \ge 1}\widetilde{G}_{2k}(q)\,f_{L[2k-2]u}(q),
		\]
		in which the sum over $k$ is finite and every series involved converges on $\D$. By \eqref{eq:formalG2kq}, evaluating $\widetilde{G}_{2k}(q)$ at $q = q_\tau$ and multiplying by $\frac{1}{(2\pi\iu)^2}\e^{2\pi\iu\tau\alpha}$ yields the following identity of holomorphic functions on $\H$
		\begin{align*}
			\upxi\braket<\varphi_I(\tau)|u>
			&= 
			\frac{1}{(2\pi\iu)^2}\braket<\varphi_I(\tau)|L[-2]u> - 
			\frac{1}{(2\pi\iu)^2}\sum_{k \ge 1} G_{2k}(\tau) \braket<\varphi_I(\tau)|L[2k-2]u>\\
			&= 
			\frac{1}{(2\pi\iu)^2}\braket<\varphi_I(\tau)|\ab(L[-2] - \sum_{k \ge 1}G_{2k}(\tau)L[2k-2])u>
			= \frac{1}{(2\pi\iu)^2} \braket<\varphi_I(\tau)|T(\tau)u>,
		\end{align*}
		where we used \eqref{eq:defTq}. Then by  \eqref{eq:defofconnectionvarphi},  for any $\tau\in \H$ and $u\in W$, we have
		\begin{align*}
			\braket<\nabla_{\upxi}~\varphi_I(\tau)|u> = 
			\upxi\braket<\varphi_I(\tau)|u> - 
			\frac{1}{(2\pi\iu)^2} \braket<\varphi_I(\tau)|T(\tau)u> = 0.
		\end{align*}
		Hence $\nabla_{\upxi}~\varphi_I = 0$. 
		Since $\upxi$ generates the tangent bundle $\mathscr{T}_\H$, it follows that $\varphi_I$ is a horizontal section of $\mathscr{C}_\H(W)$.
	\end{proof}

	\subsubsection{Comparison with the DGT connection}
	
	In \cite{DGT21}, one uses Virasoro uniformization to give the coinvariant sheaf on the moduli space a projectively flat connection. 
	We write $\nabla^{\mathrm{DGT}}$ for the induced projective connection on conformal blocks by duality.
	
	By \cref{prop:original-square-realizations} and \eqref{eq:square-module-conjugation}, the transport maps $\mathcal{U}_{\cdots}$ intertwine the Virasoro action used in the Virasoro uniformization; by \cref{prop:square-transport:module-grading}, they also intertwine the $\C^\times$-action used in descent. Hence they identify the DGT connections on both the original and square-bracket realizations of $\mathscr{C}_\H(W)$.
	\emph{In what follows, we only consider $\nabla^{\mathrm{DGT}}$ in the square-bracket realization.}
	
	\begin{theorem}\label{thm:DGT-comparison}
		The pullback of $\nabla^{\mathrm{DGT}}$ along $\pi \colon \H \longrightarrow \mathcal{M}_{1,1}^{\mathrm{an}}$ is precisely the projectivization of the connection $\nabla$ on $\mathscr{C}_{\H}(W)$ defined by \eqref{eq:defofconnectionvarphi}.
	\end{theorem}

	\begin{proof}
		We describe the DGT connection in terms of infinitisimal deformations following the convention of \S~\ref{subsec:square-bracket-voa-and-modules}.
		First, consider the universal elliptic curve
		\[
		\mathcal{E}_{\H} := (\C\times\H)/\Z^2\longrightarrow\H, \qquad (m,n)\cdot(z,\tau) = (z+m+n\tau,\tau),
		\]
		with marked section $\mPt{P}$ induced by $z=0$. 
		As discussed before, we only need to consider the vector field $\upxi=(2\pi\iu)^{-1}\odv{\tau}$ on $\H$ and view it as a deformation direction.
		
		To describe the deformation, we compare two lifts of $\upxi$ to the total space of the universal elliptic curve. 
		The first is the coordinate-preserving lift, which we denote by $\upxi$ again. 
		On the other hand, consider the following vector field
		\[
		X := \frac{1}{2\pi\iu}\pdv{\tau} - \frac{\wp_1(z,\tau)}{(2\pi\iu)^2}\pdv{z}
		\]
		on the covering space $\Set*{(z,\tau)\in\C\times\H \given z\notin\Z+\tau\Z}$. 
		By \eqref{eq:quasiwp1}, this vector field  is invariant under the lattice action and hence gives a lift of $\upxi$ on $\mathcal{E}_{\H}\setminus\mPt{P}$.
		On the punctured formal disc, their difference is
		\[
		\ell_\upxi := \upxi-X = \frac{\iota_z\wp_1(z,\tau)}{(2\pi\iu)^2}\pdv{z}.
		\]
		This difference describes the infinitesimal change in the gluing of the disc to the punctured curve. 
		Thus, in the right-action convention of Virasoro uniformization of \cite[Theorem~3.1 and \S~3.2.1]{DGT21}, the anchor of it is precisely the vector field $\upxi$.
		
		As in \cite{DGT21}, the Virasoro generator $L_n$ projects to the vector field $-z^{n+1}\pdv{z}$. Hence, the Laurent field $\ell_\upxi$ acts on $W$ as $(2\pi\iu)^{-2}\Res_z\iota_z\wp_1(z,\tau)Y_W[\widetilde{\cfv},z]=(2\pi\iu)^{-2}T(\tau)$, up to a scalar coming from its Virasoro lift. 
		Together with the anchor identity $a(\ell_\upxi)=\upxi$, this gives, by \cite[(32)]{DGT21} and its $\O_\H$-linear dependence on the Virasoro argument,
		\[
		\begin{aligned}
			\ell_\upxi\diamond(u\otimes g)
			&= \alpha_W(\ell_\upxi)u\otimes g+u\otimes a(\ell_\upxi)(g)\\
			&= \frac{1}{(2\pi\iu)^2}T(\tau)u\otimes g+u\otimes\upxi(g),
		\end{aligned}
		\]
		where $u\in W$ and $g$ is a section of $\mathscr{O}_\H$.
		By \cite[Lemma~7.4]{DGT21}, this action respects the chiral Lie
		algebra. Then the induced operator on coinvariants is, by the
		construction in the proof of \cite[Theorem~7.1]{DGT21}, a representative
		of the covariant derivative in the direction $\upxi$. Hence, its dual on conformal blocks is a representative of $\nabla^{\mathrm{DGT}}_{\upxi}$.
		This dual operator agrees with \eqref{eq:defofconnectionvarphi} up to a scalar.
	\end{proof}
	
	\subsection{The trace-frame theorem}\label{subsec:trace-frame}
	
	In the rest of this section, we assume in addition that $V$ is simple and fix a complete set $\Irr=\Set*{V=M^0,M^1,\cdots,M^N}$ of irreducible ordinary $V$-modules up to isomorphism. Since $W$ is a finite direct sum of irreducible modules and the constructions below are additive in $W$, we may assume $W\in\Irr$.
	
	We now use the connection $\nabla$ on $\mathscr{C}_\H(W)$, together with the \nameref{thm:factorization} and \nameref{coro:VB}, to show that the trace functions associated to intertwining operators form a basis of $\mathscr{C}(E_\tau,\mpt{p},z,W)$.
	
	We need the following standard fact about semisimple algebras.
	\begin{lemma}\label{lem:ssrep}
		Let $A$ be a semisimple algebra, $M$ a left $A$-module, and $S$ a simple left $A$-module. If $\Fam{g_i}{1\le i\le n}$ is a basis of the vector space $\Hom_A(M,S)$, then the linear map
		\[
		g = \oplus_{i=1}^n g_i \colon M \longrightarrow nS (:= S^{\oplus n}),\quad g(u) = (g_1(u),\cdots, g_n(u)),
		\]
		where $u \in M$, is surjective.
	\end{lemma}
	
	For each $k$, let $\Fam{I^k_j}{1\le j\le n_k}$ be a basis of $\Ifusion[W][M^k][M^k]$. Given $\tau\in \H$, for each intertwining operator $I^k_j$, we have an associated one-point conformal block on the torus $E_\tau$, seeing \eqref{eq:defofvarphitau}:
	\begin{equation}\label{eq:varphiIjk}
		\varphi_{I^k_j}(\tau) \colon [W]_{(E_\tau, \mpt{p},z)} \longrightarrow \C,\quad u \mapsto \tr_{M^k}\ab(o_{I^k_j}(u) q_\tau^{L(0)-\frac{\ccharge}{24}}).
	\end{equation}
	Each $\varphi_{I_j^k}$ is a horizontal section of the bundle $\mathscr{C}_\H(W)$ with respect to the connection $\nabla$ in \eqref{eq:defofconnectionvarphi}; see \cref{coro:horizontalsection}.

	\begin{theorem}\label{thm:basiscfb}
		For each $\tau\in\H$, with $[E_\tau,\mpt{p}]\in\mathcal{M}_{1,1}^{\mathrm{an}}$ the associated pointed elliptic curve and a choice of local coordinate $z$ around $\mpt{p}$, the family $\Fam{\varphi_{I^k_j}(\tau)}{\substack{0\le k\le N\\1\le j\le n_k}}$ is a basis of the vector space of one-point conformal blocks $\mathscr{C}(E_\tau,\mpt{p},z,W)$.
	\end{theorem}
	
	\begin{proof}
		We divide our proof into four steps.
		
		\proofstep{1}{Independence of the leading trace functionals.}
		We first show that the restrictions of the trace functions to the top degrees $M^k(0)$ of the irreducible modules $M^k$, or equivalently the leading terms in the series $\varphi_{I_j^k}(\tau)$ in \eqref{eq:varphiIjk},
		\[
		\Fam{\tr_{M^k(0)}\ab(o_{I^k_j}(\placeholder))}{\substack{0\le k\le N\\1\le j\le n_k}} \colon W \longrightarrow \C,
		\]
		are linearly independent as elements in $W^\ast$.
		
		Since $V$ is rational, the fusion-rules theorem \cite{FZ92,Li99,Liu23,GLZ25,Liu26} gives the following isomorphism:
		\begin{equation}\label{eq:fusionrulesthm}
			\Ifusion[W][M^j][M^k] \cong \ab(M^k(0)^\ast\otimes_{A(V)}A(W)\otimes _{A(V)}M^j(0))^\ast \cong \Hom_{A(V)}(A(W)\otimes_{A(V)}M^j(0),M^k(0)),
		\end{equation}
		where $A(W)=W/O(W)$ is an $A(V)$-bimodule. Moreover, under the isomorphism \eqref{eq:fusionrulesthm}, $I\in \Ifusion[W][M^j][M^k]$ corresponds to $\phi_I \in \Hom_{A(V)}(A(W)\otimes_{A(V)}M^j(0),M^k(0))$, which is uniquely determined by $o_I(\placeholder) \colon A(W) \longrightarrow \Hom_\C(M^j(0),M^k(0))$ through the formula
		\begin{equation}\label{eq:fusionrulesfI}
			\braket<u^3|\phi_I([u^1]\otimes u^2)> = \braket<u^3|o_I(u^1)u^2>,\quad u^1\in W,\ u^2\in M^j(0),\ u^3\in M^k(0)^\ast,
		\end{equation}
		where $o_I(\placeholder)$ is given by \eqref{eq:oI}.
		
		The algebra $A(V)$ is semisimple by \cite[Theorem~2.2.1]{Zhu96}. Hence the category of $A(V)$-bimodules is also semisimple, with simple objects given by
		\[
		\Hom_\C(M^j(0),M^k(0)) \cong M^k(0)\otimes M^j(0)^\ast,\quad 
		0 \le j,k \le N.
		\]
		
		As a semisimple $A(V)$-bimodule, $A(W)$ has the following direct sum decomposition:
		\begin{equation}\label{decAW}
			A(W) \cong \bigoplus_{j,k=0}^N \Nfusion[W][M^j][M^k] M^k(0)\otimes  M^j(0)^\ast \cong \bigoplus_{j,k=0}^N \Nfusion[W][M^j][M^k] \Hom_\C(M^j(0),M^k(0)).
		\end{equation}
		For any $I\in \Ifusion[W][M^k][M^k]$, the corresponding map
		$o_I \colon A(W) \longrightarrow \End_\C (M^k(0))$
		is an $A(V)$-bimodule homomorphism since $o_I([a]\ast [u])=o([a]) o_I([u])$ and $o_I([u]\ast [a])=o_I([u])o([a])$ \cite{FZ92,Liu23}.
		Moreover, in view of \eqref{eq:fusionrulesfI}, we can rewrite \eqref{eq:fusionrulesthm} as
		\[
		\Ifusion[W][M^k][M^k] \cong \Hom_{A(V)-A(V)}\ab(A(W),\End_{\C}(M^k(0))),\quad I\mapsto o_I(\placeholder),
		\]
		Now the basis $\Fam{I^k_j}{1\le j\le n_k}$ of the vector space $\Ifusion[W][M^k][M^k]$ for $0 \le k \le N$ corresponds to the basis $\Fam{o_{I^k_j}(\placeholder)}{1\le j\le n_k}$ of the Hom-space of $A(V)$-bimodules.
		
		Since the category of $A(V)$-bimodules is equivalent to the category of left modules over the semisimple algebra $A(V)\otimes A(V)^{\mathrm{op}}$, we may apply \cref{lem:ssrep} to $M=A(W)$ and $S=\End_\C(M^k(0))$, and conclude that  $o^k=\oplus_{j=1}^{n_k}o_{I^k_j} \colon A(W) \longrightarrow \Nfusion[W][M^k][M^k] \End_\C(M^k(0))$ is surjective. Moreover, by Schur's Lemma, $o^k$ vanishes on every simple bimodule summand of $A(W)$ not isomorphic to  $ \End_\C(M^k(0))$ in the decomposition~\eqref{decAW}. Thus, the direct sum of $o^k$ is a surjective linear map:
		\[
		o = \oplus_{k=0}^N\oplus_{j=1}^{n_k} o_{I^k_j} \colon A(W) \longrightarrow \bigoplus_{k=0}^N  \Nfusion[W][M^k][M^k] \End_\C (M^k(0)).
		\]
		
		For each fixed pair $(k,j)$, there exists $[u]\in A(W)$, represented by some $u\in W$, such that
		\[
		o_{I^k_j}(u) = \Id_{M^k(0)},\quad o_{I^{k'}_{j'}}(u)=0,\quad \mathrm{for}\quad (k',j')\neq (k,j).
		\]
		Then $\tr_{M^k(0)}\ab(o_{I^k_j}(u))=\dim M^k(0)\neq 0$ and $\tr_{M^k(0)}\ab(o_{I^{k'}_{j'}}(u))=0$ for $(k',j')\neq (k,j)$.
		It follows that the family of functionals $\Fam{\tr_{M^k(0)}\ab(o_{I^k_j}(\placeholder))}{\substack{0\le k\le N\\1\le j\le n_k}}$ is linearly independent.
		
		\proofstep{2}{Independence near the cusp.} We show that there exists an open subset $\Omega\subset \H$ such that, for every $\tau\in \Omega$, the family $\Fam{\varphi_{I^k_j}(\tau)}{\substack{0\le k\le N\\1\le j\le n_k}}$ is linearly independent in $\mathscr{C}(E_\tau,\mpt{p},z,W)$.
		
		Indeed, let $K=\sum_{i=0}^N n_i=\#\Set{(k,j) \given 0 \le k \le N, 1 \le j \le n_k}$. Impose a lexicographic order on this index set. Since the family $\Fam{\tr_{M^k(0)}\ab(o_{I^k_j}(\placeholder))}{\substack{0\le k\le N\\1\le j\le n_k}}$ is linearly independent, for each $(k',j')$ in the index set there exists $u^{(k',j')}\in W$ such that the matrix \[A(0)=\ab(\tr_{M^k(0)}\ab(o_{I^k_j}(u^{(k',j')})))_{K\times K}\]
		is invertible. Define a matrix in the variable $q$ as follows:
		\[
		A(q) = \ab(\sum_{n=0}^\infty \tr_{M^k(n)}\ab(o_{I^k_j}(u^{(k',j')})) q^n)_{K\times K}.
		\]
		For $\tau\in\H$, evaluation at $q=q_\tau$ gives
		\[
		A(q_\tau) = \ab( \braket<q_\tau^{\frac{\ccharge}{24}-h_{M^k}}\cdot \varphi_{I_j^k}(\tau)|u^{(k',j')}>)_{K\times K}.
		\]
		By \cref{thm:conv}, $A(q)$ is holomorphic on $\D$. Since $A(0)$ is invertible, so is $A(q)$ for all sufficiently small $q$. Thus, for $\operatorname{Im}\tau$ sufficiently large, the normalized trace functionals have an invertible evaluation matrix and are linearly independent. Their normalization factors $q_\tau^{\frac{\ccharge}{24}-h_{M^k}}$ are nonzero, so the original trace functionals are also linearly independent. This gives the required nonempty open subset $\Omega\subset\H$.
		
		\proofstep{3}{Propagation by horizontality.} We now show that, for every $\tau\in \H$, the family $\Fam{\varphi_{I^k_j}(\tau)}{\substack{0\le k\le N\\1\le j\le n_k}}$ is linearly independent in the vector space $\mathscr{C}(E_\tau,\mpt{p},z,W)$.
		
		Since $V$ is $C_2$-cofinite, $W$ is $C_2$-cofinite. So \cref{coro:horizontalsection} applies and the sections $\varphi_{I_j^k}$ are therefore horizontal. By \cite[Chapter~0, Theorem~12.8]{Sab08}, the horizontal sections of $\nabla$ form a local system. Hence the pointwise rank of this family is locally constant, and thus constant on the connected space $\H$. Step~2 shows that this rank is $K$ on $\Omega$, so the family is linearly independent in every fiber.

		\proofstep{4}{Dimension count and spanning.} It remains to show that the dimension of each fiber $\mathscr{C}(E_\tau,\mpt{p},z,W)$ is $K$.
		For this, notice that \cref{prop:original-square-realizations} identifies $V^\square$ and $V$, and hence they share the same hypotheses. Furthermore, \cref{prop:square-transport} identifies the irreducible modules, their contragredients, and the fusion rules for the two inputs.
		Applying the genus-one, one-point case of \cref{thm:Verlinde} gives
		\begin{equation}
			\dim\mathscr{C}(E_\tau,\mpt{p},z,W) = \sum_{k=0}^N\Nfusion[W][M^k][M^k] = \sum_{k=0}^N n_k = K.
		\end{equation}
		Hence the trace family is a basis of the vector space $\mathscr{C}(E_\tau,\mpt{p},z,W)$ for any $\tau\in \H$.
	\end{proof}
	
	\begin{remark}\label{rmk:C20-trace-frame}
		Without assuming $V$ is $C_2$-cofinite, the conclusion of \cref{thm:basiscfb} also holds for $C_{[2,0]}$-cofinite $W$.
		Indeed, \cref{rmk:C20coinvfd} implies coherence of the genus-one coinvariant sheaf. \cref{rmk:C20-convergence} and the proof of \cref{coro:horizontalsection} supply convergence and horizontality in Steps~2--3. Finally, \nameref{coro:VB} gives constant rank; nodal factorization and \cref{prop:threepointcfbIO} compute this rank as $\sum_k n_k=K$, completing Step~4.
	\end{remark}
	
	Combining the fiberwise basis with horizontality gives the following intrinsic form of the trace-frame theorem.
	Here, we write
	\begin{equation}\label{eq:CW}
		\Gamma^\nabla(\H,\mathscr{C}_\H(W))
		:=\Set*{\varphi\in\Gamma(\H,\mathscr{C}_\H(W))\given\nabla\varphi=0}
	\end{equation}
	for the space of global horizontal sections.
	
	\begin{corollary}\label{coro:frame}
		The trace map
		\begin{equation}\label{eq:trace-frame-map}
			\varPhi\colon
			\ab(\bigoplus_{k=0}^N\Ifusion[W][M^k][M^k])\otimes_{\C}\O_\H
			\isomto\mathscr{C}_\H(W),
			\qquad I\otimes f\longmapsto f\varphi_I,
		\end{equation}
		is an isomorphism of holomorphic vector bundles with flat connection, where the source carries the trivial connection. 
		In particular, the family $\Fam{\varphi_{I_j^k}}{\substack{0\le k\le N\\1\le j\le n_k}}$ is a global horizontal frame, and $I\mapsto\varphi_I$ identifies
		\[
		\bigoplus_{k=0}^N\Ifusion[W][M^k][M^k]
		\cong
		\Gamma^\nabla(\H,\mathscr{C}_\H(W)).
		\]
	\end{corollary}
	\begin{proof}
		The trace construction is linear in $I$. In the chosen intertwining-operator bases, $\varPhi$ sends $I_j^k\otimes1$ to $\varphi_{I_j^k}$. \cref{thm:basiscfb} therefore makes $\varPhi$ an isomorphism on every fiber, hence an isomorphism of holomorphic vector bundles. By \cref{coro:horizontalsection} and the Leibniz rule,
		\[
		\nabla(f\varphi_I)=df\otimes\varphi_I.
		\]
		Thus $\varPhi$ intertwines the trivial connection with $\nabla$. A section written in this frame is horizontal precisely when all its coefficients are locally constant, hence constant on the connected space $\H$. This proves the last assertion.
	\end{proof}
	
	\subsection{Comparison with Miyamoto's one-point functions}\label{subsec:miyamoto-comparison}
	
	We compare the global horizontal sections of $\mathscr{C}_\H(W)$ with Miyamoto's one-point functions in \cite{Mi00}.
	Following \cref{rmk:C20-trace-frame}, we assume $C_{[2,0]}$-cofiniteness of $W$ and do not assume $C_2$-cofiniteness of $V$. 
	Miyamoto uses the conventions
	\[
	Y_{\mathrm{Mi}}[\placeholder,z] = Y(\e^{zL(0)}\placeholder,\e^{z}-1),
	\qquad
	\widetilde{\cfv}_{\mathrm{Mi}} = \cfv-\frac{\ccharge}{24}\vac,
	\qquad
	E_{2k}(\tau) = (2\pi\iu)^{-2k}G_{2k}(\tau),
	\]
	whereas we use $Y[\placeholder,z] = Y(\e^{2\pi\iu zL(0)}\placeholder,\e^{2\pi\iu z}-1)$ and $\widetilde{\cfv}=(2\pi\iu)^2\widetilde{\cfv}_{\mathrm{Mi}}$.
	We write $a[n]_{\mathrm{Mi}}$ and $L[n]_{\mathrm{Mi}}$ for the corresponding modes. Comparing coefficients gives
	\begin{equation}\label{eq:Miyamoto-normalization}
		a[n] = (2\pi\iu)^{-n-1}a[n]_{\mathrm{Mi}}, \qquad
		L[n] = (2\pi\iu)^{-n}L[n]_{\mathrm{Mi}}.
	\end{equation}
	In Miyamoto's notation, the subspace in \cref{rmk:C20coinvfd} becomes
	\[
	C_{[2,0]}(W)
	=\operatorname{span}\Set*{a[-2]_{\mathrm{Mi}}u,\ a[0]_{\mathrm{Mi}}u
		\given a\in V,\ u\in W}.
	\]

	For $M\in\Irr$ and $I\in \Ifusion[W][M][M]$, set
	\[
	S^I(u,\tau):=\braket<\varphi_I(\tau)|u> = \tr_{M}\ab(o_I(u)q_\tau^{L(0)-\frac{\ccharge}{24}}).
	\]
	The $C_{[2,0]}$-cofinite version of \cref{thm:conv}, explained in \cref{rmk:C20-convergence}, ensures convergence of this trace series and holomorphy of $S^I(u,\placeholder)$ on $\H$ for each $u\in W$.
	For homogeneous $u\in W$, the mode $u_{(\deg u - 1)}$ defining $o_I(u)$ in~\eqref{eq:oI} is the mode $u^I(h_W + \deg u - 1) = u^I(\wt u-1)$ in Miyamoto's unshifted mode convention. Thus it is exactly his zero mode $o^I(u)$, so $S^I$ is exactly Miyamoto's trace function.
	
	\begin{theorem}\label{thm:Miyamoto-comparison}
		Let $S\colon W\times\H\to\C$ be a map such that $S(\placeholder,\tau)\colon W\to\C$ is $\C$-linear for each $\tau\in\H$, and define $\braket<\varphi(\tau)|u>:=S(u,\tau)$. The extension
		\[
		S(f\otimes u,\tau):=f(\tau)S(u,\tau),\qquad f\in\C[E_4,E_6],
		\]
		satisfies conditions~\textnormal{(C1)}--\textnormal{(C4)} of \cite[Definition~4.2]{Mi00}, extended additively to $W$, if and only if $\varphi$ is a global horizontal section of $\mathscr{C}_\H(W)$.
		Consequently, when these conditions hold, there are unique constants $c_{k,j}\in\C$ such that
		\[
		S(u,\tau)=\sum_{k=0}^N\sum_{j=1}^{n_k}c_{k,j}S^{I_j^k}(u,\tau).
		\]
	\end{theorem}
	\begin{proof}
		Suppose first that $\varphi$ is a global horizontal section. Condition~\textnormal{(C1)} follows from \eqref{eq:CintoWvee}, and condition~\textnormal{(C2)} holds by the definition of the extension. The fiber relations \eqref{eq:oneptcfb-characterization}, expressed in Miyamoto's normalization with $m=2$, give
		\[
		S(a[0]_{\mathrm{Mi}}u,\tau)=0,
		\qquad
		S(a[-2]_{\mathrm{Mi}}u,\tau)
		+\sum_{k\ge2}(2k-1)E_{2k}(\tau)
		S(a[2k-2]_{\mathrm{Mi}}u,\tau)=0.
		\]
		Thus $S$ vanishes on the generators of Miyamoto's $O_q(W)$, giving condition~\textnormal{(C3)}.
		
		For any section $\varphi$ of $\mathscr{C}_{\H}(W)$, \eqref{eq:defTq} and~\eqref{eq:defofconnectionvarphi}, together with the mode conversion~\eqref{eq:Miyamoto-normalization} and the Eisenstein-series convention above, give
		\[
		\braket<\nabla_{\upxi}~\varphi(\tau)|u> = 
		\upxi\braket<\varphi(\tau)|u> -\braket<\varphi(\tau)|
		\ab(L[-2]_{\mathrm{Mi}}-\sum_{k\ge1}E_{2k}(\tau)L[2k-2]_{\mathrm{Mi}})u>.
		\]
		Thus the horizontality equation $\nabla\varphi=0$ is equivalent to
		\begin{equation}\label{eq:Miyamoto-C4}
			\upxi\braket<\varphi|u>
			=\braket<\varphi|L[-2]_{\mathrm{Mi}}u>
			-\sum_{k\ge1}E_{2k}\braket<\varphi|L[2k-2]_{\mathrm{Mi}}u>.
		\end{equation}
		If $L[0]_{\mathrm{Mi}}u=ru$, rearranging~\eqref{eq:Miyamoto-C4} gives
		\[
		\braket<\varphi|L[-2]_{\mathrm{Mi}}u>
		=\ab(\upxi+rE_2)\braket<\varphi|u>
		+\sum_{k\ge2}E_{2k}\braket<\varphi|L[2k-2]_{\mathrm{Mi}}u>,
		\]
		which is exactly condition~\textnormal{(C4)}.
		
		Conversely, suppose the extension of $S$ satisfies \textnormal{(C1)}--\textnormal{(C4)}. Condition~\textnormal{(C3)} gives the $a[0]_{\mathrm{Mi}}$ relation and the $m=2$ relation displayed above. The translation identity gives
		\[
		(L[-1]_{\mathrm{Mi}}^r a)[n]_{\mathrm{Mi}}
		=(-1)^r n(n-1)\cdots(n-r+1)a[n-r]_{\mathrm{Mi}}.
		\]
		For $m\ge2$, apply the $m=2$ relation to $L[-1]_{\mathrm{Mi}}^{m-2}a$ and divide by $(m-1)!$. This yields
		\[
		S(a[-m]_{\mathrm{Mi}}u,\tau)
		+(-1)^m\sum_{k\ge2}\binom{2k-1}{m-1}E_{2k}(\tau)
		S(a[2k-m]_{\mathrm{Mi}}u,\tau)=0.
		\]
		The $k=1$ term vanishes by the $a[0]_{\mathrm{Mi}}$ relation when $m=2$, and its binomial coefficient is zero when $m>2$. Using \eqref{eq:Miyamoto-normalization} and the Eisenstein-series convention above, these become all the relations in \eqref{eq:oneptcfb-characterization}. Thus $\varphi(\tau)$ belongs to the geometric fiber for every $\tau$.
		By \cref{rmk:C20coinvfd}, a finite-dimensional representative space $U\subset W$ yields, after analytic base change, a surjection $U\otimes\O_\H\twoheadrightarrow\mathscr{C}_\H(W)^\vee$. As in the proof of \cref{lem:sectioncriterion}, restriction to $U$ identifies $\mathscr{C}_\H(W)$ with a subbundle of $U^\ast\otimes\O_\H$. Condition~\textnormal{(C1)} makes $\varphi|_U$ holomorphic, and the fiber relations place it in this subbundle. Thus $\varphi$ is a holomorphic section. Condition~\textnormal{(C4)} is precisely its horizontality equation by \eqref{eq:Miyamoto-C4}.
		
		Finally, \cref{rmk:C20-trace-frame} and the argument of \cref{coro:frame} give a unique expression $\varphi=\sum_{k,j}c_{k,j}\varphi_{I_j^k}$ with constants $c_{k,j}\in\C$. Pairing with $u$ gives the stated trace expansion.
	\end{proof}
	
	\cref{thm:Miyamoto-comparison} strengthens Miyamoto's spanning conclusion \cite[Theorem~4.15]{Mi00} to a basis theorem, yielding unique trace expansions.
	The hypotheses on $W$ are also weaker: neither irreducibility nor Virasoro-generation is required.

	\section{Modular invariance and monodromy}\label{sec:modular-invariance-monodromy}
	
	Throughout this section, unless otherwise stated, $V$ is simple, rational, $C_2$-cofinite, and of CFT type, and $W$ is a finitely generated admissible $V$-module of conformal weight $h_W$. We use the horizontal trace frame to establish modular invariance and identify the resulting projective representation with the monodromy of the DGT connection.
	
	\subsection{Modular invariance and the projective representation}

	\subsubsection{The modular action and compatibility with the connection}
	
	We first construct a projective right action of $\SL(2,\Z)$ on $\Gamma(\H,\mathscr{C}_\H(W))$ and show that it preserves the connection.
	
	Let $\gamma=\twobytwo{a}{b}{c}{d}\in\SL(2,\Z)$. It induces an isomorphism of pointed elliptic curves
	\[
	f_{\gamma,\tau} \colon (E_\tau,\mpt{p})\longrightarrow(E_{\gamma\tau},\mpt{p}), \qquad [z] \longmapsto \ab[(c\tau+d)^{-1}z],
	\]
	with local coordinate $t:=(c\tau+d)^{-1}z$ at the marked point.
	Let $\Omega\subseteq\H$ be open and let $\varphi\in\Gamma(\Omega,\mathscr{C}_\H(W))$. 
	Define, a priori, a section $\varphi\cdot\gamma\in\Hdual{\mathscr{W}}(\gamma^{-1}\Omega)$ of the ambient dual bundle by
	\begin{equation}\label{eq:modular-action-on-sections}
		\begin{aligned}
			\braket<(\varphi\cdot\gamma)(\tau)|u>
			&:=
			\braket<\varphi(\gamma\tau)|(c\tau+d)^{-L[0]}u>\\
			&=
			(c\tau+d)^{-h_W}\braket<\varphi(\gamma\tau)|(c\tau+d)^{-L[0]+h_W}u>,
			\qquad \tau\in\gamma^{-1}\Omega,\quad u\in W.
		\end{aligned}
	\end{equation}
	Since $h_W$ need not be integral, we fix the logarithm used in
	\[
	(c\tau+d)^{-L[0]} := \exp\ab(-\log(c\tau+d)L[0]).
	\]
	If $c\neq0$, we use the principal branch; if $c=0$ and $d=-1$, we set $\log(-1) = \pi\iu$; the case $c=0$ and $d=1$ is unambiguous.

	\begin{lemma}\label{lem:welldefofgaact}
		For every open subset $\Omega \subseteq \H$ and $\gamma\in \SL(2,\Z)$, \eqref{eq:modular-action-on-sections} defines a map
		\[
		(-)\cdot\gamma\colon \Gamma(\Omega,\mathscr{C}_\H(W)) \longrightarrow \Gamma(\gamma^{-1}\Omega,\mathscr{C}_\H(W)),\quad \varphi \longmapsto \varphi\cdot\gamma.
		\]
		In other words, if $\varphi\in \Gamma(\Omega,\mathscr{C}_{\H}(W))$, then $\varphi\cdot\gamma\in \Hdual{\mathscr{W}}(\gamma^{-1}\Omega)$ satisfies
		\[  
		\braket<(\varphi\cdot\gamma)(\tau)|\mathsf{R}_1(a)u>=0\quad \mathrm{and}\quad   \braket<(\varphi\cdot\gamma)(\tau)|\mathsf{R}_{\wp_m(z,\tau)}(a)u>=0, 
		\]
		for all $a\in V$, $u\in W$, $m \ge 2$, and $\tau\in\gamma^{-1}\Omega$. 
	\end{lemma}
	\begin{proof}
		Fix $\tau\in\gamma^{-1}\Omega$ and put
		\[
		a'=(c\tau+d)^{1-L[0]}a,\qquad u'=(c\tau+d)^{-L[0]}u.
		\]
		For $f(z)=1$ or $f(z)=\wp_m(z,\tau)$, the $L[0]$-conjugation identity and the change of variables $z=(c\tau+d)t$ give
		\begin{align*}
			(c\tau+d)^{-L[0]}\mathsf{R}_{f(z)}(a)u
			&=\Res_z Y_W\ab[(c\tau+d)^{-L[0]}a,\frac{z}{c\tau+d}]u'\,\iota_z(f(z))\\
			&=\Res_t(c\tau+d)Y_W[(c\tau+d)^{-L[0]}a,t]u'\,\iota_t(f((c\tau+d)t))\\
			&=\mathsf{R}_{f((c\tau+d)t)}(a')u'.
		\end{align*}
		The factor $c\tau+d$ from the residue substitution is absorbed into $a'$.
		
		Since $ad-bc=1$, the lattice identity
		\[
		(c\tau+d)\Lambda_{\gamma\tau}
		=\Z(c\tau+d)+\Z(a\tau+b)=\Lambda_\tau
		\]
		shows that $t\mapsto\wp_m((c\tau+d)t,\tau)$ is $\Lambda_{\gamma\tau}$-periodic and has poles only in $\Lambda_{\gamma\tau}$. Thus it descends to a meromorphic function on $E_{\gamma\tau}$ with its only possible pole at the marked point. The same is true for the constant function $f(z)=1$. 
		As $\varphi(\gamma\tau)$ is a conformal block, we obtain
		\[
		\braket<(\varphi\cdot\gamma)(\tau)|\mathsf{R}_{f(z)}(a)u>
		=\braket<\varphi(\gamma\tau)|\mathsf{R}_{f((c\tau+d)t)}(a')u'>=0
		\]
		for both choices of $f$. Together with holomorphicity, the section criterion \eqref{eq:condforsec} proves the assertion.
	\end{proof}
	
	Since the formula is pointwise and compatible with restriction to open subsets, \cref{lem:welldefofgaact} yields holomorphic bundle morphisms
	\begin{equation}\label{eq:modularfibermap}
		\Phi_\gamma\colon
		\gamma^{\ast}\mathscr{C}_\H(W)\longrightarrow\mathscr{C}_\H(W),
		\qquad \gamma\in\SL(2,\Z),
	\end{equation}
	whose fiber at $\tau$ is characterized by
	\[
	\braket<\Phi_\gamma(\tau)(\psi)|u>
	=
	\braket<\psi|(c\tau+d)^{-L[0]}u>,
	\qquad
	\psi\in\mathscr{C}_\H(W)_{\gamma\tau},\quad u\in W.
	\]
	The transformation of sections in \eqref{eq:modular-action-on-sections} is therefore
	\begin{equation}\label{eq:sectionactionfromPhi}
		(\varphi\cdot\gamma)(\tau)
		=
		\Phi_\gamma(\tau)\ab(\varphi(\gamma\tau)).
	\end{equation}
	
	\begin{lemma}\label{lem:multiplier}
		For any $\gamma=\twobytwo{a}{b}{c}{d}\in \SL(2,\Z)$ and $\tau\in \H$, write $j(\gamma,\tau)=(c\tau+d)$. With the branch convention specified immediately after \cref{eq:modular-action-on-sections}, define
		\[
		\sigma_W \colon \SL(2,\Z)\times \SL(2,\Z) \longrightarrow \mu_\infty,\quad
		(\theta,\gamma)\longmapsto\frac{j(\theta,\gamma\tau)^{-h_W}j(\gamma,\tau)^{-h_W}}{j(\theta\gamma,\tau)^{-h_W}},
		\]
		where $\mu_\infty$ is the group of roots of unity.
		Then $\sigma_W$ is independent of the choice of $\tau\in \H$ and satisfies the cocycle condition:
		\[\sigma_W(\eta,\theta\gamma)\sigma_W(\theta,\gamma)
		=\sigma_W(\eta\theta,\gamma)\sigma_W(\eta,\theta),
		\]
		for any $\eta,\theta,\gamma\in \SL(2,\Z)$.
	\end{lemma}
	\begin{proof}
		Matrix multiplication gives $j(\theta\gamma,\tau)=j(\theta,\gamma\tau)j(\gamma,\tau)$. Hence
		\[
		\log j(\theta\gamma,\tau)-\log j(\theta,\gamma\tau)-\log j(\gamma,\tau)
		\in 2\pi\iu\Z.
		\]
		This difference is continuous on the connected domain $\H$, so it equals $2\pi\iu m$ for a fixed integer $m$. Thus $\sigma_W(\theta,\gamma)=\e^{2\pi\iu h_Wm}$ is constant and is a root of unity since $h_W\in\Q$.
		
		For the cocycle identity, evaluate $\sigma_W(\eta,\theta)$ at $\gamma\tau$ and the other three factors at $\tau$. Substituting the defining quotients and cancelling the intermediate factors gives the same expression on both sides.
	\end{proof}
	
	\begin{remark}\label{rmk:C20-multiplier}
		The above proof already gives a $\C^\times$-valued cocycle $\sigma_W$ for arbitrary $h_W\in\C$, without using $C_2$-cofiniteness of $V$. The projective group law in \cref{prop:sectionproperty} therefore still holds; $h_W\in\Q$ is needed only to ensure that the values of $\sigma_W$ are roots of unity.
	\end{remark}

	\begin{lemma}\label{lem:Ttauconj}
		The linear operator $T(\tau) \colon W \longrightarrow W$ in \eqref{eq:defTq} satisfies the following transformation formula for $\gamma=\twobytwo{a}{b}{c}{d}\in \SL(2,\Z)$:
		\[
		(c\tau+d)^{-L[0]} T(\tau)=(c\tau+d)^{-2} T(\gamma\tau) (c\tau+d)^{-L[0]} +(2\pi\iu)^2 \upxi ((c\tau+d)^{-L[0]}),
		\]
	\end{lemma}
	\begin{proof}
		Note that $ \upxi ((c\tau+d)^{-L[0]})=-\frac{1}{2\pi\iu}L[0]c (c\tau+d)^{-L[0]-1}$.
		Then by the quasi-modularity of $\wp_1(z,\tau)$ in \cref{lem:wp1modular}, together with the facts that $L[0]\widetilde{\cfv}=2\widetilde{\cfv}$ and $\widetilde{\cfv}[1]=L[0]$, we have
		\begin{align*}
			(c\tau+d)^{-L[0]} T(\tau)
			&=
			\Res_z (c\tau+d)^{-L[0]} Y_{W}[\widetilde{\cfv},z] (c\tau+d)^{L[0]}(c\tau+d)^{-L[0]} \iota_z(\wp_1(z,\tau))\\
			&=
			\Res_zY_W\ab[(c\tau+d)^{-L[0]}\widetilde{\cfv},\frac{z}{c\tau+d}] (c\tau+d)^{-L[0]-1}\ab(\wp_1\ab(\frac{z}{c\tau+d},\gamma\tau)-2\pi\iu cz ) \\
			&=
			\Res_u(c\tau+d)^{-2} Y_W[\widetilde{\cfv},u] \wp_1(u,\gamma\tau) (c\tau+d)^{-L[0]}\\
			&\quad 
			-\Res_u(c\tau+d)^{-2} Y_W[\widetilde{\cfv},u](c\tau+d)^{-L[0]} (2\pi\iu c)\cdot  (c\tau+d)u\\
			&= 
			(c\tau+d)^{-2} T(\gamma\tau) (c\tau+d)^{-L[0]} - \widetilde{\cfv}[1](c\tau+d)^{-L[0]-1} (2\pi\iu c) \\
			&=
			(c\tau+d)^{-2} T(\gamma\tau) (c\tau+d)^{-L[0]} +(2\pi\iu)^2 \upxi ((c\tau+d)^{-L[0]}),
		\end{align*}
		where we make the change of variables $u=\frac{z}{c\tau+d}$ in the third equality.
	\end{proof}
	
	\begin{proposition}\label{prop:sectionproperty}
		The morphisms \eqref{eq:modularfibermap} are bundle isomorphisms and satisfy
		\[
		\Phi_\gamma\circ\gamma^{\ast}\Phi_\theta
		=
		\sigma_W(\theta,\gamma)\Phi_{\theta\gamma}.
		\]
		The induced right action on global sections therefore satisfies
		\[
		(\varphi\cdot\theta)\cdot\gamma
		=\sigma_W(\theta,\gamma)\ab(\varphi\cdot(\theta\gamma)).
		\]
		This action is compatible with the connection: for $\gamma=\twobytwo{a}{b}{c}{d}$,
		\[
		\ab(\nabla_\upxi(\varphi\cdot\gamma))(\tau)
		=(c\tau+d)^{-2}\ab((\nabla_\upxi\varphi)\cdot\gamma)(\tau).
		\]
		Hence the action restricts to a projective representation
		\begin{equation}\label{eq:firstrep}
			\rho\colon\SL(2,\Z)\longrightarrow\mathrm{PGL}(\Gamma^\nabla(\H,\mathscr{C}_\H(W))),
			\qquad g\longmapsto[\tilde\rho(g)],
		\end{equation}
		where $\tilde\rho(g)\varphi:=\varphi\cdot g^{-1}$. If $h_W\in\Z$, then $\sigma_W=1$, so $\tilde\rho$ is a linear representation.
	\end{proposition}
	\begin{proof}
		By \cref{lem:welldefofgaact}, we have $\varphi\cdot\gamma\in \Gamma(\H, \mathscr{C}_\H(W))$. Let $\gamma=\twobytwo{a}{b}{c}{d}$ and $\theta=\twobytwo{a'}{b'}{c'}{d'}$ be elements in $\SL(2,\Z)$. Then by \eqref{eq:modular-action-on-sections} and \cref{lem:multiplier}, for any $\tau\in \H$ and $u\in W$, we have
		\begin{align*}
			\MoveEqLeft
			\braket<\ab((\varphi\cdot\theta)\cdot\gamma)(\tau)|u> = 
			j(\gamma,\tau)^{-h_W}\braket<(\varphi\cdot\theta)(\gamma\tau)|(c\tau+d)^{-L[0]+h_W}u>\\
			&=
			j(\gamma,\tau)^{-h_W}j(\theta,\gamma\tau)^{-h_W}\braket<\varphi(\theta(\gamma\tau))|\ab(c'\ab(\frac{a\tau+b}{c\tau+d})+d')^{-L[0]+h_W}(c\tau+d)^{-L[0]+h_W}u>\\
			&=
			\sigma_W(\theta,\gamma)j(\theta\gamma,\tau)^{-h_W}\braket<\varphi((\theta\gamma)\tau)|\ab((c'a+d'c)\tau+(c'b+d'd))^{-L[0]+h_W}u>\\
			&=
			\sigma_W(\theta,\gamma)\braket<\ab(\varphi\cdot(\theta\gamma))(\tau)|u>.
		\end{align*}
		Hence $(\varphi\cdot\theta)\cdot\gamma=\sigma_W(\theta,\gamma)\ab(\varphi\cdot(\theta\gamma))$ for all $\gamma,\theta\in \SL(2,\Z)$. By \eqref{eq:sectionactionfromPhi}, this is equivalent to the asserted composition law for $\Phi_\gamma$. Since $\Phi_1=\Id$, taking $\theta=\gamma^{-1}$ shows that every $\Phi_\gamma$ is invertible. Thus $\tilde\rho(g)\varphi=\varphi\cdot g^{-1}$ defines the asserted projective left representation.
		
		By the chain rule, \eqref{eq:defofconnectionvarphi}, and \cref{lem:Ttauconj}, we obtain
		\begin{align*}
			\MoveEqLeft
			\braket<(\nabla_\upxi(\varphi\cdot\gamma))(\tau)|u>\\
			&=\upxi\braket<\varphi(\gamma\tau)|(c\tau+d)^{-L[0]}u>
			-\frac{1}{(2\pi\iu)^2}\braket<\varphi(\gamma\tau)|(c\tau+d)^{-L[0]}T(\tau)u>\\
			&=(c\tau+d)^{-2}\braket<(\upxi\varphi)(\gamma\tau)|(c\tau+d)^{-L[0]}u>
			+\braket<\varphi(\gamma\tau)|\upxi((c\tau+d)^{-L[0]})u>\\
			&\quad-\frac{(c\tau+d)^{-2}}{(2\pi\iu)^2}\braket<\varphi(\gamma\tau)|T(\gamma\tau)(c\tau+d)^{-L[0]}u>
			-\braket<\varphi(\gamma\tau)|\upxi((c\tau+d)^{-L[0]})u>\\
			&=(c\tau+d)^{-2}\braket<(\nabla_\upxi\varphi)(\gamma\tau)|(c\tau+d)^{-L[0]}u>\\
			&=(c\tau+d)^{-2}\braket<\ab((\nabla_\upxi\varphi)\cdot\gamma)(\tau)|u>.
		\end{align*}
		Here $(c\tau+d)^{-2}$ is the derivative of $\gamma\tau$ with respect to $\tau$.
		
		If $\varphi$ is horizontal, the right-hand side vanishes. Then the section $\varphi\cdot\gamma$ is horizontal. The action therefore restricts to the projective representation \eqref{eq:firstrep}. If $h_W\in\Z$, then \cref{lem:multiplier} gives $\sigma_W=1$, so $\tilde\rho$ is a linear representation.
	\end{proof}
	
	\subsubsection{Modular transformation matrices and fiber representations}
	
	We now derive the modular transformation law from the horizontal frame. 
	
	\begin{theorem}\label{thm:modularinv}
		For any $\gamma\in \SL(2,\Z)$, there exist unique constants $\gamma_{(k,j)}^{(k',j')}\in \C$, depending only on $\gamma$ and the intertwining operators, such that the following relations hold in $\Gamma(\H,\mathscr{C}_\H(W))$:
		\begin{equation}\label{eq:modular-transformation-of-trace-frame}
			\varphi_{I^k_j}\cdot\gamma=\sum_{k'=0}^N\sum_{j'=1}^{n_{k'}} \gamma_{(k,j)}^{(k',j')} \cdot \varphi_{I^{k'}_{j'}},\quad 0 \le k \le N,\ 1 \le j \le n_k,
		\end{equation}
		where the section $\varphi_{I^k_j}\cdot\gamma$ is defined by \eqref{eq:modular-action-on-sections}.
		
		In particular, for any $\tau\in \H$, $u\in W$,  and $\gamma=\twobytwo{a}{b}{c}{d}\in \SL(2,\Z)$, denote
		\begin{equation}\label{eq:defZ}
			Z_{I^k_j}(u,\tau):=\braket<\varphi_{I^k_j}(\tau)|u>=\tr_{M^k}\ab(o_{I^k_j}(u)q_\tau^{L(0)-\frac{\ccharge}{24}}),\quad 0 \le k \le N,\ 1 \le j \le n_k.
		\end{equation}
		Then  we have
		\begin{equation}\label{eq:modularinv}
			Z_{I^k_j}\ab((c\tau+d)^{-L[0]}u,\frac{a\tau+b}{c\tau+d})=\sum_{k'=0}^N\sum_{j'=1}^{n_{k'}} \gamma_{(k,j)}^{(k',j')}\cdot  Z_{I^{k'}_{j'}}(u,\tau ).
		\end{equation}
	\end{theorem}
	We call \eqref{eq:modularinv} the \concept{modular invariance of trace functions associated to intertwining operators} for the VOA $V$.
	
	\begin{proof}
		By \cref{coro:frame}, the trace family $\Fam{\varphi_{I_j^k}}{\substack{0\le k\le N\\1\le j\le n_k}}$ is a $\C$-basis of the space $\Gamma^\nabla(\H,\mathscr{C}_\H(W))$. 
		For $\gamma\in\SL(2,\Z)$, \cref{prop:sectionproperty} shows that $\varphi_{I_j^k}\cdot\gamma$ is horizontal. Its unique expansion in this basis therefore has constant coefficients, giving \eqref{eq:modular-transformation-of-trace-frame}.
		Evaluating \eqref{eq:modular-transformation-of-trace-frame} at $\tau$ and pairing with $u\in W$, the action formula \eqref{eq:modular-action-on-sections} and the notation \eqref{eq:defZ} give \eqref{eq:modularinv}.
	\end{proof}
	
	Consider the special case $W = M^0 = V$. For each $0 \le k \le N$, the vector space $\Ifusion[V][M^k][M^k]$ is one-dimensional, with a canonical basis given by the vertex operator $Y_{M^k}(\placeholder,z)$; see \cite{FHL93}. Moreover, $o_{Y_{M^k}}(a) = a_{(\wt a-1)} = o(a)$ for any $a\in V$, and
	\[
	Z_{Y_{M^k}}(a,\tau)=\tr_{M^k}\ab(o(a) q_\tau^{L(0)-\frac{\ccharge}{24}})=Z_{M^k}(a,\tau), \quad 0 \le k \le N,
	\]
	as in \cite{MS89,Zhu96,DLM00}. Then the above theorem specializes to the following statement, which is \cite[Theorem~5.3.2]{Zhu96} under slightly different conditions.
	\begin{corollary}\label{coro:originalmodularinv}
		For any $\gamma=\twobytwo{a}{b}{c}{d}\in \SL(2,\Z)$, there exist constants $\gamma_k^\ell\in \C$, depending only on $\gamma$ and the irreducible modules, such that, for any $\tau\in \H$ and $u\in V$, we have
		\begin{equation}
			Z_{M^k}\ab((c\tau+d)^{-L[0]}u,\frac{a\tau+b}{c\tau+d})=\sum_{\ell=0}^N \gamma_k^\ell \cdot Z_{M^\ell}(u,\tau),\quad 0 \le k \le N.
		\end{equation}
	\end{corollary}

	To compare this representation with monodromy, we fix $\tau\in\H$. By \cref{coro:frame}, the evaluation map
	\[
	\mathrm{ev}_\tau\colon \Gamma^\nabla(\H,\mathscr{C}_\H(W))\isomto\mathscr{C}(E_\tau,\mpt{p},z,W),
	\qquad\varphi\longmapsto\varphi(\tau),
	\]
	is an isomorphism and transports $\rho$ to the fiber:
	\begin{equation}\label{eq:firstrepofSL2Z}
		\tilde\rho_\tau(g):=\mathrm{ev}_\tau\circ\tilde\rho(g)\circ\mathrm{ev}_\tau^{-1},
		\qquad\rho_\tau(g):=[\tilde\rho_\tau(g)].
	\end{equation}
	For $\varphi\in \Gamma^\nabla(\H,\mathscr{C}_\H(W))$ and $u\in W$, the action formula becomes
	\begin{equation}\label{eq:fiber-action-pairing}
		\braket<\tilde\rho_\tau(\gamma^{-1})(\varphi(\tau))|u>
		=\braket<\varphi(\gamma\tau)|(c\tau+d)^{-L[0]}u>.
	\end{equation}
	In the trace basis of the fiber, \eqref{eq:modular-transformation-of-trace-frame} reads
	\begin{equation}
		\tilde\rho_\tau(\gamma^{-1})(\varphi_{I_j^k}(\tau))
		=\sum_{k'=0}^N\sum_{j'=1}^{n_{k'}}\gamma_{(k,j)}^{(k',j')}\varphi_{I_{j'}^{k'}}(\tau).
	\end{equation}
	
	In particular, the next proposition computes the $T$-action in the trace basis.
	\begin{proposition}\label{prop:T-action-diagonal}
		For $0\le k\le N$ and $1\le j\le n_k$,
		\[
		\tilde\rho_\tau(T)(\varphi_{I_j^k}(\tau))
		=\e^{-2\pi\iu(h_{M^k}-\frac{\ccharge}{24})}\varphi_{I_j^k}(\tau).
		\]
		Hence $\tilde\rho_\tau(T)$ is diagonal and has finite order.
	\end{proposition}
	\begin{proof}
		Since $L(0)$ acts on $M^k(n)$ by $h_{M^k}+n$, the trace expansion gives
		\[
		Z_{I_j^k}(u,\tau-1)
		=\e^{-2\pi\iu(h_{M^k}-\frac{\ccharge}{24})}Z_{I_j^k}(u,\tau).
		\]
		Taking $\gamma=T^{-1}$ in \eqref{eq:fiber-action-pairing} proves the displayed action. Its diagonal entries are roots of unity because $h_{M^k},\ccharge\in\Q$; see \cite[Theorem~11.3]{DLM00}.
	\end{proof}
	
	\begin{remark}\label{rmk:C20-T-action}
		By \cref{rmk:C20-trace-frame}, the displayed formula for $\tilde\rho_\tau(T)$ also holds for $C_{[2,0]}$-cofinite $W$ without $C_2$-cofiniteness of $V$. The finite-order conclusion additionally requires $h_{M^k}-\frac{\ccharge}{24}\in\Q$ for every $k$ with $n_k>0$.
	\end{remark}
	
	For $W=M^i$ and $S=\twobytwo{0}{-1}{1}{0}$, let
	\begin{equation}\label{eq:Si}
		S(i)\colon
		\bigoplus_{k=0}^N \Ifusion[M^i][M^k][M^k]\longrightarrow
		\bigoplus_{k=0}^N \Ifusion[M^i][M^k][M^k]
	\end{equation}
	be the operator corresponding to $\tilde\rho(S^{-1})$ via \eqref{eq:trace-frame-map}.
	This is the genus-one $S$-transformation of \cite[(3.18)--(3.20)]{MS89}; the inverse occurs because $\tilde\rho(S^{-1})\varphi=\varphi\cdot S$ by the convention in \cref{prop:sectionproperty}.

	\subsection{Monodromy interpretation of the modular representation}
	
	We now interpret the above modular representation in terms of monodromy of the Damiolini--Gibney--Tarasca projective connection.

	We first recall the standard equivariant description of vector bundles on a quotient orbifold. Let $G$ be a discrete group acting holomorphically on a complex manifold $X$. A \emph{$G$-equivariant structure} on a holomorphic vector bundle $\mathcal{F}$ on $X$ consists of isomorphisms
	\[
	\Phi_g\colon g^{\ast}\mathcal{F}\longrightarrow\mathcal{F},
	\qquad g\in G,
	\]
	satisfying $\Phi_g\circ g^{\ast}\Phi_h = \Phi_{hg}$. Holomorphic vector bundles on $[G\dblbackslash X]$ are the same as $G$-equivariant holomorphic vector bundles on $X$; see \cite[\S~3.3]{Ha21}.
	A family of such isomorphisms may satisfy the cocycle condition only up to a scalar $2$-cocycle $\sigma\colon G\times G\to\C^\times$. We call such a family a \emph{$\sigma$-twisted $G$-equivariant structure}. In explicit, the condition above is replaced by
	\[
	\Phi_g\circ g^{\ast}\Phi_h
	=
	\sigma(h,g)\Phi_{hg}.
	\]
	
	In our scenario, $G=\SL(2,\Z)$ acting on $X=\H$ by fractional linear transformations. 
	For our vector bundle $\mathscr{C}_\H(W)=\pi^{\ast}\mathscr{C}^{\mathrm{an}}_{\mathcal{M}_{1,1}}(W)$, the canonical equivariant structure on the pullback gives, for each $\gamma\in\SL(2,\Z)$, the fiber isomorphism
	\[
	\Phi_\gamma^{\mathrm{can}}(\tau)\colon
	\mathscr{C}_\H(W)_{\gamma\tau}\longrightarrow
	\mathscr{C}_\H(W)_\tau.
	\]
	By \cref{prop:sectionproperty}, the modular fiber maps $\Phi_\gamma$ form a $\sigma_W$-twisted $\SL(2,\Z)$-equivariant structure on $\mathscr{C}_\H(W)$.
	Fix $\tau\in\H$. An orbifold loop based at $[\tau]$ lifts to a path $\widetilde\eta$ in $\H$ from $\tau$ to some $\gamma\tau$. Let $P_{\widetilde\eta}$ denote parallel transport for $\nabla$ along this path. By \cref{thm:DGT-comparison}, the DGT projective monodromy of the loop is the class in $\mathrm{PGL}(\mathscr{C}_\H(W)_\tau)$ of
	\[
	\Phi_\gamma^{\mathrm{can}}(\tau)\circ P_{\widetilde\eta}.
	\]
	On the modular side, since $P_{\widetilde\eta}(\varphi(\tau))=\varphi(\gamma\tau)$ for every horizontal section $\varphi$, \eqref{eq:sectionactionfromPhi} gives
	\begin{equation}\label{eq:modularastransport}
		\tilde\rho_\tau(\gamma^{-1})
		=
		\Phi_\gamma(\tau)\circ P_{\widetilde\eta}.
	\end{equation}
	Thus the monodromy comparison reduces to comparing $\Phi_\gamma^{\mathrm{can}}$ with $\Phi_\gamma$.
	
	In the square-bracket realization, the above fiber maps use the degree operator $L[0]-h_W$, which is the grading used in the descent preceding \cref{thm:DGT-comparison}. The modular fiber maps instead use $L[0]$. Their difference is therefore scalar:
	
	\begin{lemma}\label{lem:SL2Zequi}
		Let $\gamma=\twobytwo{a}{b}{c}{d}\in\SL(2,\Z)$.
		For $\psi\in \mathscr{C}_{\H}(W)_{\gamma\tau}$ and $u\in W$,
		\begin{equation}\label{eq:explicitcanonicaldescent}
			\braket<\Phi_{\gamma}^{\mathrm{can}}(\tau)(\psi)|u>
			=
			\braket<\psi|(c\tau+d)^{-L[0]+h_W}u>.
		\end{equation}
		Consequently,
		\begin{equation}\label{eq:modular-fiber-map-scalar-relation}
			\Phi_{\gamma}(\tau)
			=
			(c\tau+d)^{-h_W}\Phi_{\gamma}^{\mathrm{can}}(\tau).
		\end{equation}
	\end{lemma}
	\begin{proof}
		With the coordinate-change convention of \cite[\S~2.7 and 3.2]{DGT21}, the $\C^\times$-descent in \cite[\S~4.2.1, (16), and \S~6.3.2]{DGT21} identifies $u \in W$ in the $z$-coordinate with $(c\tau+d)^{-L[0]+h_W}u$ in the coordinate $t=(c\tau+d)^{-1}z$. Dualizing the induced map on coinvariants gives
		\[
		\braket<\Phi_{\gamma}^{\mathrm{can}}(\tau)(\psi)|u>
		=\braket<\psi|(c\tau+d)^{-L[0]+h_W}u>,
		\]
		proving \eqref{eq:explicitcanonicaldescent}. 
		The definition of $\Phi_\gamma(\tau)$ replaces the inserted operator by $(c\tau+d)^{-L[0]}$, which proves \eqref{eq:modular-fiber-map-scalar-relation}.
	\end{proof}
	
	\begin{theorem}\label{thm:monodromy-representation}
		The projective monodromy representation of $\nabla^{\mathrm{DGT}}$ based at $[\tau]$ is $\rho_\tau$ in \eqref{eq:firstrepofSL2Z}: a loop whose lift runs from $\tau$ to $\gamma\tau$ acts as $\rho_\tau(\gamma^{-1})\in\mathrm{PGL}(\mathscr{C}_\H(W)_\tau)$.
	\end{theorem}
	\begin{proof}
		By \cref{lem:SL2Zequi}, the two composites
		$\Phi_\gamma^{\mathrm{can}}(\tau)\circ P_{\widetilde\eta}$ and
		$\Phi_\gamma(\tau)\circ P_{\widetilde\eta}$ differ by a nonzero scalar. They therefore define the same element of $\mathrm{PGL}(\mathscr{C}_\H(W)_\tau)$, which is $\rho_\tau(\gamma^{-1})$ by \eqref{eq:modularastransport}.
	\end{proof}
	
	\begin{corollary}\label{coro:DW-action}
		Suppose that $V$ is self-dual. Under the identification between the original and square-bracket realizations of \S~\ref{sec:elliptic-cfb}, dualizing from coinvariants to conformal blocks and projectively descending from the moduli space with tangent data identifies the genus-one mapping-class-group representation of the modular functor in \cite[Theorem~4.2.2]{DW25} with $\rho_\tau$.
	\end{corollary}
	\begin{proof}
		Under the standing hypotheses, self-duality makes $V$ strongly rational, and $\mathcal{U}_V$ makes $V^\square$ strongly rational as well. The preceding bundle isomorphism intertwines the DGT projective connections. It respects propagation since $\mathcal{U}_V\vac=\vac$. On each finite-dimensional $M(n)\otimes M(n)^*$, the map $\mathcal{U}_M\otimes(\mathcal{U}_M^{-1})^*$ takes the coevaluation tensor to that of $M[n]\otimes M[n]^*$; the sewing factor $q^n$ is unchanged since $\mathcal{U}_M M(n)=M[n]$. Thus factorization and sewing commute with transport. The modular functor of \cite[\S~2.2 and Theorem~4.2.2]{DW25} is built from these data. Its dual genus-one action, after this fiber identification, is the square-bracket input DGT projective monodromy, equal to $\rho_\tau$ by \cref{thm:monodromy-representation}.
	\end{proof}
	
	\subsubsection{SPIN-OFF: Hodge twists}
	
	We now isolate the scalar factor in \eqref{eq:modular-fiber-map-scalar-relation}. Let $\lambda$ be the Hodge line bundle. Its pullback to $\H$ is trivialized by $dz$, and
	\[
	f_{\gamma,\tau}^{\ast}(du)=(c\tau+d)^{-1}dz.
	\]
	Thus $\lambda$ has automorphy factor $J_1$, where we define
	\[
	J_{h}(\gamma,\tau):=(c\tau+d)^{-h}.
	\]
	By \cref{lem:multiplier},
	\[
	J_{h_W}(\theta,\gamma\tau)J_{h_W}(\gamma,\tau)
	=\sigma_W(\theta,\gamma)J_{h_W}(\theta\gamma,\tau).
	\]
	Therefore $J_{h_W}$ defines a $\sigma_W$-twisted $\SL(2,\Z)$-equivariant line bundle on $\H$, and hence a twisted line bundle on $\mathcal{M}_{1,1}^{\mathrm{an}}$. We denote it by $\lambda^{\otimes h_W}$ and call it the \emph{Hodge twist of weight $h_W$}. When $h_W \notin \Z$, this notation should be understood via the preceding twisted equivariant structure rather than a tensor power. 
	Set
	\[
	\mathscr{C}^{\mathrm{an}}_{\mathcal{M}_{1,1}}(W)_{h_W}
	:=\mathscr{C}^{\mathrm{an}}_{\mathcal{M}_{1,1}}(W)\otimes\lambda^{\otimes h_W}.
	\]
	Its pullback to $\H$ is $\mathscr{C}_\H(W)$, and \eqref{eq:modular-fiber-map-scalar-relation} identifies its twisted equivariant structure with the maps $\Phi_\gamma$. Thus the induced transformation of pullback sections is exactly the modular action \eqref{eq:sectionactionfromPhi}.
	
	\begin{lemma}\label{lem:descendconnection}
		The flat connection $\nabla$ on $\mathscr{C}_\H(W)$ descends, with respect to the twisted equivariant structure $\Fam{\Phi_\gamma}{\gamma\in\SL(2,\Z)}$, to a flat connection, again denoted by $\nabla$, on $\mathscr{C}^{\mathrm{an}}_{\mathcal{M}_{1,1}}(W)_{h_W}$. Moreover, its projectivization is the Damiolini--Gibney--Tarasca projective connection on $\mathscr{C}^{\mathrm{an}}_{\mathcal{M}_{1,1}}(W)$.
	\end{lemma}
	\begin{proof}
		The compatibility formula in \cref{prop:sectionproperty} says precisely that $\nabla$ is compatible with the twisted equivariant structure $\Fam{\Phi_\gamma}{\gamma\in\SL(2,\Z)}$; hence it descends. The descended connection is flat because its pullback is $\nabla$. After projectivization the scalar cocycle $\sigma_W$ disappears, and twisting by a line does not change the projective bundle. The last assertion now follows from \cref{thm:DGT-comparison}.
	\end{proof}
	
	\begin{remark}\label{rmk:C20-descent}
		\cref{thm:modularinv,thm:monodromy-representation,lem:descendconnection} remain valid when $W$ is $C_{[2,0]}$-cofinite, without assuming $V$ is $C_2$-cofinite. Indeed, \cref{rmk:C20-trace-frame} supplies the horizontal trace frame, and the proofs then use the same coordinate-change and connection formulas. The DGT connection construction itself requires no cofinite hypothesis.
	\end{remark}
	
	\subsection{Further questions on the modular representation}
	
	We conclude with two directions for further study: categorical interpretations of the modular representation and invariant Hermitian structures.
	
	\begin{question}
		Suppose that $V$ is additionally self-dual, and fix $M^i\in\Irr$. Under the natural comparison of tensor products in \cite[\S~5.4]{DW25}, does the punctured-torus $S$-move determined by the Huang--Lepowsky modular tensor category agree on
		$\bigoplus_{k=0}^N\Ifusion[M^i][M^k][M^k]$ with the operator $S(i)$ in \eqref{eq:Si}?
	\end{question}
	
	If $V$ is additionally self-dual, it is strongly rational. For its modular tensor category, the genus-one $S$-matrix is unitary, and the normalized categorical $S$-matrix agrees with it up to sign \cite[Corollary~3.6 and Theorem~3.10(i)]{DLN15}. Equivalently, the matrix $S(0)$ in \eqref{eq:Si} is unitary. For general background on modular tensor categories and modular functors, see \cite{BK01,Tur94,EGNO15}. This motivates the following question.
	\begin{question}\label{question:modular-invariant-hermitian-form}
		Does the space of conformal blocks $ \mathscr{C}(E_\tau,\mpt{p},z,W)$ carry a natural positive-definite Hermitian form which is invariant under the action of $\SL(2,\Z)$ in \eqref{eq:fiber-action-pairing}?   If so, can one naturally normalize the chosen bases of the intertwining-operator spaces so that the resulting matrices $S(i)$, $0 \le i \le N$, are unitary?
	\end{question}
	
	Let $(\mathcal{F},\nabla)$ be a complex vector bundle with a connection on a complex manifold $X$. Recall that a \emph{Hermitian metric} $\mathsf{h}$ on $\mathcal{F}$ is a smoothly varying positive-definite Hermitian form on its fibers. It is compatible with $\nabla$ if
	\[
	d\mathsf{h}(s_1,s_2)
	=\mathsf{h}(\nabla s_1,s_2)+\mathsf{h}(s_1,\nabla s_2)
	\]
	for all local sections $s_1,s_2$. We apply this notion to $\mathscr{C}_\H(W)$ and ask whether the metric can also respect the modular fiber maps.
	\begin{question}
		Does $\mathscr{C}_\H(W)$ carry a natural Hermitian metric $\mathsf{h}$ compatible with $\nabla$ such that every $\Phi_\gamma(\tau)$ is an isometry?
	\end{question}
	If so, $\rho_\tau$ is projectively unitary. When $h_W\in\Z$, this means the lift $\tilde\rho_\tau$ is a unitary representation.
	
	\section{Applications to modular forms}\label{sec:modularforms}
	
	The modular transformation law also turns intertwining-operator traces into vector-valued and scalar modular forms. Throughout this section, we keep the same setup as in the previous section. We use \cref{thm:modularinv} to develop these applications, including a realization of the cusp form $\Delta$, and then discuss the remaining realization and congruence questions.
	
	\subsection{Vector-valued modular forms}
	
	The modular transformation law gives a vector-valued modular form when evaluated on an $L[0]$-homogeneous vector. Here we use a version of the definition in \cite[\S~2]{KM03} that allows rational weights and poles at the cusp. For broader constructions and structural results, including Poincar\'e-series and hypergeometric approaches, see \cite{KM04,BG07}.
	
	\begin{definition}
		Let $r\in\Q$. For $\gamma=\twobytwo{a}{b}{c}{d}\in\SL(2,\Z)$, let $j_r(\gamma,\tau) := (c\tau+d)^r$, where we choose the branch of logarithm as specified immediately after \cref{eq:modular-action-on-sections}.
		Let $K$ be a positive integer. A \concept{matrix-valued multiplier system of weight $r$} is a map
		\[
		\chi \colon \SL(2,\Z) \longrightarrow \mathrm{GL}_K(\C)
		\]
		such that the matrix-valued function $J_{r,\chi}(\gamma,\tau) = j_r(\gamma,\tau)\chi(\gamma)$ satisfies the automorphy property:
		\begin{equation}\label{eq:automorphy}
			J_{r,\chi}(\theta\gamma,\tau)=
			J_{r,\chi}(\theta,\gamma\tau)
			J_{r,\chi}(\gamma,\tau),\quad \theta,\gamma\in\SL(2,\Z).
		\end{equation}
		
		A \concept{vector-valued modular form} of \concept{weight} $r$ and with \concept{matrix-valued multiplier system} $\chi$ is a holomorphic map $\vect{F} \colon \H\to\C^K$ satisfying
		\[
		\vect{F}\ab(\gamma\tau) =
		J_{r,\chi}(\gamma,\tau)\vect{F}\ab(\tau) =
		j_r(\gamma,\tau)\chi(\gamma)\vect{F}\ab(\tau),\quad \gamma\in\SL(2,\Z),
		\]
		and whose components are meromorphic at the cusp. 
		In this paper, $\chi(T)$ is always of finite order and semisimple (in the trace-function construction below, $\chi(T)$ is diagonal with entries of the form $\e^{2\pi\iu(h_{M^k}-\frac{\ccharge}{24})}$). 
		Under this hypothesis, for some positive integer $\mathbf{d}$, each component of $\vect{F}$ has a convergent expansion
		\[
		F_i(\tau) = \sum_{n \ge n_i}a_i(n)q^{n/\mathbf{d}},
		\qquad q=\e^{2\pi\iu\tau},\quad n_i\in \Z.
		\]
	\end{definition}

	Assume that the trace space has positive dimension $K=\sum_{k=0}^N n_k$. With the notation \eqref{eq:defZ}, fix an $L[0]$-homogeneous vector $u\in W$ of weight $r=h_W+n\in\Q$.
	By \cref{thm:conv} and the rationality of $h_{M^k}-\frac{\ccharge}{24}$, each $Z_{I_j^k}(u,\tau)$ is holomorphic on $\H$ and meromorphic at the cusp. With the lexicographic order on $\Set*{(k,j) \given 0 \le k \le N,\ 1 \le j \le n_k}$, set
	\[
	\vect{Z}_W(u,\tau) := \ab (Z_{I^k_j}(u,\tau))_{(k,j)}^{\mathsf{t}}\in\C^K.
	\]
	Define $\chi \colon \SL(2,\Z) \to \mathrm{GL}_K(\C)$ by
	\[
	\chi(\gamma)_{(k,j)}^{(k',j')}:=\gamma_{(k,j)}^{(k',j')},
	\]
	where $\gamma_{(k,j)}^{(k',j')}$ are the coefficients in \cref{thm:modularinv}. Then \eqref{eq:modularinv} gives
	\begin{equation}\label{eq:vectorvaluedform}
		\vect{Z}_W(u,\gamma\tau) = j_{r}(\gamma,\tau)\chi(\gamma)\vect{Z}_W(u,\tau).
	\end{equation}
	By the cocycle property of the modular action, $J_{r,\chi} := j_r(\gamma,\tau)\chi(\gamma)$ satisfies \eqref{eq:automorphy}. Therefore, $\chi$ is a matrix-valued multiplier system of weight $r$, and $\vect{Z}_W(u,\tau)$ is a vector-valued modular form of weight $r$. This construction extends Zhu's vector-valued modular forms from traces of vertex operators to traces of intertwining operators.
	
	\subsection{Realization of scalar modular forms}
	
	When $K=1$, the transformation law is scalar-valued, with a scalar multiplier system.
	We give a criterion for the multiplier to be trivial and for the trace to be holomorphic at the cusp, then apply it to realize the cusp form $\Delta$.
	
	\subsubsection{A criterion for scalar modularity}
	
	Suppose that the trace space is one-dimensional, so that
	\[
	\bigoplus_{k=0}^N \Ifusion[W][M^k][M^k]=\C I
	\]
	for some nonzero $I\in\Ifusion[W][M][M]$ and $M\in\Irr$. 
	In this one-dimensional case, $\chi(\gamma)\in\C^\times$.
	For an $L[0]$-homogeneous vector $u\in W$ of weight $r$,
	\eqref{eq:vectorvaluedform} becomes
	\begin{equation}\label{eq:scalar-trace-modular-transformation}
		Z_I\ab(u,\gamma\tau) = j_{r}(\gamma,\tau)\chi(\gamma)Z_I(u,\tau).
	\end{equation}
	At the cusp, its $q$-expansion is
	\begin{equation}\label{eq:ZIexp}
		Z_I(u,\tau) = q_\tau^{h_M-\frac{\ccharge}{24}}\sum_{n=0}^\infty\tr_{M(n)}\ab(o_I(u)) q_\tau^n,
	\end{equation}
	where $h_M-\frac{\ccharge}{24}\in\Q$. 
	From the proof of \cref{prop:T-action-diagonal}, we also have
	\begin{equation}\label{eq:ChiT}
		\chi(T)=\e^{2\pi\iu(h_M-\frac{\ccharge}{24})}, \qquad T=\twobytwo{1}{1}{0}{1}.
	\end{equation}

	\begin{theorem}\label{thm:modularform}
		Suppose the trace space is spanned by a single nonzero intertwining operator $I\in\Ifusion[W][M][M]$ for some $M\in\Irr$, and assume that $h_M-\frac{\ccharge}{24}\in \N$.
		If $u\in W$ is an $L[0]$-homogeneous vector of integral weight $r$, then $Z_I(u,\tau)$ is a modular form of weight $r$.
	\end{theorem}
	\begin{proof}
		Since $r$ is integral, $j_r(\gamma,\tau)=(c\tau+d)^r$ satisfies the automorphy property without a branch ambiguity. 
		Therefore \eqref{eq:automorphy} makes the scalar multiplier $\chi$ a character. By \eqref{eq:ChiT}, $\chi(T)=1$. 
		Recall that the generators $S,T$ of $\SL(2,\Z)$ satisfy $(ST)^3=S^2$. Applying $\chi$ gives $\chi(S)^3=\chi(S)^2$, hence $\chi(S)=1$ and we conclude that $\chi$ is trivial. Thus \eqref{eq:scalar-trace-modular-transformation} becomes
		\[
		Z_I(u,\gamma\tau) = j_{r}(\gamma,\tau)Z_I(u,\tau).
		\]
		Moreover, by the $q$-expansion \eqref{eq:ZIexp}, $Z_I(u,\tau)$ is holomorphic at the cusp $q=0$. Hence it is a modular form of weight $r$.
	\end{proof}
	
	\begin{remark}\label{rmk:C20-modular-forms}
		Under the weaker hypotheses of \cref{rmk:C20-trace-frame}, \cref{thm:modularform} remains valid with its stated assumptions on $r$ and $h_M-\frac{\ccharge}{24}$. The vector-valued construction \eqref{eq:vectorvaluedform} also applies provided $h_W\in\Q$ and $h_{M^k}-\frac{\ccharge}{24}\in\Q$ for every $k$ with $n_k>0$; \cref{rmk:C20-convergence} gives convergence, and these rationality conditions give the required weights and cusp expansions.
	\end{remark}

	\subsubsection{Realization of the cusp form \texorpdfstring{$\Delta$}{Delta}}
	
	We now apply \cref{thm:modularform} to realize the weight-$12$ cusp form $\Delta(\tau)$ as a trace $Z_I(u,\tau)$. A similar realization follows by specializing \cite[Theorem~5.3]{KSW26} to $k=16$. Here we give a direct construction using the structure of the $A(V)$-bimodule $A(W)$ from \cite{FZ92}.
	
	We first check the one-dimensionality and cusp conditions for $V=L_{\widehat{\sl}_2}(16,0)$. Let $\alpha$ be the positive root of $\sl_2$, normalized by $(\alpha|\alpha)=2$. The irreducible modules and their conformal weights are
	\[
	M^i=L_{\widehat{\sl}_2}(16,\alpha i/2),\qquad
	h_{M^i}=\frac{i(i+2)}{72},\qquad 0\le i\le16,
	\]
	and the central charge is $\ccharge=8/3$. Since $V$ is strongly rational, \cref{thm:modularinv} applies.
	Set $W=M^{16}$, so $h_W=4$. The affine $\sl_2$ fusion rules give \cite{V88,Bea96,FZ92}
	\begin{equation}\label{eq:affinefusion}
		M^{16}\boxtimes_V M^i\cong M^{16-i},\qquad 0\le i\le16.
	\end{equation}
	This shows $\bigoplus_{i=0}^{16}\Ifusion[M^{16}][M^i][M^i]$ is one-dimensional and the nonzero intertwining operators $I$ are of type $\fusion[M^{16}][M^{8}][M^{8}]$. 
	Then by \eqref{eq:ZIexp} and \cref{thm:modularform}, for any $u\in W[n]$, the trace function
	\[
	Z_I(u,\tau)=q\cdot\sum_{\ell=0}^\infty \tr_{M^8(\ell)}\ab(o_I(u)) q^{\ell}
	\]
	is a cusp form of weight $h_W+n=4+n$. Since there are no cusp forms of weight less than $12$, we have $Z_I(u,\tau)=0$ for $u\in W[n]$ when $n<8$.
	
	It remains to obtain a nonzero trace of weight $12$. Step~1 of the proof of \cref{thm:basiscfb} guarantees a vector with nonzero trace on $M^8(0)$, but gives no weight bound. We therefore construct such a vector in relative filtration degree at most $8$. Its lower $L[0]$-components have zero trace by the preceding vanishing, leaving only weight $12$.
	
	\begin{theorem}\label{thm:affinecusp}
		For the affine VOA $V=L_{\widehat{\sl}_2}(16,0)$, with $W=M^{16}$ and $M = M^8$, there exist $\tilde{u}\in W$ and $I\in \Ifusion[W][M][M]$ such that
		\[
		Z_I(\tilde{u},\tau)=  \Delta(\tau)=q\prod_{n=1}^\infty (1-q^n)^{24}.
		\]
	\end{theorem}
	
	\begin{proof}
		First, we normalize the top-level map. 
		Choose a nonzero $I\in\Ifusion[W][M][M]$. By \eqref{eq:fusionrulesfI}, it determines a nonzero $A(V)$-homomorphism
		\[
		\phi_I\colon A(W)\otimes_{A(V)}M(0)\to M(0),\qquad
		\phi_I([u]\otimes v)=o_I(u)v.
		\]
		In our scenario, $A(V)$ is a quotient algebra of $U(\sl_2)$.
		Let $e,f,h$ be the Chevalley basis of $\sl_2$, with $[h,e]=2e$, $[h,f]=-2f$, and $[e,f]=h$. Write $L(16)$ and $L(8)$ for its irreducible modules of highest weights $16$ and $8$, respectively.
		We identify $\sl_2$ with $V_1$ via the standard isomorphism; on the top levels, $a\in\sl_2$ acts as $\vo{a}{0}$.
		Identify $W(0)=L(16)$ and $M(0)=L(8)$, and let $u_{16}$ and $u_8$ be their highest-weight vectors.
		By \cite[(3.2.3) and Theorem~3.2.2]{FZ92}, we have an isomorphism of $A(V)$-bimodules
		\begin{equation}\label{eq:AWiso}
			\begin{aligned}
				A(W) 
				&\quad{\cong}\quad
				\ab(L(16)\otimes U(\sl_2))/\ab\langle u_{16}\otimes e\rangle,\\
				[\vo{(a_m)}{-1}\cdots\vo{(a_1)}{-1}u]
				&\longleftrightarrow
				u\otimes a_1\cdots a_m+\ab\langle u_{16}\otimes e\rangle,
			\end{aligned}
		\end{equation}
		where $a_i\in\sl_2$, $u\in L(16)$, and $\ab\langle u_{16}\otimes e\rangle$ denotes the $U(\sl_2)$-subbimodule generated by $u_{16}\otimes e$. The right action is multiplication on $U(\sl_2)$, and the left action is given by
		\[
		X\cdot(u\otimes a)=Xu\otimes a+u\otimes Xa,\qquad X\in\sl_2.
		\]
		In particular, there is a surjective $\sl_2$-homomorphism
		\begin{equation}\label{eq:AWtensorid}
			\pi\colon L(16)\otimes L(8)\longrightarrow A(W)\otimes_{A(V)}M(0),
			\qquad u\otimes v\longmapsto[u]\otimes v,
		\end{equation}
		where the source carries the diagonal action. Surjectivity follows from the balancing relation $(u\otimes a)\otimes v=(u\otimes1)\otimes av$ in the tensor product of the quotient in \eqref{eq:AWiso} with $M(0)$. Thus $F:=\phi_I\circ\pi$ is a nonzero $\sl_2$-homomorphism.
		
		Weight considerations give $F(u_{16}\otimes f^su_8)=0$ for $s<8$, since the source vector has weight $24-2s>8$. Moreover, $u_{16}\otimes L(8)$ generates $L(16)\otimes L(8)$ under the diagonal action: this follows by induction on $r$ from
		\[
		f^{r+1}u_{16}\otimes u
		=f\cdot(f^ru_{16}\otimes u)-f^ru_{16}\otimes fu.
		\]
		Hence $F(u_{16}\otimes f^8u_8)\neq0$, for otherwise $F$ would vanish identically. This image has weight $8$, so rescaling $I$ gives
		\[
		F(u_{16}\otimes f^8u_8)=u_8.
		\]
		
		Now set $u=\vo{f}{-1}^8u_{16}$. Since each $\vo{f}{-1}$ raises the $L(0)$-weight by $1$, we have $u\in W(8)$. Under the isomorphism \eqref{eq:AWiso}, $[u]$ corresponds to $u_{16}\otimes f^8$. Hence, for $0\le s\le8$, we have 
		\[
		o_I(u)f^su_8 = \phi_I([u]\otimes f^su_8) 
		= F(u_{16}\otimes f^{8+s}u_8) = 
		\begin{cases}
			u_8,&s=0,\\
			0,&1\le s\le8,
		\end{cases}
		\]
		where the first case follows from the normalization above and the second follows from $f^9u_8=0$. Thus $\tr_{M(0)}\ab(o_I(u))=1$.
		
		Finally, \cref{prop:square-transport} implies that every $L[0]$-homogeneous component of $u$ has weight at most $12$.
		Let $\tilde{u}$ denote its weight-$12$ component. 
		The preceding vanishing of trace functions for weights below $12$ gives
		\[
		Z_I(\tilde{u},\tau) = Z_I(u,\tau).
		\]
		This is a cusp form of weight $12$ with expansion $q_\tau+O(q_\tau^2)$ by \eqref{eq:ZIexp} and the trace calculation above. 
		Since every cusp form of weight $12$ for $\SL(2,\Z)$ is a scalar multiple of $\Delta$, comparing the leading coefficients therefore gives $Z_I(\tilde{u},\tau)=\Delta(\tau)$.
	\end{proof}
	
	\subsubsection{Further questions on realizations}
	
	The construction of $\Delta$ combines a nonzero trace on $M(0)$ with an upper bound on the square-bracket weights of the realizing vector. In the example above, this bound isolates the weight-$12$ component because all lower-weight cusp forms vanish. Using a similar method, one may look for elements of higher $L[0]$-weights in $W$ in order to realize other cusp forms, such as $\Delta E_4, \Delta E_6$, and so on. Miyamoto also realizes the Dedekind $\eta$-function using the Ising-model Virasoro VOA \cite[\S~6]{Mi00}. That example has a nontrivial multiplier and $h_M-\frac{\ccharge}{24}=1/24$, so it belongs to the broader setting of scalar-valued forms with multipliers.
	It is therefore natural to ask the following question.
	
	\begin{question}
		Are all modular forms realizable as trace functions associated to intertwining operators for strongly rational VOAs? Explicitly, given any modular form $f_k(\tau)\in M_k(\SL_2(\Z))$, is it always possible to find a strongly rational VOA $V$, irreducible $V$-modules $W$ and $M$, and a nonzero intertwining operator
		$
		I\in \Ifusion[W][M][M],
		$
		such that
		\[
		\bigoplus_{i=0}^N \Ifusion[W][M^i][M^i]=\C I,
		\qquad
		h_M-\frac{\ccharge}{24}\in\N,
		\]
		together with an $L[0]$-homogeneous element $\tilde{u}\in W$ satisfying
		$L[0]\tilde{u}=k\tilde{u}$
		and
		$Z_I(\tilde{u},\tau)=f_k(\tau)$?
	\end{question}

	\subsection{Congruence properties of the modular representation}
	
	The modular transformation law for intertwining-operator traces gives the projective representation \eqref{eq:firstrep}. We now ask whether its kernel contains a principal congruence subgroup. In the rest, let $\Gamma(N)$ denote the subgroup of matrices in $\SL(2,\Z)$ congruent to the identity modulo $N$. 
	Containment of $\Gamma(N)$ in the kernel means that the projective action depends only on the matrix entries modulo $N$.
	
	The ordinary VOA trace functions provide the starting point. When $W=V$, the intertwining operators are the module vertex operators $Y_{M^i}$, and $Z_{Y_{M^i}}(u,\tau)=Z_i(u,\tau)$; thus \cref{coro:originalmodularinv} recovers Zhu's modular invariance theorem. If $V$ is additionally self-dual, Dong--Lin--Ng prove that the corresponding linear modular representation has congruence kernel of level $N$, where $N$ is the smallest positive integer such that
	\[
	N\ab(h_{M^i}-\frac{\ccharge}{24})\in\Z
	\qquad\text{for every irreducible module }M^i.
	\]
	Consequently, for each $u \in V[r]$, the trace $Z_i(u,\tau)$ is a weakly holomorphic modular form of weight $r$ on a congruence subgroup \cite[Theorem~I]{DLN15}. In particular, the projectivization of this representation also has congruence kernel. The categorical congruence property underlying this result is studied in \cite{NS10}; related arithmetic questions are also discussed in \cite{CDT25}.
	
	For traces from general intertwining-operators, linear representations obtained after removing a scalar multiplier can have noncongruence kernels. Such examples occur for Virasoro minimal-model VOAs and positive-integral-level $\sl_2$ affine VOAs \cite{KM18,KSW26}. Thus the congruence property of ordinary VOA traces does not extend automatically.
	
	In contrast, this paper focuses on projective representations. An element belongs to the projective kernel precisely when it acts by a scalar. The projective kernel therefore may be strictly larger than the linear kernel. This leads to the following question.
	
	\begin{question}
		Under what conditions on a strongly rational VOA $V$ and an
		irreducible $V$-module $W$ does the projective modular
		representation
		\[
		\rho\colon\SL(2,\Z)\longrightarrow
		\mathrm{PGL}(\Gamma^\nabla(\H,\mathscr{C}_\H(W)))
		\]
		associated with intertwining-operator trace functions have congruence
		kernel? Equivalently, for which pairs $(V,W)$ does there exist
		$N\ge1$ such that $\Gamma(N)\subseteq\ker\rho$?
	\end{question}

	\subsection*{Acknowledgment}
	
	We would like to thank Professors Angela Gibney and Yongchang Zhu for valuable discussions and useful comments. Our main results have been presented at various conferences and seminars, including SRI Algebraic Geometry 2025 at CSU, the Geometry Seminar at VCU, the Algebra and Number Theory Seminar at LSU, and the Conference on Algebra, Geometry and Mathematical Physics at NewUU. We would like to thank the organizers of these conferences and seminars for the invitation and for the opportunity to present our work.

\end{document}